\documentclass [reqno, a4paper]{amsart}
\usepackage[english]{babel}
\usepackage[latin1]{inputenc}
\usepackage{amscd}
\usepackage{hyperref}
\hypersetup{
	colorlinks=true,
	citecolor=blue,
	linkcolor=blue,
	filecolor=blue, 
	urlcolor=blue,
}
\usepackage{amssymb}
\usepackage{amsmath}
\usepackage{latexsym}
\usepackage{tikz}
\usepackage{tikz-cd}
\usepackage[nobysame]{amsrefs}
\usepackage{mathtools}
\usepackage{relsize}
\usepackage[normalem]{ulem}
\usepackage{mathrsfs}
\usepackage{amsthm}
\usepackage[toc,page]{appendix}
\usepackage{xcolor}
\usepackage[breakable]{tcolorbox}
\definecolor{ggreen}{rgb}{0.0, 0.5, 0.0}
\definecolor{yyellow}{rgb}{1.0, 0.77, 0.05}
\DeclareMathSymbol{\shortminus}{\mathbin}{AMSa}{"39}

\newcommand{\U}{\mathsf{U}}

\newcommand{\sfA}{\mathsf{A}}
\newcommand{\sfP}{\mathsf{P}}
\newcommand{\sfQ}{\mathsf{Q}}
\newcommand{\AAA}{\mathsf{AA}}

\newcommand{\LA}{\mathsf{LA}}
\newcommand{\LCA}{\mathsf{LCA}}
\newcommand{\LBA}{\mathsf{LBA}}
\newcommand{\DY}{\mathsf{DY}}
\newcommand{\g}{\mathfrak{g}}

\renewcommand{\b}{\mathfrak{b}}
\renewcommand{\d}{\mathfrak{d}}

\newcommand{\B}{\mathcal{B}}
\newcommand{\G}{\mathcal{G}}
\newcommand{\T}{\mathcal{T}}
\newcommand{\Hom}{\mathsf{Hom}}

\newcommand{\End}{\mathsf{End}}

\newcommand{\ten}{\otimes}
\newcommand{\id}{\mathsf{id}}
\newcommand{\Obj}{\mathsf{Obj}}

\newcommand{\C}{\mathscr{C}}

\renewcommand{\P}{\mathscr{P}}

\newcommand{\K}{\mathbb{K}}

\newcommand{\op}{\mathsf{op}}

\newcommand{\Vect}{\mathsf{Vect}}
\newcommand{\QUE}{\mathsf{QUE}}

\newcommand{\TopFree}{\mathsf{TopFree}}

\DeclareFontFamily{U}{wncy}{}
\DeclareFontShape{U}{wncy}{m}{n}{<->wncyr10}{}
\DeclareSymbolFont{mcy}{U}{wncy}{m}{n}
\DeclareMathSymbol{\RuK}{\mathord}{mcy}{"4B}
\DeclareMathSymbol{\Run}{\mathord}{mcy}{"6E}
\DeclareMathSymbol{\RuZ}{\mathord}{mcy}{"5A}
\DeclareMathSymbol{\Rua}{\mathord}{mcy}{"61}
\DeclareMathSymbol{\RuD}{\mathord}{mcy}{"44}
\DeclareSymbolFont{rsfs}{U}{rsfs}{m}{n}
\DeclareSymbolFontAlphabet{\mathscrsfs}{rsfs}

\newcommand{\mosaic}[2]{\mathfrak{M}_{#1,#2}}
\newcommand{\loom}[2]{\mathfrak{L}_{#1,#2}}
\newcommand{\grid}[2]{\G_{#1,#2}}
\newcommand{\Neg}{\mathsf{Neg}}
\newcommand{\cob}{\alpha}
\newcommand{\bra}{\beta}

\newcommand{\gri}[2]{\draw[step=1cm,black,very thick] (0,0) grid (#2,-#1);}

\newcommand{\ver}[2]{\draw[red, very thick] (#2-0.5,-#1+1)--(#2-0.5,-#1);}

\newcommand{\cro}[2]{\draw[red, very thick] (#2-0.5,-#1+1)--(#2-0.5,-#1); \draw[red, very thick] (#2-1,-#1+0.5)--(#2,-#1+0.5); }

\newcommand{\hor}[2]{\draw[red, very thick] (#2-1,-#1+0.5)--(#2,-#1+0.5);}

\newcommand{\tbra}[2]{\draw[red, very thick] (#2-0.5,-#1+1)--(#2-0.5,-#1); \draw[red, very thick] (#2-1,0.5-#1)--(#2-0.5,0.5-#1);}

\newcommand{\tcobra}[2]{\draw[red, very thick] (#2-1,-#1+0.5)--(#2,-#1+0.5); \draw[red, very thick] (#2-0.5,-#1+1)--(#2-0.5,-#1+0.5);}

\newcommand{\lcro}[2]{\draw[red, very thick] (#2-0.5,-#1+1)--(#2-0.5,-#1); \draw[red, very thick] (#2-1,-#1+0.5)--(#2,-#1+0.5); \draw[blue, very thick] (#2-0.3,-#1+1)--(#2-0.3,-#1); \draw[yyellow, very thick] (#2-1,-#1+0.3)--(#2,-#1+0.3); }

\newcommand{\lver}[2]{\draw[red, very thick] (#2-0.5,-#1+1)--(#2-0.5,-#1); \draw[blue, very thick] (#2-0.3,-#1+1)--(#2-0.3,-#1); }

\newcommand{\lhor}[2]{\draw[red, very thick] (#2-1,-#1+0.5)--(#2,-#1+0.5); \draw[yyellow, very thick] (#2-1,-#1+0.3)--(#2,-#1+0.3); }

\newcommand{\lmua}[2]{\draw[red, very thick] (#2-1,-#1+0.5) arc (-90:0:0.5); \draw[blue, very thick] (#2-0.15,-#1+1)--(#2-0.15, -#1);\draw[red, very thick] (#2-0.35,-#1+1)--(#2-0.35, -#1);}

\newcommand{\lmub}[2]{\draw[red, very thick] (#2-1+0.25,-#1+0.5) arc (-90:0:0.5); \draw[blue, very thick] (#2-0.5,-#1+1)--(#2-0.5, -#1); \draw[red,very thick] (#2-0.75,-#1+1)--(#2-0.75, -#1); \draw[red,very thick] (#2-1+0.25,-#1+0.5)--(#2-1,-#1+0.5);}

\newcommand{\dela}[2]{\draw[red, very thick] (#2-1,-#1+0.5) arc (-90:0:0.5); \draw[yyellow, very thick] (#2-1, -#1+0.15)--(#2, -#1+0.15); \draw[red, very thick] (#2-1, -#1+0.35)--(#2, -#1+0.35);}

\newcommand{\delb}[2]{\draw[red, very thick] (#2-1,-#1+0.25) arc (-90:0:0.5); \draw[red, very thick] (#2-1, -#1+0.75)--(#2, -#1+0.75); \draw[yyellow, very thick] (#2-1, -#1+0.5)--(#2, -#1+0.5); \draw[red, very thick] (#2-0.5,-#1+1)--(#2-0.5,-#1+0.75);}

\newcommand{\bpdver}[2]{\draw[violet, very thick] (#2-0.5,-#1+1)--(#2-0.5,-#1);}

\newcommand{\bpdcro}[2]{\draw[violet, very thick] (#2-0.5,-#1+1)--(#2-0.5,-#1); \draw[violet, very thick] (#2-1,-#1+0.5)--(#2,-#1+0.5); }

\newcommand{\bpdhor}[2]{\draw[violet, very thick] (#2-1,-#1+0.5)--(#2,-#1+0.5);}

\newcommand{\bpdone}[2]{\draw[violet, very thick] (#2-0.5,-#1+1)--(#2-0.5,-#1+0.5); \draw[violet, very thick] (#2-1,0.5-#1)--(#2-0.5,0.5-#1);}

\newcommand{\bpdtwo}[2]{\draw[violet, very thick] (#2-0.5,-#1+0.5)--(#2,-#1+0.5); \draw[violet, very thick] (#2-0.5,-#1)--(#2-0.5,-#1+0.5);}

\newcommand{\FM}[2]{F_{#1,#2}}
\newcommand{\stirling}[2]{{#1\brace #2}}
\newcommand{\countingone}{\zeta} 
\newcommand{\countingtwo}{\xi} 
\numberwithin{equation}{section}
\theoremstyle{plain}  
\newtheorem{theorem}{Theorem}[section] 
\newtheorem{proposition}[theorem]{Proposition} 
\newtheorem{corollary}[theorem]{Corollary} 
\newtheorem{lemma}[theorem]{Lemma}               
\newtheorem{definition}[theorem]{Definition}
\newtheorem{example}[theorem]{Example}
\newtheorem{remark}[theorem]{Remark}   
  
\newtheorem{notation}[theorem]{Notation}  
\newtheorem{conjecture}[theorem]{Conjecture}  

\title {On the universal Drinfeld--Yetter algebra}

\author{Andrea Rivezzi}
\address{Dipartimento di Matematica e Applicazioni, Università degli studi di Milano-Bicocca, Via R.Cozzi 55, 20125 Milano, Italy}
\email{\href{mailto:a.rivezzi@campus.unimib.it}{a.rivezzi@campus.unimib.it}}

\begin{document}

\begin{abstract}
	The universal Drinfeld--Yetter algebra is an associative algebra whose 
	co--Hochschild cohomology controls the existence of quantization functors of Lie bialgebras,
	such as the renowned one due to Etingof and Kazhdan. It was initially introduced by Enriquez and later re-interpreted by Appel and Toledano Laredo as an algebra of endomorphisms in the colored PROP of a Drinfeld--Yetter module over a Lie bialgebra. In this paper, we provide an explicit formula for its structure constants in terms of certain diagrams, which we term Drinfeld--Yetter looms.
\end{abstract}	 


\maketitle\thispagestyle{empty}
\setcounter{tocdepth}{1}
\tableofcontents

\section{Introduction}

\subsection{}
In the renowned series of articles \cite{etingof1996quantization, etingof1998quantization,etingof2000quantization,etingof2000quantizationb,etingof2008quantization}, Etingof and Kazhdan constructed a quantization functor of Lie bialgebras, and provided a positive answer to several problems posed by Drinfeld in \cite{drinfeld1992some}.
In \cite{enriquez2001quantization,enriquezonsome,enriquez2005cohomological}, Enriquez proposed an alternative construction of the Etingof--Kazhdan quantization functor, based on cohomological techniques. Specifically, he introduced a family of {\em universal} algebras
$\{\mathcal{U}^n\}_{n\geqslant 0}$, and relied on their co--Hochschild cohomology to prove the existence of a quantization functor. Moreover, by a rigidity argument, he proved that the Etingof--Kazhdan functor could also be obtained in this way. The algebras $\{\mathcal{U}^n\}_{n\geqslant 0}$ are universal in that they are equipped with canonical morphisms of algebras $\mathcal{U}^n\to (\U(\g))^{\ten n}$ for {\em any} finite--dimensional quasitriangular Lie bialgebra $\g$. In particular, the co--Hochschild complex of $\{\mathcal{U}^n\}_{n\geqslant 0}$ shall be thought of as a {\em universal} analogue of the co-Hochschild complex of $\g$. 

The definition of $\mathcal{U}^n$ is based on an explicit formula of their product given in \cite[App. C]{enriquez2001quantization}. Unfortunately, such formula appears to be incorrect 
in that it does not satisfy the above universal property. In Example \ref{counterexample} we show it by direct inspection for $n=1$ and $\g=\mathfrak{sl}_2$.

In \cite[Sec.~5]{appeluniqueness} Appel and Toledano Laredo provided a new realization of Enriquez's algebras, which readily satisfy the required universal property. Roughly, they 
introduced a colored PROP $\DY^n$ generated by a Lie bialgebra object $[\b]$ and 
$n$ Drinfeld--Yetter $[\b]$--modules $[V_1], \dots, [V_n]$. Then, 
$\mathcal{U}^n$ \footnote{For a better clarity, in the main body of the paper we denote, as in the original articles, Enriquez's algebras by $\mathcal{U} \mathfrak{S}_{univ}^n$ and Appel--Toledano Laredo algebras by $\mathfrak{U}_\DY^n$} is defined to be the algebra of endomorphism of $[V_1]\ten\cdots\ten[V_n]$
in $\DY^n$. Their definition simplifies significantly Enriquez's argument, and becomes a fundamental ingredient in their proof of the monodromy theorem for the Casimir connection
in \cite{appel-toledano-24}.

\subsection{}
While, on one side, this construction clarified the conceptual origin of $\mathcal{U}^n$,
on the other it did not provide an explicit formula of its product, akin to the one originally given by Enriquez. For $n=1$, the morphisms in $\DY^1$ are generated by the bracket and the cobracket on $[\b]$, the action and the coaction on $[V_1]$. These are commonly represented by diagrams 
satisfying a long list of relations. For instance, action and coaction satisfy, respectively,
\begin{equation*}
	\begin{tikzpicture}
		\draw[-, very thick,ggreen] (-0.25,0)--(1.5,0);
		\draw (1,0) arc (0:90:0.75);
		\draw(0.25,0.75)--(0,0.75);
		\draw (0,0.75)--(-0.25,1);
		\draw (0,0.75)--(-0.25,0.5);
		\node at (2,0.5) {$=$};
		\draw[-, very thick,ggreen] (-0.25+2.75,0)--(1.5+2.75,0);
		\draw (1+2.75,0) arc (0:90:0.75);
		\draw(0.25+2.75,0.75)--(-0.25+2.75,0.75);
		\draw (1+2.55,0) arc (0:90:0.55);
		\draw(0.5+2.55,0.55)--(0+2.5,0.55);
		\node at (2+2.75,0.5) {$-$};
		\draw[-, very thick,ggreen] (-0.25+2.75+3,0)--(1.5+2.75+3,0);
		\draw (1+2.75+3+0.2,0) arc (0:90:0.75);
		\draw(0.45+2.55+3,0.55)--(2.55+3,0.75);
		\draw (2.75+3+0.25,0.75)--(2.75+3+0.45,0.75);
		\draw (2.75+3+0.25,0.55)--(2.75+3+0.45,0.55);
		\draw (1+2.55+3+0.2,0) arc (0:90:0.55);
		\draw(0.25+2.75+3,0.75)--(2.55+3,0.55);
	\end{tikzpicture}
\end{equation*}
\begin{equation*}
	\begin{tikzpicture}
		\draw[-, very thick,ggreen] (0,0)--(1.5,0);
		\draw (1,0.75) arc (90:180:0.75);
		\draw(1,0.75)--(1.25,0.75);
		\draw(1.25,0.75)--(1.5,1);
		\draw(1.25,0.75)--(1.5,0.5);
		\node at (2.3,0.5) {$=$};
		\draw[-, very thick,ggreen] (0+2.75,0)--(1.5+2.95,0);
		\draw (1+2.75,0.75) arc (90:180:0.75);
		\draw(1+2.75,0.75)--(1.25+2.75,0.75)--(1.5+2.95,0.55);
		\draw (1+2.75,0.55) arc (90:180:0.55);
		\draw(1+2.75,0.55)--(1.5+2.5,0.55)--(1.5+2.95,0.75);
		\node at (2+3.2,0.5) {$-$};
		\draw[-, very thick,ggreen] (0+2.75+3,0)--(1.5+2.75+3,0);
		\draw (1+2.75+3,0.75) arc (90:180:0.75);
		\draw(1+2.75+3,0.75)--(1.25+2.75+3,0.75)--(1.5+2.75+3,0.75);
		\draw (1+2.75+3,0.55) arc (90:180:0.55);
		\draw(1+2.75+3,0.55)--(1.5+2.5+3,0.55)--(1.5+2.75+3,0.55);
	\end{tikzpicture}
\end{equation*}
and they obey the following commutation rule
\begin{equation*}
	\begin{tikzpicture}
		\draw[-, very thick,ggreen] (0+0.5,0)--(2.5+0.5,0);
		\draw (1+0.5,0) arc (0:90:0.75);
		\draw(0.25+0.5,0.75)--(0+0.5,0.75);
		\draw (2.25+0.5,0.75) arc (90:180:0.75);
		\draw(2.25+0.5,0.75)--(2.5+0.5,0.75);
		\node at (3.5+0.5,0.3) {$=$};
		\draw[-, very thick,ggreen] (5,0)--(6.25,0);
		\draw (1+5,0) arc (0:90:0.75);
		\draw(0.25+5,0.75)--(0+5,0.75);
		\draw (1+5,0.75) arc (90:180:0.75);
		\draw(1+5,0.75)--(1.25+5,0.75);
		\node at (1.8+5,0.4) {$+$}; 
		\draw[-, very thick,ggreen] (2.25+5,0)--(3.5+5,0);
		\draw (3.25+5,0.75) arc (90:180:0.75); 
		\draw(3.25+5,0.75)--(3.5+5,0.75);
		\draw (2.8+5,0.6) arc (-300:-260:1); 
		\node at (4+5,0.4) {$-$}; 
		\draw[-, very thick,ggreen] (4.5+5,0)--(6+5,0); 
		\draw (5.5+5,0) arc (0:90:0.75);
		\draw(4.75+5,0.75)--(4.5+5,0.75);
		\draw (5.75+5,0.75) arc (90:131:0.75); 
		\draw (5.75+5,0.75)--(6+5,0.75);
	\end{tikzpicture}
\end{equation*}
By convention, we read the diagrams from left to right.

\subsection{}
The elements of $\mathcal{U}=\mathcal{U}^1$ are linear combination of diagrams of the form 
\vspace{0.3 cm}
\begin{equation*}
\label{eq:generic-element}
\begin{tikzpicture} 
\draw[-, ggreen,very thick] (0,0)--(3.8,0);
\draw[very thick] (1,0.75) arc (90:180:0.75);
\draw[very thick] (1,0.75)--(1.2,0.75);
\draw[black] (1.2,0.3) rectangle (2.6,1);
\node at (1.9, 0.65) {$\phi$};
\draw[very thick] (1.6+1,0.75)--(1.8+1,0.75);
\draw[very thick] (2.55+1,0) arc (0:90:0.75);
\node at (2.55+1,-0.3) {$m$};
\node at (0.3,-0.3) {$n$};
\end{tikzpicture}
\end{equation*}
The diagram above is to be understood as an endomorphism of
$[V_1]$, given by $n$ iterated coactions, a universal morphism $\phi\colon [\b]^{\ten n}\to[\b]^{\ten m}$, and finally $m$ iterated actions. 
Relying on the diagrammatic relation above, it is natural to impose on $\End_{\DY^1}([V_1])$ a \emph{normal ordering}, featuring no bracket nor cobracket, and where coactions preceed actions.
By relying on such normal ordering, Appel and Toledano Laredo proved in \cite[Prop.~5.10]{appeluniqueness} that, as a vector space,
\[ \mathcal{U} \simeq \bigoplus_{n \geqslant 0} \K[\mathfrak{S}_n]\]
where $\mathfrak{S}_n$ is the symmetric group (cf.~\cite[App. C - Rmk. 12]{enriquez2001quantization}). Under this isomorphism, a permutation $\sigma\in\mathfrak{S}_n$
corresponds to the endomorphism
\vspace{0.3 cm}
\begin{equation*}
	\begin{tikzpicture} 
		\draw[-, very thick,ggreen] (0,0)--(2.8,0);
		\draw[very thick] (1,0.75) arc (90:180:0.75);
		\draw[very thick]  (1,0.75)--(1.2,0.75);
		\node[shape=circle,draw,inner sep=1pt] (char) at (1.4,0.75) {$\sigma$}; 
		\draw[very thick]  (1.6,0.75)--(1.8,0.75);
		\draw[very thick]  (2.55,0) arc (0:90:0.75);
		\node at (2.55,-0.3) {$n$};
		\node at (0.3,-0.3) {$n$};
	\end{tikzpicture}
\end{equation*}
By applying the same procedure to the product of two permutations
$\sigma\in\mathfrak{S}_n$ and $\tau\in\mathfrak{S}_m$, one readily sees that 
\begin{equation}
\label{eq:product-constants}
\begin{tikzpicture}[scale=1] \draw[-, very thick,ggreen] (0,0)--(5.8,0);
\draw[very thick]  (1,0.75) arc (90:180:0.75);
\draw[very thick] (1,0.75)--(1.2,0.75);
\node[shape=circle,draw,inner sep=1pt] (char) at (1.4,0.75) {$\sigma$}; 
\draw[very thick] (1.6,0.75)--(1.8,0.75);
\draw[very thick]  (2.55,0) arc (0:90:0.75);
\node at (2.55,-0.3) {$n$};
\node at (0.3,-0.3) {$n$};
\draw[very thick]  (4,0.75) arc (90:180:0.75);
\draw[very thick] (4,0.75)--(4.2,0.75);
\node[shape=circle,draw,inner sep=1pt] (char) at (4.4,0.75) {$\tau$}; 
\draw[very thick] (4.6,0.75)--(4.8,0.75);
\draw[very thick]  (5.55,0) arc (0:90:0.75);
\node at (5.55,-0.3) {$m$};
\node at (3.3,-0.3) {$m$};
\node at (7,0.25) {$=$ $\mathlarger{\sum_{\pi \in \mathfrak{S}_{n+m}}} c_{\sigma, \tau}^\pi$};
\begin{scope}[shift={(8.25,0)}]
\draw[-, very thick,ggreen] (0,0)--(2.8,0);
\draw[very thick]  (1,0.75) arc (90:180:0.75);
\draw[very thick] (1,0.75)--(1.2,0.75);
\node[shape=circle,draw,inner sep=1pt] (char) at (1.4,0.75) {$\pi$}; 
\draw[very thick] (1.6,0.75)--(1.8,0.75);
\draw[very thick]  (2.55,0) arc (0:90:0.75);
\node at (2.55,-0.3) {$n+m$};
\node at (0.3,-0.3) {$n+m$};
\end{scope}
\end{tikzpicture}
\end{equation}
with integral structure constants $c_{\sigma,\tau}^\pi\in\mathbb{Z}$. 
The computation of the integers $c_{\sigma,\tau}^\pi$ is computationally quite challenging. By applying a normal ordering algorithm, the number of terms necessary to express 
$c_{\sigma,\tau}^\pi$ appears to be growing exponentially (see also Appendix \ref{Appendix}).

\subsection{}
In this paper, we provide an explicit expression of the structure constants $c_{\sigma,\tau}^\pi\in\mathbb{Z}$ in terms of new combinatorial objects, which we term
\emph{Drinfeld--Yetter looms}. These are certain colored string diagrams, which are constructed as follows. We first considers the set of admissible tiles
\begin{center}
\begin{tikzpicture}[scale=0.6]
\node at (0,0) {$\mathcal{T}_{\loom{n}{m}}= \Big\{$};
\draw[black,very thick] (1.5,0.5) rectangle (2.5,-0.5);
\lcro{0.5}{2.5}
\node at (2.75,-0.5) {,};
\draw[black,very thick] (3,0.5) rectangle (4,-0.5);
\lmua{0.5}{4}
\node at (4.25,-0.5) {,};
\draw[black,very thick] (4.5,0.5) rectangle (5.5,-0.5);
\lmub{0.5}{5.5}
\node at (5.75,-0.5) {,};
\draw[black,very thick] (6,0.5) rectangle (7,-0.5);
\dela{0.5}{7}
\node at (7.25,-0.5) {,};
\draw[black,very thick] (7.5,0.5) rectangle (8.5,-0.5);
\delb{0.5}{8.5}
\node at (8.75,-0.5) {,};
\draw[black,very thick] (9,0.5) rectangle (10,-0.5);
\lhor{0.5}{10}
\node at (10.25,-0.5) {,};
\draw[black,very thick] (10.5,0.5) rectangle (11.5,-0.5);
\lver{0.5}{11.5}
\node at (11.75,-0.5) {,};
\draw[black,very thick] (12,0.5) rectangle (13,-0.5);
\node at (13.5,-0) {\Big\}.};
\end{tikzpicture}
\end{center}
Then, we define $\loom{n}{m}$ as the set of all the possible fillings of the empty grid
with $n$ rows and $m$ columns with the elements of $\mathcal{T}_{\loom{n}{m}}$
such that the resulting picture is a continuous string diagram.
Any $L \in \loom{n}{m}$, $\sigma \in \mathfrak{S}_n$, $\tau \in \mathfrak{S}_m$ can naturally be  associated to a permutation $\tilde{\gamma}(\sigma, L, \tau) \in \mathfrak{S}_{n+m}$, {\em e.g.},
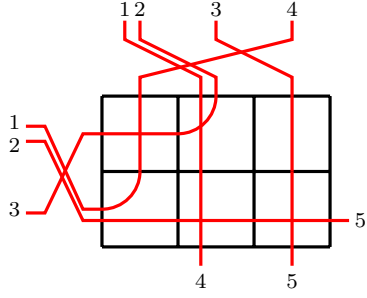
\begin{center}
	\begin{tikzpicture}
		\gri{2}{3} \ver{1}{3} \cro{1}{1} \ver{2}{3}
		\draw[red, very thick] (2-1,-1+0.5) arc (-90:0:0.5); 
		\draw[red, very thick] (2-0.7,-1+1)--(2-0.7, -2);
		\draw[red, very thick] (1-1,-2+0.5) arc (-90:0:0.5);
		\draw[red, very thick] (1-1, -2+0.35)--(3, -2+0.35);
		\draw[red, very thick] (2.5,-2)--(2.5,-2.25);
		\draw[red, very thick] (1.3,-2)--(1.3,-2.25);
		\draw[red, very thick] (1.3,0)--(1.3,0.25)--(0.3,0.75)--(0.3,1);
		\draw[red, very thick] (1.5,0)--(1.5,0.25)--(0.5,0.75)--(0.5,1);
		\draw[red, very thick] (2.5,0)--(2.5,0.25)--(1.5,0.75)--(1.5,1);
		\draw[red, very thick] (0.5,0)--(0.5,0.25)--(2.5,0.75)--(2.5,1);
		\draw[red, very thick] (3,-1.65)--(3.25,-1.65);
		\draw[red, very thick] (0,-0.5)--(-0.25,-0.5)--(-0.75,-1.55)--(-1,-1.55);
		\draw[red, very thick] (0,-1.5)--(-0.25,-1.5)--(-0.75,-0.4)--(-1,-0.4);
		\draw[red, very thick] (0,-1.65)--(-0.25,-1.65)--(-0.75,-0.6)--(-1,-0.6);
		\node at (-1.15,-0.35) {\footnotesize{1}};
		\node at (-1.15,-0.65) {\footnotesize{2}};
		\node at (-1.15,-1.5) {\footnotesize{3}};
		\node at (1.3,-2.45) {\footnotesize{4}};
		\node at (2.5,-2.45) {\footnotesize{5}};
		\node at (0.3,1.15)  {\footnotesize{1}};
		\node at (0.5,1.15)  {\footnotesize{2}};
		\node at (1.5,1.15)  {\footnotesize{3}};
		\node at (2.5,1.15)  {\footnotesize{4}};
		\node at (3.4,-1.65)  {\footnotesize{5}};
	\end{tikzpicture}
\end{center}
corresponds to the permutation $(14)(253) \in \mathfrak{S}_5$. Relying on the number of occurrencies of certain specific tiles, we introduce the notion of a \emph{positive} (resp. \emph{negative}) loom.

The main result of this paper (see Theorem \ref{maintheorem} and Corollary \ref{maincor})
is an explicit expression of the structure constants of $\mathcal{U}$ in terms of cardinalities of
Drinfeld--Yetter looms. Specifically, 
\[
c_{\sigma,\tau}^\pi=P_{n,m}^{\sigma,\tau,\pi} - N_{n,m}^{\sigma,\tau,\pi}
\]
where $P_{n,m}^{\sigma,\tau,\pi}$ (resp. $N_{n,m}^{\sigma,\tau,\pi}$) denotes the number of
 positive (resp. negative) looms corresponding to $\pi\in\mathfrak{S}_{n+m}$.

\subsection{}
The proof relies on an intermediate combinatorial object, which we refer to as a \emph{Drinfeld--Yetter mosaic}. Similarly to the case of a loom, a mosaic is defined as an admissible filling of an $n\times m$ grid with the tiles  
\begin{equation*}
\begin{tikzpicture}[scale=0.6]
\node at (0.2,0) {$\mathcal{T}_\mathfrak{M}= \Big\{$};
\draw[black,very thick] (1.5,0.5) rectangle (2.5,-0.5);
\draw[red,very thick] (2,0.5)--(2,-0.5);
\draw[red,very thick] (1.5,0)--(2.5,0);
\node at (2.75,-0.5) {,};
\draw[black,very thick] (3,0.5) rectangle (4,-0.5);
\draw[red,very thick] (3,0)--(3.5,0);
\draw[red,very thick] (3.5,0.5)--(3.5,-0.5);
\node at (4.25,-0.5) {,};
\draw[black,very thick] (4.5,0.5) rectangle (5.5,-0.5);
\draw[red,very thick] (5,0.5)--(5,0);
\draw[red,very thick] (4.5,0)--(5.5,0);
\node at (5.75,-0.5) {,};
\draw[black,very thick] (6,0.5) rectangle (7,-0.5);
\draw[red,very thick] (6,0)--(7,0);
\node at (7.25,-0.5) {,};
\draw[black,very thick] (7.5,0.5) rectangle (8.5,-0.5);
\draw[red,very thick] (8,0.5)--(8,-0.5);
\node at (8.75,-0.5) {,};
\draw[black,very thick] (9,0.5) rectangle (10,-0.5);
\node at (10.5,-0) {\Big\}.};
\end{tikzpicture}
\end{equation*}
To each mosaic, we associate a morphism in $\DY^1$. This yields a combinatorial description 
the commutation rule between $n$ iterated actions and $m$ iterated coactions, that is
\begin{equation*}
\begin{tikzpicture} 
\draw[-, very thick,ggreen] (0,0)--(2.5,0);
\draw[very thick]  (1,0) arc (0:90:0.75);
\draw[very thick] (0.25,0.75)--(0,0.75);
\draw[very thick]  (2.25,0.75) arc (90:180:0.75);
\draw[very thick] (2.25,0.75)--(2.5,0.75);
\node at (1.65,-0.3) {$m$};
\node at (0.9,-0.3) {$n$}; 
\node at (5,0.2) {$= \mathlarger{\sum_{M \in \mosaic{n}{m}} (-1)^{\cob(M)} \varphi_{n,m}(M)}$};
\end{tikzpicture}
\end{equation*} 
Here, $\varphi_{n,m}(M)$ is the endomorphism associated to $M$, and $\cob(M)$ is the number of certain distinguished tiles appearing in $M$, see Proposition \ref{claim} for more details. 
Finally, we prove in Proposition \ref{proposition-partition-drinfeld-yetter-tableaux} that looms naturally appear as a refinement of mosaics, thus yielding the desired result.

\subsection{}
Our results lay the ground for a deeper study of the universal Drinfeld--Yetter algebra in relation with the combinatorics of the symmetric groups. As we show in the Appendix, the combinatorics of looms and mosaics, although quite complicated and challenging, appear to be controlled by a relatively small family of permutations, which for small $n$ and $m$ shows already interesting properties.

\subsection{Outline}
This paper is structured as follows. In \S \ref{section-one} we first recall the definitions of Lie bialgebras and Drinfeld--Yetter modules, and we sketch the main idea of the Etingof--Kazhdan quantization of Lie bialgebras. We then introduce the PROPic counterparts of such algebraic structures. \S \ref{section-universal-algebrass} is devoted to the Enriquez and Appel--Toledano Laredo universal algebras, with a particular focus on the algebra structure of $\mathcal{U}$. We present its vector space structure, and we exhibit an algorithm to compute the multiplication of two elements of the standard basis, following the ideas of Appel and Toledano Laredo \cite{appeluniqueness}. In \S \ref{section-two} we define the sets of Drinfeld--Yetter mosaics and of Drinfeld--Yetter looms and show their main properties. In \S \ref{section-explicit-formula} we then present the main result, that is a formula expressing the multiplication of $\mathcal{U}$ in terms of Drinfeld--Yetter looms. 
Finally, in the Appendix we collect some explicit computations and conjectures.
\subsection{Aknowledgments}
The author would like to thank A. Appel, M. Bordemann, and F. Ciliegi for numerous useful discussions. \\ 
This publication is based upon work from \emph{COST Action CaLISTA CA21109} supported by
COST (European Cooperation in Science and Technology). www.cost.eu. \\
The author is a member of the \emph{Gruppo Nazionale per le Strutture Algebriche, Geometriche e le loro Applicazioni} (GNSAGA) of the \emph{Istituto Nazionale di Alta Matematica} (INdAM).

 \subsection{Conventions and notations}
In this paper $\K$ always denotes a field of characteristics zero. If $A$ is a set, we denote by $|A|$ its cardinality and by $\P(A)$ its power set. 
\section{Lie bialgebras and PROPs}
\label{section-one}
\subsection{Lie bialgebras}
The following definition is standard, see e.g. \cite{majidfoundations}.
\begin{definition}
A Lie bialgebra is a triple $( \b, [ \cdot, \cdot], \delta )$, where $\b$ is a vector space and
\[ [\cdot , \cdot] : \b \ten \b \to \b \qquad \text{and} \qquad \delta : \b \to \b \ten \b \]
 are linear maps  (called respectively the Lie bracket and the Lie cobracket) satisfying the following conditions for any $x,y,z \in \b$:
\begin{itemize}
\item[(i)] (antisymmetry of the Lie bracket)
\begin{equation}
\label{eq:antis-bracket}
[x,y] = -[y,x]
\end{equation}
\item[(ii)] (Jacobi rule):
\begin{equation}
\label{eq:jacobi-rule}
[x,[y,z]] + [y,[z,x]] + [z,[x,y]] =0
\end{equation}
\item[(iii)] (antisymmetry of the Lie cobracket):
\begin{equation}
\label{eq:antis-cobracket}
\delta = - \tau \circ \delta
\end{equation}
\item[(vi)] (coJacobi rule):
\begin{equation}
\label{eq:cojacobi-rule}
(\sigma + \sigma^2 + \id_{\b,\b,\b})\circ (\id_\b \ten \delta) \circ \delta=0
\end{equation}
\item[(v)] (cocycle condition):
\begin{equation}
\label{eq:cocycle}
 \delta([x,y]) = [x \ten 1 + 1 \ten x , \delta(y)] - [y \ten 1 + 1 \ten y, \delta(x)]
\end{equation}
where $\tau: \b \ten \b \to \b \ten \b$ is the map permuting two variables and $\sigma : \b \ten \b \ten \b \to \b \ten \b \ten \b$ is the map ciclically permuting three variables.
\end{itemize}
\end{definition}
Conditions $(i)-(ii)$ mean that the pair $(\b, [\cdot, \cdot])$ is a Lie algebra, whence conditions $(iii)-(iv)$ mean that the pair $(\b, \delta)$ is a Lie coalgebra. Condition $(v)$ represents the compatibility between the Lie bracket and the Lie cobracket. A morphism of Lie bialgebras is a linear map that is both a morphism of Lie algebras and Lie coalgebras. The class of all Lie bialgebras together with the class of all morphisms of Lie bialgebras form a symmetric braided monoidal category, where the constraints are the same of the category of vector spaces.

\subsection{Drinfeld--Yetter modules}
The appropriate notion of a module of a Lie bialgebra is the following:
\begin{definition}
Let $(\b,[\cdot , \cdot],\delta)$ be a Lie bialgebra. A Drinfeld--Yetter module over $\b$ is a triple $(V, \pi, \pi^*)$, where $V$ is a vector space and 
\[ \pi : \b \ten V \to V \qquad \text{and} \qquad \pi^* : V \to \b \ten V \]
are linear maps (called respectively the action and the coaction) satisfying the following conditions:
\begin{itemize}
\item[(i)] (Lie action):
\begin{equation}
\label{eq:lie-action}
\pi \circ ([\cdot, \cdot] \ten \id_V) = \pi \circ (\id_\b \ten \pi) - \pi \circ (\id_\b \ten \pi) \circ (\tau \ten \id_V)
\end{equation}
\item[(ii)] (Lie coaction):
\begin{equation}
\label{eq:lie-coaction}
(\delta \ten \id_V) \circ \pi^* = (\tau \ten \id_V) \circ (\id_\b \ten \pi^*) \circ \pi^* -  (\id_\b \ten \pi^*) \circ \pi^* 
\end{equation}
\item[(iii)] (Drinfeld--Yetter relation):
\begin{equation}
\label{drinfeld-yetter-relation}
\pi^* \circ \pi = (\id_\b \ten \pi)\circ (\tau \ten \id) \circ (\id_\b \ten \pi^*) + ([\cdot, \cdot] \ten \id_V)\circ (\id_\b \ten \pi^*) - (\id_\b \ten \pi) \circ (\delta \ten \id_V).
\end{equation}
\end{itemize}
\end{definition}
Condition $(i)$ means that the pair $(V, \pi)$ is a (left) Lie $\b$--module, whence condition $(ii)$ means that the pair $(V, \pi^*)$ is a (right) Lie $\b$--comodule. Finally, condition $(iii)$ is the compatibility relation between the Lie action and the Lie coaction. A morphism of Drinfeld--Yetter $\b$--modules is a linear map which is both a morphism of Lie modules and of Lie comodules. The class of all Drinfeld--Yetter $\b$--modules together with the class of all morphisms of Drinfeld--Yetter $\b$--modules form a symmetric braided monoidal category, where the constraints are the same of the category of vector spaces. We shall denote it by $\DY_\b$. In particular, it can be shown that $\DY_\b$ is monoidally equivalent to the category of all equicontinuous $\d_\b$--modules, where $\d_\b$ is the Drinfeld double of $\b$, see \cite[7.3]{etingof1996quantization} and \cite[2.2]{appeluniqueness} for more details. 
\subsection{Quantization of Lie bialgebras}
In this Section we briefly recall the Etingof--Kazhdan quantization of Lie bialgebras. The reader can find more details in \cite{etingof1996quantization} \cite{etingof1998quantization} \cite{etingof2002lecture} \cite{appel20182}.
\begin{definition}
Let $(\b, [\cdot, \cdot], \delta)$ be a Lie bialgebra and $\U(\b)$ be its universal enveloping algebra. A topological Hopf algebra $H$ is said to be a quantization of $\b$ if there exists an isomorphism of Hopf algebras $ H / \hbar \cdot H \cong \U(\b)$
such that, for any $x \in \b$ and any lifting $\tilde{x} \in H$ of $x$ one has 
\[ \delta(x) = \frac{\Delta(\tilde{x}) - \Delta^\op(\tilde{x})}{\hbar} \ \mod \hbar\]
where $\Delta$ (resp. $\Delta^\op$) is the comultiplication (resp. coopposite comultiplication) of $H$.
\end{definition}
In the article \cite{drinfeld1992some}, Drinfeld announced a number of open problems in quantum group theory. Among them, there was the problem of finding a universal quantization functor assigning to any Lie bialgebra a topological Hopf algebra quantizing it.
This (and many others) problem of Drinfeld was solved by Etingof and Kazhdan in the articles \cite{etingof1996quantization} \cite{etingof1998quantization}, and further developed in \cite{etingof2000quantization} \cite{etingof2000quantizationb} \cite{etingof2008quantization}. We now sketch --very briefly-- the idea of the Etingof--Kazhdan quantization construction. Fix a formal parameter $\hbar$ and a Drinfeld associator $\Phi$ (see \cite{drinfeld1990quasitriangular} \cite{drinfeld1989quasi} \cite{bordemann2023gentle} for more details on Drinfeld associators). Then we may consider the \emph{deformed} category $\DY_\b^\Phi$, where the objects are topologically free Drinfeld--Yetter $\b$--modules, and the associativity and commutativity constraints are respectively 
\[ \Phi(\hbar t_{U,V}, \hbar t_{V,W}) \qquad \text{and} \qquad \tau_{V,W} \circ e^{\hbar t_{V,W}/2},\]
where $t$ is the standard infinitesimal braiding of $\DY_\b$, i.e. 
\[t_{V,W} = (\id_V \ten \pi_W) \circ (\tau_{\b,V} \ten \id_W )\circ (\pi^*_W \ten \id_V)+(\pi_V \ten \id_W) \circ (\tau_{V, \b}  \ten \id_W ) \circ (\id_V \ten \pi^*_W).\]
Consider the forgetful functor $F: \DY_\b^\Phi \to \TopFree_\K$ assigning to any topological Drinfeld--Yetter $\b$--module its underlying topologically free $\K[[\hbar]]$--module and to any morphism of Drinfeld--Yetter $\b$--modules its underlying morphism of topologically free modules. Etingof and Kazhdan constructed a non--trivial monoidal structure $J^{\mathrm{EK}}$ on $F$, inducing by Tannaka--Krein reconstruction a topological Hopf algebra structure on $\End(F)$, see \cite[\S 18]{etingof2002lecture} for more details. Twisting $\End(F)$ by $J^{\mathrm{EK}}$ produces a topological Hopf algebra quantizing the Drinfeld double $\d_\b$. In order to construct a quantization of $\b$, they finally consider the universal Verma module $M_- \coloneqq \mathrm{Ind}_{\b^*}^{\d_\b} \mathbb{C} \cong \U(\b)$, and show that $F(M_-)$ is a topological Hopf algebra quantizing $\b$, where the comultiplication is given by the composition $(J^{\mathrm{EK}}_{M_-,M_-})^{-1}\circ F(\Delta_0)$, where $\Delta_0$ is the comultiplication induced by standard one of $\U(\b)$. The construction is shown to be universal, see the next Sections for more details.
\subsection{PROPs}
We now introduce the concept of PROP, which goes back to \cite{maclanecategorical} and \cite{lawvere}. PROPs are used in order to encode the data of algebraic structures by modeling them in a $\K$--linear strict symmetric monoidal category. All the PROPs of our interests are quotients of the free PROP (see \cite{vallettekoszul} for more details) subject to generators and relations, depending on the algebraic structure we are considering. Richer algebraic structures, such as Drinfeld--Yetter modules over a Lie bialgebra, can be encoded in a extended notion of PROP, namely colored PROPs. Further details on PROPs can be find in \cite[\S 1.1--1.2]{etingof1998quantization}, \cite[Ch. 20]{etingof2002lecture}, \cite[\S 2]{enriquez2005invertibility}, \cite[\S 6]{appel20182}.
\begin{definition}
Let $\mathbb{K}$ be a field. 
\begin{itemize}
\item A PROP (product and permutation category) is a $\mathbb{K}$--linear, strict, symmetric monoidal category whose objects are indexed by non--negative integers and whose tensor product is given by $[n] \otimes [m] = [n+m]$. In particular, the tensor unit is $[0]$ and $[1]^{\otimes n} = [n]$. \\A morphism of PROPs\index{morphism of PROPs} is a strongly monoidal functor $F : \sfP \to \sfQ$ which is the identity on the objects and whose monoidal structure is the trivial one, i.e.
\[ F([n]_\sfP) \ten F([m]_\sfP) = [n]_\sfQ \ten [m]_\sfQ = [n+m]_\sfQ =F([n+m]_\sfP).\]
\item  A colored PROP is a $\mathbb{K}$--linear, strict, symmetric monoidal category whose objects are finite sequences over a set $\sfA$.
In other words, in a colored PROP $\sfP$ one has 
\[ \Obj(\sfP) = \coprod_{n \geqslant  0} \sfA^n.\]
 Here the tensor product is the concatenation of sequences, and the tensor unit is the empty sequence. The set $\sfA$ is said to be the set of colors of $\sfP$. Note that any PROP can be considered a colored PROP with respect to a set of colors of cardinality one.
\end{itemize}
\end{definition}
If $\sfP$ is a PROP, we have that for any $n \geqslant 0 $ there is an action 
\[ \mathbb{K}[\mathfrak{S}_n] \to \Hom_\sfP([n],[n]).\]
We shall call the corresponding morphisms \emph{permutation morphisms} and we shall denote them with the related permutation. In all examples of our interest, we have that such action is faithful.
We now present the main PROPs of our interest: the PROP of Lie bialgebras and the colored PROP of Drinfeld--Yetter modules.
\begin{example}
The PROP $\LBA$ of Lie bialgebras is the PROP generated by morphisms $\delta : [1] \to [2]$ and $\mu : [2] \to [1]$ (the universal Lie cobracket and the universal Lie bracket) subject to the following relations
\begin{subequations}
\begin{align}
\mu \circ (\id_{[2]} + (12))&=0 \label{eq:prop-lba-one} \\
\mu \circ (\mu \otimes \id_{[1]}) \circ (\id_{[3]} + (123) + (312))&=0 \label{eq:prop-lba-two}  \\
(\id_{[2]} + (12)) \circ \delta &= 0 \label{eq:prop-lba-three} \\
(\id_{[3]} + (123) + (312)) \circ (\delta \otimes \id_{[1]}) \circ \delta &=0 \label{eq:prop-lba-four} \\
\delta \circ \mu -  (\id_{[2]} - (12)) \circ (\id_{[1]} \otimes \mu) \circ (\delta \otimes \id_{[1]}) \circ (\id_{[2]}- (12))&=0 \label{eq:prop-lba-five} 
\end{align}
\end{subequations}
that respectively are the PROPic counterpart of the algebraic relations \eqref{eq:antis-bracket}, \eqref{eq:jacobi-rule}, \eqref{eq:antis-cobracket}, \eqref{eq:cojacobi-rule}, \eqref{eq:cocycle}.\\ 
Similarly, the PROP $\LA$ of Lie algebras is the PROP generated by $\mu : [1] \to [2]$ with relations \eqref{eq:prop-lba-one} and \eqref{eq:prop-lba-two}, while the PROP $\LCA$ of Lie coalgebras is the PROP generated by $\delta : [2] \to [1]$ with relations \eqref{eq:prop-lba-three} and \eqref{eq:prop-lba-four}.
\end{example}

\begin{example}
Let $n \geqslant  1$. The $n$--th Drinfeld--Yetter PROP $\DY^{n}$ is the colored PROP generated by $n+1$ objects $[1]$ and $\{ [V_k]\}_{k=1,\ldots,n}$ and by $2n + 2$ morphisms
\begin{equation*}
\begin{split}
\mu&:[2] \to [1] \\
\delta&: [1] \to [2] \\
\pi_k&: [1] \ten [V_k] \to [V_k] \\
\pi^*_k &: [V_k] \to [1] \ten [V_k]
\end{split}
\end{equation*}
such that the triple $([1], \mu, \delta)$ satisfies relations \eqref{eq:prop-lba-one}--\eqref{eq:prop-lba-five} and for any $k=1,\ldots, n$ the triple $([V_k], \pi_k, \pi_k^*)$ is a Drinfeld--Yetter module over $[1]$, i.e. the following relations are satisfied
\begin{subequations}
\begin{align}
\pi_k \circ (\mu \ten \id_{[V_k]}) &= \pi_k \circ (\id_{[1]} \ten \pi_k) - \pi_k \circ (\id_{[1]} \ten \pi_k) \circ (21) \label{eq:dy-module-one} \\
(\delta \ten \id_{[V_k]}) \circ \pi^*_k &= (21) \circ (\id_{[1]} \ten \pi^*_k) \circ \pi^*_k - (\id_{[1]} \ten \pi^*_k) \circ \pi^*_k \label{eq:dy-module-two} \\
\pi^*_k \circ \pi_k &=\big( (\id_{[1]} \ten \pi_k) \circ (12) +(\mu \ten \id_{[V_k]})  \big) \circ (\id_{[1]} \ten \pi^*_k)  - (\id_{[1]} \ten \pi_k) \circ (\delta \ten \id_{[V_k]}). \label{eq:dy-module-three}
\end{align}
\end{subequations}
\end{example}
PROPs are very commonly described using a pictorial representation.
In this paper, we shall represent the universal Lie bracket and the universal Lie cobracket of $\LBA$ and $\DY^n$ respectively with the diagrams
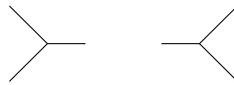
\begin{equation*}
\begin{tikzpicture}
\draw (1,0)--(1.5,-0.5);
\draw (1,-1)--(1.5,-0.5)--(2,-0.5); 
\draw (3,-0.5)--(3.5,-0.5)--(4,0);
\draw (3.5,-0.5)--(4,-1);
\end{tikzpicture}
\end{equation*}
which are read from left to right. According to this pictorial representation, we represent the PROPic Lie algebra axioms \eqref{eq:prop-lba-one}, \eqref{eq:prop-lba-two} respectively with

\vspace{0.3 cm}
\begin{equation}
\label{eq:lie}
\begin{tikzpicture}[scale=0.8]
\draw (-0.5,-0.5)--(0,0);
\draw (-0.5,0.5)--(0,0);
\draw (0,0)--(1,0);
\node at (1.5,0) {$=-$};
\draw (2,-0.5)--(3,0.5)--(4,0);
\draw(2,0.5)--(3,-0.5)--(4,0)--(4.5,0);
\node at (5.5,0.1) {and};
\node at (6.5,0) {$\Biggl($};
\draw(7,0.5)--(7.5,0.5);
\draw(7,0)--(7.5,0);
\draw(7,-0.5)--(7.5,-0.5);
\node at (8,0) {$+$};
\draw(8.5,-0.5)--(9.5,0.5);
\draw(8.5,0.5)--(9.5,0);
\draw(8.5,0)--(9.5,-0.5);
\node at (10,0) {$+$};
\draw(10.5,0.5)--(11.5,-0.5);
\draw(10.5,0)--(11.5,0.5);
\draw(10.5,-0.5)--(11.5,0);
\node at (12,0) {$\Biggr)$};
\draw(12.5,0.5)--(13,0.25);
\draw (12.5,0)--(13,0.25)--(13.5,0);
\draw(12.5,-0.5)--(13.5,0)--(14,0);
\node at (14.75,0) {$=0$};
\end{tikzpicture}
\end{equation}
\vspace{0.3 cm}

while the PROPic Lie coalgebra axioms \eqref{eq:prop-lba-three}, \eqref{eq:prop-lba-four} are represented respectively by 

\vspace{0.3 cm}
\begin{equation}
\label{eq:colie}
\begin{tikzpicture}[scale=0.77]
\draw(-1,0)--(0,0)--(0.5,0.5);
\draw(0,0)--(0.5,-0.5);
\node at (1.1,0) {$=-$};
\draw(1.7,0)--(2.5,0)--(3,0.5)--(4,-0.5);
\draw(2.5,0)--(3,-0.5)--(4,0.5);
\draw(6.5,0)--(7,0);
\draw(7,0)--(7.5,0.25)--(8,0.5);
\draw(7.5,0.25)--(8,0);
\draw(7,0)--(8,-0.5);
\node at (6.5+2,0) {$\Biggl($};
\draw(7+2,0.5)--(7.5+2,0.5);
\draw(7+2,0)--(7.5+2,0);
\draw(7+2,-0.5)--(7.5+2,-0.5);
\node at (8+2,0) {$+$};
\draw(8.5+2,-0.5)--(9.5+2,0.5);
\draw(8.5+2,0.5)--(9.5+2,0);
\draw(8.5+2,0)--(9.5+2,-0.5);
\node at (10+2,0) {$+$};
\draw(10.5+2,0.5)--(11.5+2,-0.5);
\draw(10.5+2,0)--(11.5+2,0.5);
\draw(10.5+2,-0.5)--(11.5+2,0);
\node at (12+2,0) {$\Biggr)$};
\node at (14.8,0) {$=0$}; 
\node at (5.5,0.1) {and};
\end{tikzpicture}
\end{equation}
\vspace{0.3 cm}

Finally, the PROPic cocycle condition \eqref{eq:prop-lba-five} is represented by 
\vspace{0.3 cm}
\begin{equation}
\label{eq:cocyclerule}
\begin{tikzpicture}[scale=1.7]
\draw (0.25,0.25)-- (0.75,0);
\draw (0.25,-0.25)--(0.75,0);
\draw (0.75,0)--(1,0);
\draw (1,0)--(1.5,0.25);
\draw (1,0)--(1.5,-0.25);
\node at (1.75,0) {=};
\draw (2,0.25)--(2.25,0.25);
\draw (2,-0.15)--(2.25,-0.15);
\draw (2.25,0.25)-- (2.5,0.05);
\draw (2.25,-0.15)--(2.5,0.05);
\draw (2.25,-0.15)--(2.5,-0.25);
\draw (2.5,-0.25)-- (2.75,-0.25);
\draw (2.5,0.05)-- (2.75,0.05);
\node at (3,0) {+};
\draw (3.25,-0.15)-- (3.5,-0.15);
\draw (3.5,-0.15)--(3.75,0.05);
\draw (3.5,-0.15)--(3.75,-0.25);
\draw (3.75,0.05)--(4.5,0.05);
\draw (3.75,-0.25)--(4,-0.25);
\draw (4,-0.25)--(4.25,-0.15);
\draw (4.25,-0.15)--(4.5,-0.15);
\draw (3.25,0.25)--(4,0.25);
\draw (4,0.25)--(4.25,-0.15);
\node at (4.75,0) {+};
\draw (5,-0.25)-- (5.25,-0.25);
\draw (5,0.05)--(5.25,0.05);
\draw (5.25,0.05)--(5.5,-0.10);
\draw (5.25,-0.25)--(5.5,-0.10);
\draw (5.25,0.05)--(5.5,0.15);
\draw (5.5,0.15)-- (5.75,0.15);
\draw (5.5,-0.10)-- (5.75,-0.10);
\node at (6,0) {+};
\draw (6.25,0.05)-- (6.5,0.05);
\draw (6.5,0.05)-- (6.75,0.25);
\draw (6.5,0.05)-- (6.75,-0.10);
\draw (6.75,0.25)--(7,0.25);
\draw (7,0.25)--(7.25,0.05);
\draw (7.25,0.05)--(7.5,0.05);
\draw (6.75,-0.10)--(7.5,-0.10);
\draw (6.25,-0.25)--(7,-0.25);
\draw (7,-0.25)--(7.25,0.05);
\end{tikzpicture}
\end{equation}
Similarly, we shall represent the generating morphisms of the category $\DY^{1}$ as follows. The identity $\id_{[1]}$ corresponds to a thin horizontal line, whereas the identity $\id_{[V_1]}$ is represented by a horizontal green bold line. The morphisms $\mu,\delta,\pi_1,\pi^*_1$ are thus respectively represented by the diagrams 

\vspace{0.3 cm}
\[
\begin{tikzpicture}
\draw (1,0)--(1.5,-0.5);
\draw (1,-1)--(1.5,-0.5)--(2,-0.5); 
\draw (3,-0.5)--(3.5,-0.5)--(4,0);
\draw (3.5,-0.5)--(4,-1);
\begin{scope}[shift={(5,-0.75)}]
\draw[-, very thick,ggreen] (0,0)--(1.5,0);
\draw (1,0) arc (0:90:0.75);
\draw(0.25,0.75)--(0,0.75);
\end{scope}
\begin{scope}[shift={(7,-0.75)}]
\draw[-, very thick,ggreen] (0,0)--(1.5,0);
\draw (1,0.75) arc (90:180:0.75);
\draw(1,0.75)--(1.25,0.75);
\end{scope}
\end{tikzpicture}
\]
\vspace{0.3 cm}

which are again read from left to right. The fact that the triple $([1], \mu, \delta)$ is a Lie bialgebra object in $\DY^1$ is then represented by the diagrams \eqref{eq:lie}, \eqref{eq:colie}, \eqref{eq:cocyclerule}. Finally, relations \eqref{eq:dy-module-one} \eqref{eq:dy-module-two} \eqref{eq:dy-module-three} are respectively represented by the following three pictorial identities
\vspace{0.3 cm}
\begin{equation}
\label{eq:actionrule}
\begin{tikzpicture}
\draw[-, very thick,ggreen] (-0.25,0)--(1.5,0);
\draw (1,0) arc (0:90:0.75);
\draw(0.25,0.75)--(0,0.75);
\draw (0,0.75)--(-0.25,1);
\draw (0,0.75)--(-0.25,0.5);
\node at (2,0.5) {$=$};
\draw[-, very thick,ggreen] (-0.25+2.75,0)--(1.5+2.75,0);
\draw (1+2.75,0) arc (0:90:0.75);
\draw(0.25+2.75,0.75)--(-0.25+2.75,0.75);
\draw (1+2.55,0) arc (0:90:0.55);
\draw(0.5+2.55,0.55)--(0+2.5,0.55);
\node at (2+2.75,0.5) {$-$};
\draw[-, very thick,ggreen] (-0.25+2.75+3,0)--(1.5+2.75+3,0);
\draw (1+2.75+3+0.2,0) arc (0:90:0.75);
\draw(0.45+2.55+3,0.55)--(2.55+3,0.75);
\draw (2.75+3+0.25,0.75)--(2.75+3+0.45,0.75);
\draw (2.75+3+0.25,0.55)--(2.75+3+0.45,0.55);
\draw (1+2.55+3+0.2,0) arc (0:90:0.55);
\draw(0.25+2.75+3,0.75)--(2.55+3,0.55);
\end{tikzpicture}
\end{equation}

\vspace{0.3 cm}
\begin{equation}
\label{eq:coactionrule}
\begin{tikzpicture}
\draw[-, very thick,ggreen] (0,0)--(1.5,0);
\draw (1,0.75) arc (90:180:0.75);
\draw(1,0.75)--(1.25,0.75);
\draw(1.25,0.75)--(1.5,1);
\draw(1.25,0.75)--(1.5,0.5);
\node at (2.3,0.5) {$=$};
\draw[-, very thick,ggreen] (0+2.75,0)--(1.5+2.95,0);
\draw (1+2.75,0.75) arc (90:180:0.75);
\draw(1+2.75,0.75)--(1.25+2.75,0.75)--(1.5+2.95,0.55);
\draw (1+2.75,0.55) arc (90:180:0.55);
\draw(1+2.75,0.55)--(1.5+2.5,0.55)--(1.5+2.95,0.75);
\node at (2+3.2,0.5) {$-$};
\draw[-, very thick,ggreen] (0+2.75+3,0)--(1.5+2.75+3,0);
\draw (1+2.75+3,0.75) arc (90:180:0.75);
\draw(1+2.75+3,0.75)--(1.25+2.75+3,0.75)--(1.5+2.75+3,0.75);
\draw (1+2.75+3,0.55) arc (90:180:0.55);
\draw(1+2.75+3,0.55)--(1.5+2.5+3,0.55)--(1.5+2.75+3,0.55);
\end{tikzpicture}
\end{equation}

\vspace{0.3 cm}

\begin{equation}
\label{eq:dyrule}
\begin{tikzpicture}
\draw[-, very thick,ggreen] (0+0.5,0)--(2.5+0.5,0);
\draw (1+0.5,0) arc (0:90:0.75);
\draw(0.25+0.5,0.75)--(0+0.5,0.75);
\draw (2.25+0.5,0.75) arc (90:180:0.75);
\draw(2.25+0.5,0.75)--(2.5+0.5,0.75);
\node at (3.5+0.5,0.3) {$=$};
\draw[-, very thick,ggreen] (5,0)--(6.25,0);
\draw (1+5,0) arc (0:90:0.75);
\draw(0.25+5,0.75)--(0+5,0.75);
\draw (1+5,0.75) arc (90:180:0.75);
\draw(1+5,0.75)--(1.25+5,0.75);
\node at (1.8+5,0.4) {$+$}; 
\draw[-, very thick,ggreen] (2.25+5,0)--(3.5+5,0);
\draw (3.25+5,0.75) arc (90:180:0.75); 
\draw(3.25+5,0.75)--(3.5+5,0.75);
\draw (2.8+5,0.6) arc (-300:-260:1); 
\node at (4+5,0.4) {$-$}; 
\draw[-, very thick,ggreen] (4.5+5,0)--(6+5,0); 
\draw (5.5+5,0) arc (0:90:0.75);
\draw(4.75+5,0.75)--(4.5+5,0.75);
\draw (5.75+5,0.75) arc (90:131:0.75); 
\draw (5.75+5,0.75)--(6+5,0.75);
\end{tikzpicture}
\end{equation}

\vspace{0.3 cm}
This pictorial description immediately generalizes to the case of the colored PROP $\DY^n$ by assigning a color to any $[V_k]$. \\
We shall need the following
\begin{definition}
\label{definition-normal-ordering}
We say that a morphism of $\DY^1$ is normally ordered if, in its pictorial representation, all coactions precede all actions and all cobrackets precede all brackets.
\end{definition}
\begin{remark}
Note that, if a morphism is not normally ordered, one can use the Drinfeld--Yetter, action and coaction rules in order to rewrite it as a sum of normally ordered elements. This reasoning will be the key idea of Proposition \ref{proposition-canonical-basis}.
\end{remark}
\subsection{Universal quantization functors}
The concept of universal construction is used to define the PROPic counterpart of universal functors between algebraic structures, such as the universal enveloping algebra of a Lie algebra.
\begin{definition}
Let $\sfP, \sfQ$ be two PROPs, and $\C$ be a symmetric monoidal category.
\begin{itemize}
\item A linear algebraic structure of type $\sfP$ on an object $X$ in $\Obj(\C)$ is a symmetric monoidal functor $F_X: \sfP \to \C$ such that $F_X([n]) = X^{\ten n}$.
\item A universal construction from $\sfP$ to $\sfQ$ is a strict symmetric functor $F: \sfQ \to \sfP$.
\end{itemize}
\end{definition}
\begin{example}
Consider the PROP $\LA$ of Lie algebras and the PROP $\AAA$ of associative algebras, i.e. the PROP generated by two morphisms $m: [2] \to [1]$ and $\eta : [0] \to [1]$ with relations
\begin{equation*}
\begin{split}
(\id_{[1]} \ten m) \circ m &= (m \ten \id_{[1]}) \circ m \\
m \circ (\id_{[1]} \ten m)  &= m \circ (m \ten \id_{[1]}) = \id_{[1]}.
\end{split}
\end{equation*}
The following functor is a universal construction
\begin{equation*}
\begin{split}
\mathsf{Lie} : \LA &\to \AAA \\
[1]_{\LA} & \mapsto [1]_{\AAA} \\
\mu & \mapsto m - m \circ (12).
\end{split}
\end{equation*}
If $A$ is an associative algebra, there is a symmetric monoidal functor 
\begin{equation*}
\begin{split}
F_A: \AAA &\to \Vect_\K \\
[1] & \mapsto A.
\end{split}
\end{equation*}
The fact that any associative algebra $A$ has a natural Lie algebra structure with $[a,b] = ab - ba$ is thus described by the composition of the functors $F_A \circ \mathsf{Lie}: \LA \to \Vect_\K $. 
\end{example}
In order to describe the functor \emph{universal enveloping algebra} in a PROPic way, one needs the following 
\begin{definition}
The Karoubi envelope of a category $\C$ is the category $\C^{kar}$ whose objects are pairs $(X, \pi)$, where $X \in \C$ and $\pi: X \to X$ is an idempotent morphism, and whose morphisms are 
\[ \Hom_{\C^{kar}} \big((X, \pi), (Y, \rho)\big) = \{ f \in \Hom_\C(X,Y) \ | \ \rho \circ f = f = f \circ \pi \}.\]
\end{definition}
In the Karoubi envelope of a category one has that every idempotent splits. Moreover, $\C^{kar}$ is the category containing $\C$ which is universal with respect the property that every idempotent is a split idempotent, see \cite[Lem. 1.8]{balmeridempotent} and references therein for more details.
In particular, if $\sfP$ is a PROP and $\sfP^{kar}$ is its Karoubi envelope, we can consider in $\sfP^{kar}$ the object $[n]_{\sfP^{kar}} \coloneqq ([n], \frac{1}{n!} \sum_{\sigma \in \mathfrak{S}_n} \sigma)$ for any $n \in \mathbb{N}$. Denoting by $\underline{\sfP^{kar}}$ the closure of $\sfP^{kar}$ with respect to all infinite inductive limits, one can consider the object 
\[ S[1] \coloneqq \bigoplus_{n \geqslant 0} \Bigg([n],\frac{1}{n!} \sum_{\sigma \in \mathfrak{S}_n} \sigma \Bigg) \in \underline{\sfP^{kar}}\]
which is the PROPic symmetric algebra. Recalling that for any Lie algebra $\g$ there is an isomorphism of coalgebras $S(\g) \cong \U(\g)$, one can define the universal enveloping algebra functor in a PROPic way through a universal construction, see \cite[6.6]{appel20182} for more details. \\ 
The Etingof--Kazhdan quantization of Lie bialgebras can be expressed through a universal construction, as stated in the following
\begin{theorem}{(Etingof--Kazhdan)}
Let $\QUE$ be the PROP of quantized universal enveloping algebras (see \cite[p.5]{etingof1998quantization} and \cite[p.6]{enriquez2005invertibility}). Then there exists a universal construction
\begin{equation*}
Q : \QUE \to \underline{\LBA}^{kar}[[\hbar]]
\end{equation*}
such that 
\begin{equation*}
\begin{split}
Q([1]_{\QUE}) &= S[[1]_{\LBA}] \\
Q(m) &= m_0 \ \mod \hbar \\
Q(\Delta) &= \Delta_0 \ \mod \hbar \\
Q(\Delta - (12) \circ \Delta) & = \hbar \delta \ \mod \hbar^2
\end{split}
\end{equation*}
where $m_0$ (resp. $\Delta_0$) is the multiplication (resp. comultiplication) of the (universal) universal enveloping algebra $\U([1]_{\LBA})$.
\end{theorem}
The Etingof--Kazhdan quantization technique provides thus a universal quantization functor, (see \cite[\S 1]{etingof1998quantization} and \cite[\S 6.7-- \S 6.17]{appel20182}), solving the problem $Q 1.2$ stated by Drinfeld \cite{drinfeld1992some}.

\section{The universal Drinfeld--Yetter algebra}
\label{section-universal-algebrass}
\subsection{Enriquez universal algebras}
In this Section we recall Enriquez's universal algebras  \cite{enriquez2001quantization} \cite{enriquezonsome} \cite{enriquez2005cohomological}.
The original idea of Enriquez was to define a family of universal algebras $\{\mathcal{U} \mathfrak{S}_{univ}^n\}_{n \geqslant 0}$ with the following properties: 
\begin{itemize}
\item (Universal property):  for any quasi--triangular, finite--dimensional Lie bialgebra $\b$ there exists a morphism of algebras 
\begin{equation}
\label{eq:enriquez-realization-map}
\rho^n_\b : \mathcal{U} \mathfrak{S}^n_{univ} \to \U(\b)^{\ten n}.
\end{equation}
\item There exists a family of insertion-coproduct maps $\mathcal{U} \mathfrak{S}^n_{univ} \to \mathcal{U} \mathfrak{S}^{n+1}_{univ}$ which gives rise to a universal version of the coHochschild differential of $\U(\b)$.
\end{itemize}
It is well--known that the existence of quantizations of a Lie bialgebra $\b$ is governed by the Hochschild cohomology of $\U(\b)$, see e.g. \cite[XVIII]{kassel}. Hence, the idea of Enriquez was to replicate the Drinfeld's cohomological proof of the existence of quantization of Lie bialgebras, obtaining a cohomological interpretation of the Etingof--Kazhdan quantization. In particular, Enriquez's main result provides, for any Drinfeld associator $\Phi$, a universal twist $J_\Phi \in \mathcal{U}\mathfrak{S}^2_{univ}$ \emph{killing the associator}, see \cite[Th. 2.1]{enriquez2005cohomological}. The universal realization maps \eqref{eq:enriquez-realization-map} allows thus to, for any finite--dimensional, quasi--triangular Lie bialgebra $\b$, realize the twist on $\U(\b)$, giving rise to a universal quantization.  \\
Enriquez's algebras are defined as follows:
\begin{definition}
For any $n,N \geqslant 1$ let $\mathcal{A}_N$ be the free algebra in $N$ variables $x_i$, $i=1, \ldots, N$ and let $(\mathcal{A}_N^{\ten n})_{\delta_N}$ be the subspace of $\mathcal{A}_N^{\ten n}$ generated by elements of degree one in each variable. We have that the symmetric group $\mathfrak{S}_N$ acts diagonally on $(\mathcal{A}_N^{\ten n})_{\delta_N} \ten (\mathcal{A}_N^{\ten n})_{\delta_N}$ by simultaneus permutation of the variables. The $n$--th Enriquez's universal algebra is 
\[ \mathcal{U} \mathfrak{S}^n_{univ} = \sum_{N \geqslant 0} \big( (\mathcal{A}_N^{\ten n})_{\delta_N} \ten (\mathcal{A}_N^{\ten n})_{\delta_N} \big)_{\mathfrak{S}_N}\]
where $\big( (\mathcal{A}_N^{\ten n})_{\delta_N} \ten (\mathcal{A}_N^{\ten n})_{\delta_N} \big)_{\mathfrak{S}_N}$ is the space of $\mathfrak{S}_n$--coinvariants. 
\end{definition}
The algebra $\mathcal{U} \mathfrak{S}^n_{univ} $ is equipped with a standard basis defined as follows. For any $\underline{N}, \underline{N}' \in \mathbb{N}^n$ with $|\underline{N}| = |\underline{N}'| = N $ and $\sigma \in \mathfrak{S}_N$, consider elements $x_{\underline{N}}$ and $y_{\sigma(\underline{N}')}$ of $(\mathcal{A}_N^{\ten n})_{\delta_N}$ defined by 
\begin{equation*}
\begin{split}
x_{\underline{N}} &= x_1 \cdots x_{N_1} \ten x_{N_1 + 1} \cdots x_{N_1 + N_2} \ten \cdots \ten x_{N_1 + \cdots + N_{n-1}+1} \cdots x_N\\
y_{\sigma(\underline{N}')} &= y_{\sigma(1)} \cdots y_{\sigma(N_1')} \ten \cdots \ten y_{\sigma(N_1' + \cdots + N_{n-1}'+1} \cdots y_{\sigma(N)}
\end{split}
\end{equation*}
Then the collection $\{ x_{\underline{N}} \ten y_{\sigma(\underline{N}')}\}$ is a basis of $\mathcal{U} \mathfrak{S}^n_{univ}$. The algebra structure of $\mathcal{U} \mathfrak{S}^n_{univ}$ is provided in \cite{enriquez2001quantization} through a very intricate formula, and is proved to be associative by a lenghty computation. For $n=1$, one has that $\mathcal{U} \mathfrak{S}^1_{univ}$ is isomorphic -- as a vector space -- to the direct sum $\bigoplus_{N \geqslant 0} \K[\mathfrak{S}_N]$, and the product is the concatenation of permutations. \\
Finally, for any finite--dimensional, quasi--triangular Lie bialgebra $\b$ with $r$--matrix $r = \sum_i b_i \ten b^i$, the realization map -- we provide for simplicity the case $n=1$ -- $\rho_\b$ is defined by
\begin{equation*}
\begin{split}
\rho_\b : \mathcal{U} \mathfrak{S}^1_{univ} &\to \U(\b) \\
x_{N} \ten y(\sigma(N)) & \mapsto \sum_{i \in I^N} b_{i_1} \cdots b_{i_N} b^{i_{\sigma(1)}} \cdots b^{i_{\sigma(N)}}.
\end{split}
\end{equation*}
However, it turns out that such a map does not satisfy the desired universal property, as showed in the following 
\begin{example}
\label{counterexample}
Consider the quasi--triangular complex Lie bialgebra $\mathfrak{sl}_2$ with standard generators $e,f,h$ and standard $r$--matrix $r = e \ten f + \frac{1}{4} h \ten h$. Let $\{e^i f^j h^k , \ i,j,k \in \mathbb{N} \}$ be the Poincar\'e--Birkhoff--Witt basis of $\U( \mathfrak{sl}_2)$ and consider the elements $\id_1 \in \mathfrak{S}_1$, $\id_2\in \mathfrak{S}_2$, $(12) \in \mathfrak{S}_2$. Then it is easy to see -- through a lenghty but elementary computation -- that
\begin{equation*}
\begin{split}
\rho_{\mathfrak{sl}_2}(\id_1) &= ef + \frac{1}{4} h^2 \\
\big(\rho_{\mathfrak{sl}_2}(\id_1)\big)^2 &= e^2f^2 - efh  + \frac{efh^2}{2} + \frac{h^4}{16} +2ef\\
\rho_{\mathfrak{sl}_2}(\id_2) &= e^2f^2 - efh + \frac{efh^2}{2}+\frac{h^4}{16} +ef \\
\rho_{\mathfrak{sl}_2}\big((12)\big) &= e^2f^2 - efh + \frac{efh^2}{2}+\frac{h^4}{16}
\end{split}
\end{equation*}
It is thus clear that $(\rho_{\mathfrak{sl}_2}(\id_1))^2 \neq \rho_{\mathfrak{sl}_2}(\id_2) $, i.e. that the concatenation of permutations does not satisfy the required universal property. On the other hand, we have the following identity
\begin{equation}
\label{eq:counterexample-enriquez}
 (\rho_{\mathfrak{sl}_2}(\id_1))^2 = 2 \cdot \rho_{\mathfrak{sl}_2}(\id_2) -  \rho_{\mathfrak{sl}_2}((12)).
\end{equation}
\end{example}

\subsection{Appel--Toledano Laredo universal algebras}
We now present Drinfeld--Yetter universal algebras, defined by Appel and Toledano Laredo -- in their attempt to clarify the Enriquez's construction -- in \cite{appeluniqueness}. 
\begin{definition}
Let $n \geqslant 1$. The $n$--th universal Drinfeld--Yetter algebra $\mathfrak{U}_\DY^n$ is 
\[ \mathfrak{U}_\DY^n \coloneqq \End_{\DY^n}([V_1] \ten \cdots \ten [V_n]) \]
where the associative multiplication the composition of endomorphisms.
\end{definition}
The algebra $\mathfrak{U}_\DY^n $ has the following vector space structure. For any $N \in \mathbb{N}$ and $\underline{N}=(N_1, \ldots, N_n) \in \mathbb{N}^n$ such that $|\underline{N}| = N$ consider the following morphisms of $\DY^n$
\[ \pi^{(\underline{N})}: [N] \ten \bigotimes_{k=1}^n[V_k] \to \bigotimes_{k=1}^n[V_k] \qquad \text{and} \qquad \pi^{*(\underline{N})} : \bigotimes_{k=1}^n[V_k] \to  [N] \ten \bigotimes_{k=1}^n[V_k] \]
which are respectively the ordered composition of $N_i$ actions (resp. coactions) on $[V_i]$. Then we have 
\begin{proposition}{(\cite[Prop. 5.12]{appeluniqueness})}
The collection of elements 
\[ \big\{r_{\underline{N}, \underline{N}'}^\sigma \coloneqq \pi^{(\underline{N})} \circ (\sigma \ten \id) \circ \pi^{*(\underline{N}')}\big\}\]
where $N \geqslant 0$, $\underline{N}, \underline{N}' \in \mathbb{N}^n$ are such that $|\underline{N}| = |\underline{N}'|= N$, and $\sigma \in \mathfrak{S}_N$ is a basis of $\mathfrak{U}_\DY^n $.
\end{proposition}
From now on we shall focus on the case $n=1$. The following proposition gives the universal property of the algebra $\mathfrak{U}_\DY^1$:
\begin{proposition}
For any finite--dimensional, quasi--triangular Lie bialgebra $\b$ with $r$--matrix $r= \sum_{i \in I} a_i \ten b_i$ the map
\begin{equation}
\label{eq:appel-realization-map}
\begin{split}
\rho_\b : \mathfrak{U}_{\DY}^1 &\to \U(\b) \\
r_n^\sigma& \mapsto \sum_{i_1 \in I} \cdots \sum_{i_n \in I} a_{i_1} \cdots a_{i_n} b_{i_{\sigma^{-1}(n)}} \cdots b_{i_{\sigma^{-1}(1)}}
\end{split}
\end{equation}
is a morphism of algebras, where $r_n^\sigma \coloneqq \pi_1^{(n)}\circ(\sigma\ten\id_{[V_1]})\circ\;{\pi_1}^{*(n)}$.
\end{proposition} It is possible to show that for any infinite--dimensional, quasi--triangular Lie bialgebra $\b$ the map \eqref{eq:appel-realization-map} satisfies -- up to completing opportunely the algebra $\U(\b)$ -- such a universal property. \\
The algebras $\mathfrak{U}_\DY^n$ and $\mathcal{U} \mathfrak{S}^n_{univ} $ are related by the following
\begin{proposition}
We have that:
\begin{itemize}
\item[(i)] (\cite[p.31]{appeluniqueness})  There exist algebra homomorphisms $\Delta_i^n : \mathfrak{U}_\DY^n \to \mathfrak{U}_\DY^{n+1}$ giving to the tower of algebras $\{ \mathfrak{U}_\DY^n\}_{n \geqslant 0}$ the structure of a cosimplicial complex. 
\item[(ii)] (\cite[p.37]{appeluniqueness}) The collection of maps 
\begin{equation*}
\begin{split}
\xi^n : \mathfrak{U}_\DY^n &\to \mathcal{U} \mathfrak{S}^n_{univ} \\
r^\sigma_{\underline{N}, \underline{N}'} & \mapsto x_{\underline{N}} \ten y_{\tilde{\sigma}(\underline{N}')}
\end{split}
\end{equation*}
is a collection of isomorphims of vector spaces, where $\tilde{\sigma} \coloneqq \sigma^{-1} \circ \tau$, and $\tau$ is the element of $\mathfrak{S}_N$ such that $\tau(i)=N-i$.
\item[(iii)] (\cite[p.44]{appeluniqueness}) The isomorphisms $\xi^n$ induce an isomorphism of cosimplicial chains. 
\end{itemize}
\end{proposition}
Therefore, Appel--Toledano Laredo's universal algebras provide a PROPic refinement of Enriquez's ones, satisfying further the required universal property.
In the rest of the paper we shall study the structure of the algebra $\mathfrak{U}_\DY^1$.

\begin{notation}
	From now on we shall denote the composition of two morphisms in a PROP from left to right (as they are read in the respectively pictorial representation).
\end{notation}
\subsection{The universal Drinfeld--Yetter algebra}
\label{subsection-dy-algebra}
\begin{definition}
	\label{definition-u-d-y-algebra}
	The universal Drinfeld--Yetter algebra is $\mathfrak{U}_{\DY}^1 \coloneqq \End_{\DY} ([V_1])$.
\end{definition}
It is immediate to note that any element of $\mathfrak{U}_{\DY}^1 $ different from the identity can be pictorially represented by a linear combination of oriented diagrams of the form 
\vspace{0.3 cm}
\begin{equation}
	\label{eq:generic-element}
	\begin{tikzpicture} 
		\draw[-, ggreen,very thick] (0,0)--(3.8,0);
		\draw[very thick] (1,0.75) arc (90:180:0.75);
		\draw[very thick] (1,0.75)--(1.2,0.75);
		\draw[black] (1.2,0.3) rectangle (2.6,1);
		\node at (1.9, 0.65) {$\phi$};
		\draw[very thick] (1.6+1,0.75)--(1.8+1,0.75);
		\draw[very thick] (2.55+1,0) arc (0:90:0.75);
		\node at (2.55+1,-0.3) {$m$};
		\node at (0.3,-0.3) {$n$};
	\end{tikzpicture}
\end{equation}
\vspace{0.3 cm}

i.e. each of them necessarily starting with a certain number $n \geqslant  1$ of coactions, ending with a certain number $m \geqslant  1$ of actions, and containing  a morphism $\phi \in \Hom_{\DY^1}([n]\ten [V_1],[m] \ten [V_1])$ in the middle. If $n$ is a non--negative integer and $\sigma \in \mathfrak{S}_n$, we denote by $r_{n}^{\sigma}$ the following element of $\mathfrak{U}_{\DY}^1$
\[r_{n}^{\sigma}= \pi_1^{*(n)}\circ(\sigma\ten\id_{[V_1]})\circ\;{\pi_1}^{(n)}, \]
where $\pi_1^{(n)}: [n] \ten [V_1] \to [V_1]$ denotes the $n$-th iterated action and $\pi_1^{*(n)}: [V_1] \to [n] \ten [V_1]$ denotes the $n$-th iterated coaction.
The pictorial representation of $r_{n}^{\sigma}$ is the following

\vspace{0.3 cm}
\begin{equation}
	\label{eq:basis-element}
	\begin{tikzpicture} 
		\draw[-, very thick,ggreen] (0,0)--(2.8,0);
		\draw[very thick] (1,0.75) arc (90:180:0.75);
		\draw[very thick]  (1,0.75)--(1.2,0.75);
		\node[shape=circle,draw,inner sep=1pt] (char) at (1.4,0.75) {$\sigma$}; 
		\draw[very thick]  (1.6,0.75)--(1.8,0.75);
		\draw[very thick]  (2.55,0) arc (0:90:0.75);
		\node at (2.55,-0.3) {$n$};
		\node at (0.3,-0.3) {$n$};
	\end{tikzpicture}
\end{equation}
The following result describes the vector space structure of $ \mathfrak{U}_{\DY}^1$:
\begin{proposition}{ \cite[5.10]{appeluniqueness}}
	\label{proposition-canonical-basis}
	The collection of endomorphisms of $[V_1]$ given by 
	\begin{equation}
		\label{eq:canonical-basis}
		\B = \{ r_{n}^{\sigma}, n \geqslant  0, \sigma \in \mathfrak{S}_n \} 
	\end{equation}
	is a basis of  $ \mathfrak{U}_{\DY}^1$.
\end{proposition}
\begin{proof}
	In order to show that the set $\B$ generates $ \mathfrak{U}_{\DY}^1$, let $f \in \End_{\DY}([V_1])$ be represented by a linear combination of oriented diagrams of the form \eqref{eq:generic-element}. The Drinfeld--Yetter relation \eqref{eq:dyrule} 
	\begin{equation*}
		\begin{tikzpicture}
			\draw[-, very thick,ggreen] (0+0.5,0)--(2.5+0.5,0);
			\draw (1+0.5,0) arc (0:90:0.75);
			\draw(0.25+0.5,0.75)--(0+0.5,0.75);
			\draw (2.25+0.5,0.75) arc (90:180:0.75);
			\draw(2.25+0.5,0.75)--(2.5+0.5,0.75);
			\node at (3.5+0.5,0.3) {$=$};
			\draw[-, very thick,ggreen] (5,0)--(6.25,0);
			\draw (1+5,0) arc (0:90:0.75);
			\draw(0.25+5,0.75)--(0+5,0.75);
			\draw (1+5,0.75) arc (90:180:0.75);
			\draw(1+5,0.75)--(1.25+5,0.75);
			\node at (1.8+5,0.4) {$+$}; 
			\draw[-, very thick,ggreen] (2.25+5,0)--(3.5+5,0);
			\draw (3.25+5,0.75) arc (90:180:0.75); 
			\draw(3.25+5,0.75)--(3.5+5,0.75);
			\draw (2.8+5,0.6) arc (-300:-260:1); 
			\node at (4+5,0.4) {$-$}; 
			\draw[-, very thick,ggreen] (4.5+5,0)--(6+5,0); 
			\draw (5.5+5,0) arc (0:90:0.75);
			\draw(4.75+5,0.75)--(4.5+5,0.75);
			\draw (5.75+5,0.75) arc (90:131:0.75); 
			\draw (5.75+5,0.75)--(6+5,0.75);
		\end{tikzpicture}
	\end{equation*}
	allows to reorder $\pi_1$ and $\pi^*_1$, moving every coaction before any action. The cocycle condition \eqref{eq:cocyclerule}
	\begin{align*}
		&\\
		&\begin{tikzpicture}[scale=1.7]
			\draw (0.25,0.25)-- (0.75,0);
			\draw (0.25,-0.25)--(0.75,0);
			\draw (0.75,0)--(1,0);
			\draw (1,0)--(1.5,0.25);
			\draw (1,0)--(1.5,-0.25);
			\node at (1.75,0) {=};
			\draw (2,0.25)--(2.25,0.25);
			\draw (2,-0.15)--(2.25,-0.15);
			\draw (2.25,0.25)-- (2.5,0.05);
			\draw (2.25,-0.15)--(2.5,0.05);
			\draw (2.25,-0.15)--(2.5,-0.25);
			\draw (2.5,-0.25)-- (2.75,-0.25);
			\draw (2.5,0.05)-- (2.75,0.05);
			\node at (3,0) {+};
			\draw (3.25,-0.15)-- (3.5,-0.15);
			\draw (3.5,-0.15)--(3.75,0.05);
			\draw (3.5,-0.15)--(3.75,-0.25);
			\draw (3.75,0.05)--(4.5,0.05);
			\draw (3.75,-0.25)--(4,-0.25);
			\draw (4,-0.25)--(4.25,-0.15);
			\draw (4.25,-0.15)--(4.5,-0.15);
			\draw (3.25,0.25)--(4,0.25);
			\draw (4,0.25)--(4.25,-0.15);
			\node at (4.75,0) {+};
			\draw (5,-0.25)-- (5.25,-0.25);
			\draw (5,0.05)--(5.25,0.05);
			\draw (5.25,0.05)--(5.5,-0.10);
			\draw (5.25,-0.25)--(5.5,-0.10);
			\draw (5.25,0.05)--(5.5,0.15);
			\draw (5.5,0.15)-- (5.75,0.15);
			\draw (5.5,-0.10)-- (5.75,-0.10);
			\node at (6,0) {+};
			\draw (6.25,0.05)-- (6.5,0.05);
			\draw (6.5,0.05)-- (6.75,0.25);
			\draw (6.5,0.05)-- (6.75,-0.10);
			\draw (6.75,0.25)--(7,0.25);
			\draw (7,0.25)--(7.25,0.05);
			\draw (7.25,0.05)--(7.5,0.05);
			\draw (6.75,-0.10)--(7.5,-0.10);
			\draw (6.25,-0.25)--(7,-0.25);
			\draw (7,-0.25)--(7.25,0.05);
		\end{tikzpicture}\\
	\end{align*}
	allows to reorder brackets and cobrackets in such a way cobrackets horizontally precede brackets. Finally, the relations \eqref{eq:actionrule}, \eqref{eq:coactionrule}\\
	\begin{align*} 
		&\begin{tikzpicture}
			\draw[-, very thick,ggreen] (-0.25,0)--(1.5,0);
			\draw (1,0) arc (0:90:0.75);
			\draw(0.25,0.75)--(0,0.75);
			\draw (0,0.75)--(-0.25,1);
			\draw (0,0.75)--(-0.25,0.5);
			\node at (2,0.5) {$=$};
			\draw[-, very thick,ggreen] (-0.25+2.75,0)--(1.5+2.75,0);
			\draw (1+2.75,0) arc (0:90:0.75);
			\draw(0.25+2.75,0.75)--(-0.25+2.75,0.75);
			\draw (1+2.55,0) arc (0:90:0.55);
			\draw(0.5+2.55,0.55)--(0+2.5,0.55);
			\node at (2+2.75,0.5) {$-$};
			\draw[-, very thick,ggreen] (-0.25+2.75+3,0)--(1.5+2.75+3,0);
			\draw (1+2.75+3+0.2,0) arc (0:90:0.75);
			\draw(0.45+2.55+3,0.55)--(2.55+3,0.75);
			\draw (2.75+3+0.25,0.75)--(2.75+3+0.45,0.75);
			\draw (2.75+3+0.25,0.55)--(2.75+3+0.45,0.55);
			\draw (1+2.55+3+0.2,0) arc (0:90:0.55);
			\draw(0.25+2.75+3,0.75)--(2.55+3,0.55);
		\end{tikzpicture}\\
		&\begin{tikzpicture}
			\draw[-, very thick,ggreen] (0,0)--(1.5,0);
			\draw (1,0.75) arc (90:180:0.75);
			\draw(1,0.75)--(1.25,0.75);
			\draw(1.25,0.75)--(1.5,1);
			\draw(1.25,0.75)--(1.5,0.5);
			\node at (2.3,0.5) {$=$};
			\draw[-, very thick,ggreen] (0+2.75,0)--(1.5+2.95,0);
			\draw (1+2.75,0.75) arc (90:180:0.75);
			\draw(1+2.75,0.75)--(1.25+2.75,0.75)--(1.5+2.95,0.55);
			\draw (1+2.75,0.55) arc (90:180:0.55);
			\draw(1+2.75,0.55)--(1.5+2.5,0.55)--(1.5+2.95,0.75);
			\node at (2+3.2,0.5) {$-$};
			\draw[-, very thick,ggreen] (0+2.75+3,0)--(1.5+2.75+3,0);
			\draw (1+2.75+3,0.75) arc (90:180:0.75);
			\draw(1+2.75+3,0.75)--(1.25+2.75+3,0.75)--(1.5+2.75+3,0.75);
			\draw (1+2.75+3,0.55) arc (90:180:0.55);
			\draw(1+2.75+3,0.55)--(1.5+2.5+3,0.55)--(1.5+2.75+3,0.55);
		\end{tikzpicture}\\
	\end{align*}
	allow to remove from the graph every $\mu$ and every $\delta$ involved. It follows that $f$ can be represented as a linear combination of endomorphisms of type \eqref{eq:basis-element}, showing that $\B$ is a set of generators for $\mathfrak{U}_{\DY}^1$. In order to show that the $r_n^\sigma$'s are linearly independent, consider a Lie bialgebra $\b$ whose underlying Lie algebra is free. Therefore, any non--trivial linear combination $\sum_i c_i r_{n_i}^{\sigma_i}=0$ would induce, through the universal property of $\mathfrak{U}_\DY^1$, a non--trivial relation in $\U( \b)$, contradicting its freeness.
\end{proof}
A direct consequence of the result above is the following
\begin{corollary}
	There is a canonical isomorphism of vector spaces 
	\begin{equation}
		\label{eq:isomorphism-symmetric-algebras}
		\mathfrak{U}_{\DY}^1 \simeq \bigoplus_{n \geqslant  0} \mathbb{K}[\mathfrak{S}_n]
	\end{equation}
	mapping $r_{n}^{\sigma}$ to $\sigma \in \mathfrak{S}_n$.
\end{corollary}
From now on we shall refer to \eqref{eq:canonical-basis} as the standard basis of $\mathfrak{U}_{\DY}^1$.
In addition to describing the vector space structure of the algebra $\mathfrak{U}_{\DY}^1$, Proposition \ref{proposition-canonical-basis} gives an algorithmic way to compute the multiplication with respect to the standard basis \eqref{eq:canonical-basis}. Namely, one can proceed in the following way: given $r_{n}^{\sigma},r_{m}^{\tau} \in \B$, consider the following algorithm:
\begin{itemize}
	\item[(1)] Apply repeatedly the Drinfeld--Yetter rule \eqref{eq:dyrule} until there is no action preceding any coaction;
	\item[(2)] Apply repeatedly the action rule \eqref{eq:actionrule} in order to remove all the brackets from the expansion obtained in (1);
	\item[(3)] Apply repeatedly the coaction rule \eqref{eq:coactionrule} in order to remove all the cobrackets from the expansion obtained in (2).
\end{itemize}
The result of this process gives the multiplication $r_n^\sigma \circ r_m^\tau$ with respect to the standard basis \eqref{eq:canonical-basis}.

\begin{remark}
	Since in every step of the algorithm described above the total number of the strings is preserved, we have that $\mathfrak{U}_{\DY}^1$ is a $\mathbb{N}$--graded algebra: for any $r_{n}^{\sigma},r_{m}^{\tau} \in \B$ there exist unique coefficients $c_{\sigma, \tau}^{\pi}$ such that 
	\begin{equation}
		\label{eq:definition-of-coefficients}
		r_{n}^{\sigma} \circ r_{m}^{\tau} = \sum_{\pi \in \mathfrak{S}_{n+m}} c_{\sigma, \tau}^{\pi} \cdot r_{n+m}^\pi.
	\end{equation}
	Moreover, it follows from the identities involved that the coefficients $c_{\sigma, \tau}^{\pi}$ are integers. Therefore, the structure of $\mathfrak{U}_{\DY}^1$ naturally induces a $\mathbb{N}$--graded algebra structure (with integral structure constants) on the vector space $\bigoplus_{n \geqslant  0} \mathbb{K}[\mathfrak{S}_n]$.
\end{remark}
A really challenging problem is to find an explicit formula for the structure constants $c_{\sigma, \tau}^\pi$ in terms of symmetric groups. This problem appears to be highly nontrivial: for instance, the number of summands appearing after the application of the algorithm described above seems to have exponential growth, as conjectured in \ref{conjecture-cardinality-looms}.\\
In the next Sections we shall describe the algebra structure of $\mathfrak{U}_{\DY}^1$ through two combinatorial objects, namely the Drinfeld--Yetter mosaics and Drinfeld--Yetter looms.

\section{Drinfeld--Yetter mosaics and looms}
\label{section-two}

\subsection{Drinfeld--Yetter mosaics}
\label{section-d-y-mosaics}
In this Section we define the set $\mosaic{n}{m}$ of $n \times m$ Drinfeld--Yetter mosaics, which will provide, through Proposition \ref{claim}, a combinatorial interpretation of the morphism $\pi_1^{(n)} \circ \pi_1^{*(m)}$, which is pictorially represented by the picture
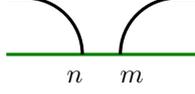
\begin{equation*}
\begin{tikzpicture} 
\draw[-, very thick,ggreen] (0,0)--(2.5,0);
\draw[very thick]  (1,0) arc (0:90:0.75);
\draw[very thick] (0.25,0.75)--(0,0.75);
\draw[very thick]  (2.25,0.75) arc (90:180:0.75);
\draw[very thick] (2.25,0.75)--(2.5,0.75);
\node at (1.65,-0.3) {$m$};
\node at (0.9,-0.3) {$n$}; 
\end{tikzpicture}
\end{equation*}
More specifically, we are going to show that every element of $\mosaic{n}{m}$ is associated (through Equation \eqref{eq:morphism-from-mosaic-to-hom}) to a morphism appearing in the sum of morphisms generated by the iterated application of the Drinfeld--Yetter rule \eqref{eq:dyrule} to $\pi_1^{(n)} \circ \pi_1^{*(m)}$, hence giving a combinatorial description of the application of the first step of the algorithm described in \S \ref{subsection-dy-algebra}. 
\begin{notation}
Let $n,m \geqslant 1$. We denote by $\grid{n}{m}$ the \index{grid}grid with $n$ rows and $m$ columns. For example, $\grid{2}{3}$ is 
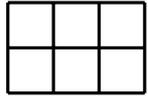
\vspace{0.5 cm}
\begin{center}
 \begin{tikzpicture}[scale=0.6]
\gri{2}{3}
\end{tikzpicture}.
\end{center}
\vspace{0.5 cm}
If $\T$ is a given set of tiles, we define a tiling\index{tiling} of $\grid{n}{m}$ with the elements of $\T$ as a map 
\[ F: \{1,\ldots, n\} \times \{1,\ldots, m\} \to \T. \]
We shall denote $F(i,j)$ by $F_{i,j}$. In other words, a tiling of $\grid{n}{m}$ consists of assigning to each position of the empty grid a tile of $\T$.
\end{notation}
\begin{definition}
\label{definition-mosaics}
Let $n,m \geqslant 1$ and let $\T_{\mathfrak{M}}$ be the following set of tiles
\begin{equation}
\label{eq:tiles-tableaux}
\begin{tikzpicture}[scale=0.6]
\node at (0.3,0) {$\mathcal{T}_\mathfrak{M}= \Big\{$};
\draw[black,very thick] (1.5,0.5) rectangle (2.5,-0.5);
\draw[red,very thick] (2,0.5)--(2,-0.5);
\draw[red,very thick] (1.5,0)--(2.5,0);
\node at (2.75,-0.5) {,};
\draw[black,very thick] (3,0.5) rectangle (4,-0.5);
\draw[red,very thick] (3,0)--(3.5,0);
\draw[red,very thick] (3.5,0.5)--(3.5,-0.5);
\node at (4.25,-0.5) {,};
\draw[black,very thick] (4.5,0.5) rectangle (5.5,-0.5);
\draw[red,very thick] (5,0.5)--(5,0);
\draw[red,very thick] (4.5,0)--(5.5,0);
\node at (5.75,-0.5) {,};
\draw[black,very thick] (6,0.5) rectangle (7,-0.5);
\draw[red,very thick] (6,0)--(7,0);
\node at (7.25,-0.5) {,};
\draw[black,very thick] (7.5,0.5) rectangle (8.5,-0.5);
\draw[red,very thick] (8,0.5)--(8,-0.5);
\node at (8.75,-0.5) {,};
\draw[black,very thick] (9,0.5) rectangle (10,-0.5);
\node at (10.5,-0) {\Big\}.};
\end{tikzpicture}
\end{equation}
We define the set of $n \times m$ Drinfeld--Yetter mosaics $\mosaic{n}{m}$ as the set of all possible tilings $M$ of $\grid{n}{m}$ with the elements of $\T_\mathfrak{M}$ such that the following three conditions are satisfied:
\begin{itemize}
\item[(1):] $M_{1,j} \notin \Big\{ \begin{tikzpicture}[scale= 0.35] \gri{1}{1} \hor{1}{1}

\end{tikzpicture} \ , \ \begin{tikzpicture}[scale= 0.35] \gri{1}{1}

\end{tikzpicture} \Big\}$ for all $j \in \{1, \ldots , m\}$.
\item[(2):] $M_{i,1} \notin \Big\{ \begin{tikzpicture}[scale= 0.35] \gri{1}{1} \ver{1}{1}

\end{tikzpicture} \ , \ \begin{tikzpicture}[scale= 0.35] \gri{1}{1}

\end{tikzpicture} \Big\}$ for all $i \in \{1, \ldots , n\}$.
\item[(3):] None of the following configurations appear in $M$:
\\ \\
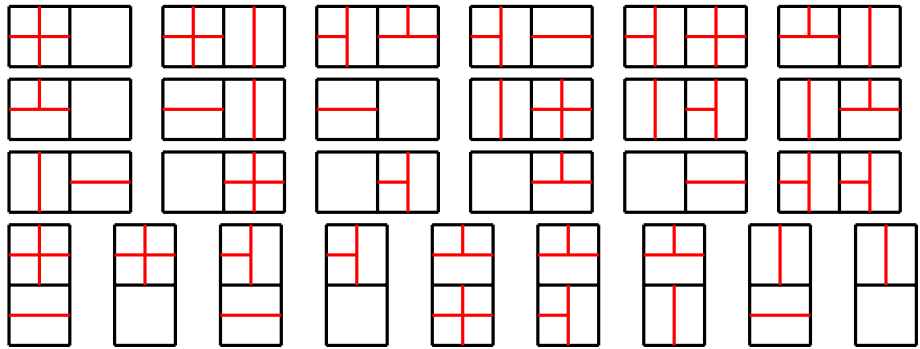
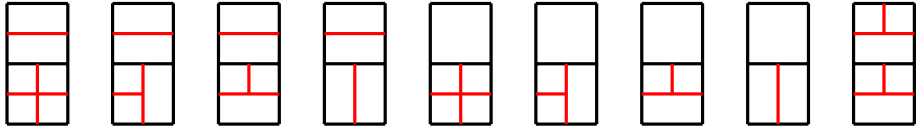
\begin{minipage}{0,16\textwidth}
\begin{tikzpicture}[scale= 0.8]
\gri{1}{2} \cro{1}{1}
\end{tikzpicture}
\end{minipage}%
\begin{minipage}{0,16\textwidth}
\begin{tikzpicture}[scale= 0.8]
\gri{1}{2} \cro{1}{1} \ver{1}{2}
\end{tikzpicture}
\end{minipage}%
\begin{minipage}{0,16\textwidth}
\begin{tikzpicture}[scale= 0.8]
\gri{1}{2} \tbra{1}{1} \tcobra{1}{2}
\end{tikzpicture}
\end{minipage}%
\begin{minipage}{0,16\textwidth}
\begin{tikzpicture}[scale= 0.8]
\gri{1}{2} \tbra{1}{1} \hor{1}{2}
\end{tikzpicture}
\end{minipage}%
\begin{minipage}{0,16\textwidth}
\begin{tikzpicture}[scale= 0.8]
\gri{1}{2} \tbra{1}{1} \cro{1}{2}
\end{tikzpicture}
\end{minipage}%
\begin{minipage}{0,16\textwidth}
\begin{tikzpicture}[scale= 0.8]
\gri{1}{2} \tcobra{1}{1} \ver{1}{2}
\end{tikzpicture}
\end{minipage}
\\ \\
\begin{minipage}{0,16\textwidth}
\begin{tikzpicture}[scale= 0.8]
\gri{1}{2} \tcobra{1}{1}
\end{tikzpicture}
\end{minipage}%
\begin{minipage}{0,16\textwidth}
\begin{tikzpicture}[scale= 0.8]
\gri{1}{2} \hor{1}{1} \ver{1}{2}
\end{tikzpicture}
\end{minipage}%
\begin{minipage}{0,16\textwidth}
\begin{tikzpicture}[scale= 0.8]
\gri{1}{2} \hor{1}{1}
\end{tikzpicture}
\end{minipage}%
\begin{minipage}{0,16\textwidth}
\begin{tikzpicture}[scale= 0.8]
\gri{1}{2} \ver{1}{1} \cro{1}{2}
\end{tikzpicture}
\end{minipage}%
\begin{minipage}{0,16\textwidth}
\begin{tikzpicture}[scale= 0.8]
\gri{1}{2} \ver{1}{1} \tbra{1}{2}
\end{tikzpicture}
\end{minipage}%
\begin{minipage}{0,16\textwidth}
\begin{tikzpicture}[scale= 0.8]
\gri{1}{2} \ver{1}{1} \tcobra{1}{2}
\end{tikzpicture}
\end{minipage}
\\ \\
\begin{minipage}{0,16\textwidth}
\begin{tikzpicture}[scale= 0.8]
\gri{1}{2} \ver{1}{1} \hor{1}{2}
\end{tikzpicture}
\end{minipage}%
\begin{minipage}{0,16\textwidth}
\begin{tikzpicture}[scale= 0.8]
\gri{1}{2} \cro{1}{2}
\end{tikzpicture}
\end{minipage}%
\begin{minipage}{0,16\textwidth}
\begin{tikzpicture}[scale= 0.8]
\gri{1}{2} \tbra{1}{2}
\end{tikzpicture}
\end{minipage}%
\begin{minipage}{0,16\textwidth}
\begin{tikzpicture}[scale= 0.8]
\gri{1}{2} \tcobra{1}{2}
\end{tikzpicture}
\end{minipage}%
\begin{minipage}{0,16\textwidth}
\begin{tikzpicture}[scale= 0.8]
\gri{1}{2} \hor{1}{2}
\end{tikzpicture}
\end{minipage}%
\begin{minipage}{0,16\textwidth}
\begin{tikzpicture}[scale= 0.8]
\gri{1}{2} \tbra{1}{1} \tbra{1}{2}
\end{tikzpicture}
\end{minipage}
\\ \\
\begin{minipage}{0,11\textwidth}
\begin{tikzpicture}[scale= 0.8]
\gri{2}{1} \cro{1}{1} \hor{2}{1}
\end{tikzpicture}
\end{minipage}%
\begin{minipage}{0,11\textwidth}
\begin{tikzpicture}[scale= 0.8]
\gri{2}{1} \cro{1}{1}
\end{tikzpicture}
\end{minipage}%
\begin{minipage}{0,11\textwidth}
\begin{tikzpicture}[scale= 0.8]
\gri{2}{1} \tbra{1}{1} \hor{2}{1}
\end{tikzpicture}
\end{minipage}%
\begin{minipage}{0,11\textwidth}
\begin{tikzpicture}[scale= 0.8]
\gri{2}{1} \tbra{1}{1}
\end{tikzpicture}
\end{minipage}%
\begin{minipage}{0,11\textwidth}
\begin{tikzpicture}[scale= 0.8]
\gri{2}{1} \tcobra{1}{1} \cro{2}{1}
\end{tikzpicture}
\end{minipage}%
\begin{minipage}{0,11\textwidth}
\begin{tikzpicture}[scale= 0.8]
\gri{2}{1} \tcobra{1}{1} \tbra{2}{1}
\end{tikzpicture}
\end{minipage}%
\begin{minipage}{0,11\textwidth}
\begin{tikzpicture}[scale= 0.8]
\gri{2}{1} \tcobra{1}{1} \ver{2}{1}
\end{tikzpicture}
\end{minipage}%
\begin{minipage}{0,11\textwidth}
\begin{tikzpicture}[scale= 0.8]
\gri{2}{1} \ver{1}{1} \hor{2}{1}
\end{tikzpicture}
\end{minipage}%
\begin{minipage}{0,11\textwidth}
\begin{tikzpicture}[scale= 0.8]
\gri{2}{1} \ver{1}{1}
\end{tikzpicture}
\end{minipage}
\\ \\
\begin{minipage}{0,11\textwidth}
\begin{tikzpicture}[scale= 0.8]
\gri{2}{1} \hor{1}{1} \cro{2}{1}
\end{tikzpicture}
\end{minipage}%
\begin{minipage}{0,11\textwidth}
\begin{tikzpicture}[scale= 0.8]
\gri{2}{1} \hor{1}{1} \tbra{2}{1}
\end{tikzpicture}
\end{minipage}%
\begin{minipage}{0,11\textwidth}
\begin{tikzpicture}[scale= 0.8]
\gri{2}{1} \hor{1}{1} \tcobra{2}{1}
\end{tikzpicture}
\end{minipage}%
\begin{minipage}{0,11\textwidth}
\begin{tikzpicture}[scale= 0.8]
\gri{2}{1} \hor{1}{1} \ver{2}{1}
\end{tikzpicture}
\end{minipage}%
\begin{minipage}{0,11\textwidth}
\begin{tikzpicture}[scale= 0.8]
\gri{2}{1} \cro{2}{1}
\end{tikzpicture}
\end{minipage}%
\begin{minipage}{0,11\textwidth}
\begin{tikzpicture}[scale= 0.8]
\gri{2}{1} \tbra{2}{1}
\end{tikzpicture}
\end{minipage}%
\begin{minipage}{0,11\textwidth}
\begin{tikzpicture}[scale= 0.8]
\gri{2}{1} \tcobra{2}{1}
\end{tikzpicture}
\end{minipage}%
\begin{minipage}{0,11\textwidth}
\begin{tikzpicture}[scale= 0.8]
\gri{2}{1} \ver{2}{1}
\end{tikzpicture}
\end{minipage}%
\begin{minipage}{0,11\textwidth}
\begin{tikzpicture}[scale= 0.8]
\gri{2}{1} \tcobra{1}{1} \tcobra{2}{1}
\end{tikzpicture}
\end{minipage} \\
\end{itemize}
\end{definition}
Roughly speaking, the third condition avoids the existence of Drinfeld--Yetter mosaic in which there is \emph{discontinuity} between the red lines. We set by convention
\begin{equation*}
\begin{tikzpicture}[scale=0.6]
\node at (-2,-0.5) {$\mosaic{0}{m}=$};
\node at (-0.5,-0.5) {\bigg\{};
\node at (5.5,-0.5) {\bigg\}};
\gri{1}{5} \ver{1}{1} \ver{1}{2} \ver{1}{4} \ver{1}{5} \node at (2.5,-0.5) {$ \dots $};
\begin{scope}[shift={(-1,0)}]
\node at (10,-0.5) {$\mosaic{n}{0}=$\bigg\{ };
\begin{scope}[shift={(-1.5,0)}]
\draw[step=1cm,black,very thick] (13,2) grid (14,-3);
\node at (14.5,-0.5) {\bigg\}};
\draw[red, very thick] (13,1.5)--(14,1.5);
\draw[red, very thick] (13,0.5)--(14,0.5);
\node at (13.5,-0.3) {$\vdots$};
\draw[red, very thick] (13,-1.5)--(14,-1.5);
\draw[red, very thick] (13,-2.5)--(14,-2.5);
\end{scope}
\end{scope}
\node at (15.5,-0.5){$\mosaic{0}{0} =$ \bigg\{ };
\begin{scope}[shift={(17,0)}]
\gri{1}{1}
\end{scope}
\node at (18.5,-0.5) {\bigg\}.};
\end{tikzpicture}
\end{equation*}
We shall respectively call the tiles of \eqref{eq:tiles-tableaux} the permutation, bracket, cobracket, action, coaction and empty tile.
\begin{notation}
We shall need the following auxiliary functions counting the number of bracket and cobracket tiles in a Drinfeld--Yetter mosaic: 
\begin{subequations}
\begin{align}
\cob : \mosaic{n}{m} &\to \mathbb{Z}_{\geqslant 0} \label{eq:aux-funct-one}\\
\bra : \mosaic{n}{m} &\to \mathbb{Z}_{\geqslant 0} \label{eq:aux-funct-two}
\end{align}
\end{subequations}
defined by 
\[ \cob(M) = \Big| \Big\{M_{i,j} = \begin{tikzpicture}[scale=0.35]
\gri{1}{1} \tcobra{1}{1}
\end{tikzpicture} \Big\} \Big| \]
and
\[ \bra(M) = \Big| \Big\{M_{i,j} = \begin{tikzpicture}[scale=0.35]
\gri{1}{1} \tbra{1}{1}
\end{tikzpicture} \Big\} \Big|. \]
\end{notation}
\begin{remark}
Let $n,m \geqslant 1$. It follows by Definition \ref{definition-mosaics} that the rows and the columns of an element of $\mosaic{n}{m}$ are respectively of the  form \\ 
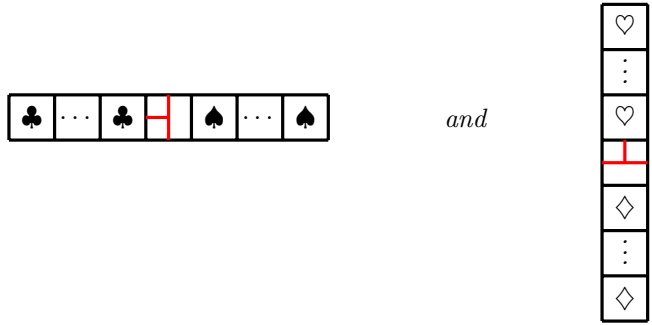
\begin{center}
\begin{tikzpicture}[scale=0.6]
\gri{1}{7} \tbra{1}{4} 
\node at (0.5,-0.5) {$ \clubsuit $};
\node at (1.5,-0.5) {$ \dots $};
\node at (2.5,-0.5) {$ \clubsuit $};
\node at (5.5,-0.5) {$ \dots $};
\node at (4.5,-0.5) {$ \spadesuit $};
\node at (6.5,-0.5) {$ \spadesuit $};
\node at (10,-0.5) {and};
\draw[step=1cm,black,very thick] (13,2) grid (14,-5);
\node at (13.5,1.5) {$ \heartsuit $};
\node at (13.5,0.7) {$\vdots$};
\node at (13.5,-0.5) {$ \heartsuit $};
\draw[red, very thick] (13,-1.5)--(14,-1.5);
\draw[red, very thick] (13.5,-1.5)--(13.5,-1);
\node at (13.5,-4.5) {$ \diamondsuit $};
\node at (13.5,-3.3) {$\vdots$};
\node at (13.5,-2.5) {$ \diamondsuit $};
\end{tikzpicture}
\end{center}
where 
\begin{center}
\begin{tikzpicture}[scale=0.6]
\draw[black,very thick] (0,0) rectangle (1,-1);
\node at (0.5,-0.5) {$\clubsuit$};
\node at (1.5,-0.5) {$\in$};
\node at (2,-0.5) {\Big\{};
\draw[black,very thick] (2.5,0) rectangle (3.5,-1);
\draw[red,very thick] (3,0)--(3,-1);
\draw[red,very thick] (2.5,-0.5)--(3.5,-0.5);
\node at (3.75,-1) {,};
\draw[black,very thick] (4,0) rectangle (5,-1);
\draw[red,very thick] (4.5,0)--(4.5,-0.5);
\draw[red,very thick] (4,-0.5)--(5,-0.5);
\node at (5.25,-1) {,};
\draw[black,very thick] (5.5,0) rectangle (6.5,-1);
\draw[red,very thick] (5.5,-0.5)--(6.5,-0.5);
\node at (7,-0.5) {\Big\}};
\node at (7.25,-1) {,};
\draw[black,very thick] (8,0) rectangle (9,-1);
\node at (8.5,-0.5) {$\spadesuit$};
\node at (9.5,-0.5) {$\in$};
\node at (10,-0.5) {\Big\{};
\draw[black,very thick] (10.5,0) rectangle (11.5,-1);
\draw[red,very thick] (11,0)--(11,-1);
\node at (11.75,-1) {,};
\draw[black,very thick] (12,0) rectangle (13,-1);
\node at (13.5,-0.5) {\Big\}};
\node at (13.75,-1) {,};
\end{tikzpicture}
\begin{tikzpicture}[scale=0.6]
\draw[black,very thick] (14.5,0) rectangle (15.5,-1);
\node at (0.5+14.5,-0.5) {$\heartsuit$};
\node at (1.5+14.5,-0.5) {$\in$};
\node at (2+14.5,-0.5) {\Big\{};
\draw[black,very thick] (2.5+14.5,0) rectangle (3.5+14.5,-1);
\draw[red,very thick] (17.5,0)--(17.5,-1);
\draw[red,very thick] (17,-0.5)--(18,-0.5);
\node at (3.75+14.5,-1) {,};
\draw[black,very thick] (4+14.5,0) rectangle (5+14.5,-1);
\draw[red,very thick] (19,0)--(19,-1);
\draw[red,very thick] (18.5,-0.5)--(19,-0.5);
\node at (5.25+14.5,-1) {,};
\draw[black,very thick] (5.5+14.5,0) rectangle (6.5+14.5,-1);
\draw[red,very thick] (20.5,0)--(20.5,-1);
\node at (7+14.5,-0.5) {\Big\}};
\node at (7.25+14.5,-1) {,};
\draw[black,very thick] (8+14.5,0) rectangle (9+14.5,-1);
\node at (8.5+14.5,-0.5) {$\diamondsuit$};
\node at (9.5+14.5,-0.5) {$\in$};
\node at (10+14.5,-0.5) {\Big\{};
\draw[black,very thick] (10.5+14.5,0) rectangle (11.5+14.5,-1);
\draw[red,very thick] (25,-0.5)--(26,-0.5);
\node at (11.75+14.5,-1) {,};
\draw[black,very thick] (12+14.5,0) rectangle (13+14.5,-1);
\node at (13.5+14.5,-0.5) {\Big\}};
\node at (13.75+14.5,-1) {.};
\end{tikzpicture}
\end{center}
We therefore deduce the following facts:
\begin{itemize}
\item Every row of a Drinfeld--Yetter mosaic has at most one bracket tile. 
\item Every column of a Drinfeld--Yetter mosaic has at most one cobracket tile. 
\item The set $\mosaic{n}{m}$ splits into disjoint union of three subsets 
\begin{equation}
\label{eq:disjoint-union-drinfeld-yetter-tableaux}
\mosaic{n}{m} = \mosaic{n}{m}^{\begin{tikzpicture}[scale=0.3] \gri{1}{1} \tbra{1}{1} \end{tikzpicture} } \sqcup \mosaic{n}{m}^{\begin{tikzpicture}[scale=0.3] \gri{1}{1} \tcobra{1}{1} \end{tikzpicture} } \sqcup \mosaic{n}{m}^{\begin{tikzpicture}[scale=0.3] \gri{1}{1} \cro{1}{1} \end{tikzpicture} } 
\end{equation}
where 
\[\mosaic{n}{m}^{\begin{tikzpicture}[scale=0.3]
\gri{1}{1} \node at (0.5,-0.5) {$ \ast$};
\end{tikzpicture}} \coloneqq \Big\{ M \in \mosaic{n}{m} \ , \ M_{1,1} = \begin{tikzpicture}[scale=0.35]
\gri{1}{1} \node at (0.5,-0.5) {$ \ast$}; \end{tikzpicture}\Big\} .\]
\end{itemize} 
\end{remark}
Relying on these facts we obtain the following 
\begin{proposition}
\label{recurrence-rule-dy-tableaux}
Let $\mosaic{1}{m}^{(\ell)}$ be the set of all $1 \times m$ Drinfeld--Yetter mosaics satisfying $\cob(M) = m- \ell$ and $\mosaic{n}{1}^{[k]}$ be the set of all $n \times 1$ Drinfeld--Yetter mosaics satisfying $\bra(M) = n-k$. Then 
\begin{equation}
\label{eq:drinfeld-yetter-tableaux-recursive-one}
\mosaic{n}{m}= \bigsqcup_{\ell=0}^m \mosaic{1}{m}^{(\ell)} \times \mosaic{n-1}{\ell}
\end{equation}  
and
\begin{equation}
\label{eq:drinfeld-yetter-tableaux-recursive-two}
\mosaic{n}{m} = \bigsqcup_{k=0}^n \mosaic{n}{1}^{[k]} \times \mosaic{k}{m-1}.
\end{equation}
\end{proposition}
\begin{proof}
We show \eqref{eq:drinfeld-yetter-tableaux-recursive-one}, the proof of \eqref{eq:drinfeld-yetter-tableaux-recursive-two} is analogous. In order to construct a Drinfeld--Yetter mosaic $M \in \mosaic{n}{m}$, one can freely assign the first row of $M$, which is an element $y$ of $\mosaic{1}{m}$ having the following form
\begin{center}
\begin{tikzpicture}[scale=0.6]
\node at (-1,-0.75) {$y=$};
\gri{1}{7} \ver{1}{7} \ver{1}{5} \tbra{1}{4} 
\node at (0.5,-0.5) {$ \clubsuit $};
\node at (1.5,-0.5) {$ \dots $};
\node at (2.5,-0.5) {$ \clubsuit $};
\node at (5.5,-0.5) {$ \dots $};
\draw[black,very thick] (10,0) rectangle (11,-1);
\node at (10.5,-0.5) {$\clubsuit$};
\node at (11.5,-0.5) {$\in$};
\node at (12,-0.5) {\Big\{};
\draw[black,very thick] (12.5,0) rectangle (13.5,-1);
\draw[red,very thick] (12.5,-0.5)--(13.5,-0.5);
\draw[red,very thick] (13,0)--(13,-1);
\node at (13.75,-1) {,};
\draw[black,very thick] (14,0) rectangle (15,-1);
\draw[red,very thick] (14,-0.5)--(15,-0.5);
\draw[red,very thick] (14.5,0)--(14.5,-0.5);
\node at (15.5,-0.5) {\Big\}.};
\end{tikzpicture}
\end{center}
For any cobracket tile in $y$, we have that the tiles below are automatically determined. More precisely, according to the defining rules of the Drinfeld--Yetter mosaics, if $M_{1,j}$ is a cobracket tile, then 
\[
M_{i,j}=
\begin{cases}
\begin{tikzpicture}[scale=0.4] \gri{1}{1}  \end{tikzpicture}  \text{ if }  M_{i,1} = \begin{tikzpicture}[scale=0.4] \gri{1}{1} \tbra{1}{1} \end{tikzpicture}  \\
\begin{tikzpicture}[scale=0.4] \gri{1}{1} \hor{1}{1} \end{tikzpicture} \text{ otherwise}  
\end{cases}.
\]
Therefore, if $\cob(y)=m-\ell$, we have $m-\ell$ columns of $M$ automatically determined. The other $\ell$ columns can be freely chosen among all the elements of $\mosaic{n-1}{\ell}$.
\end{proof}
As a direct consequence of Proposition \ref{recurrence-rule-dy-tableaux} we get the following two recursive formulas for $|\mosaic{n}{m}|$:
\begin{corollary}
\label{recursive-cardinality-mosaics}
One has 
\begin{equation}
\label{eq:cardinality-drinfeld-yetter-tableaux-recursive-one}
|\mosaic{n}{m}| = \sum_{\ell=0}^m \binom{m+1}{\ell} |\mosaic{n-1}{\ell}| 
\end{equation}
and 
\begin{equation}
\label{eq:cardinality-drinfeld-yetter-tableaux-recursive-two}
|\mosaic{n}{m}|  = \sum_{k=0}^{n} \binom{n+1}{k} |\mosaic{k}{m-1}|.
\end{equation}
\end{corollary}
\begin{proof}
We show \eqref{eq:cardinality-drinfeld-yetter-tableaux-recursive-one}, the proof of \eqref{eq:cardinality-drinfeld-yetter-tableaux-recursive-two} is analogous. Let $M \in \mosaic{1}{m}$ and let $t$ be the number of permutation and cobracket tiles of $M$. Then we have 
\[
|\mosaic{1}{m}^{(\ell)}| = \sum_{t=0}^{m} \binom{t}{m-\ell}  = \sum_{t=m-\ell}^m \binom{t}{m-\ell} = \binom{m+1}{m-\ell+1} = \binom{m+1}{\ell} 
\]
where the third equality follows by the well--known Hockey--Stick identity. 
\end{proof}
Next, we provide a closed formula for $\FM{n}{m}\coloneqq|\mosaic{n}{m}|$. It is easy to see that $\FM{n}{m}$ is symmetric
(i.e. $\FM{n}{m}=\FM{m}{n}$). By iterating equations \eqref{eq:cardinality-drinfeld-yetter-tableaux-recursive-one} and \eqref{eq:cardinality-drinfeld-yetter-tableaux-recursive-two}, we obtain
\begin{align}
	\label{eq:F-closed} \FM{n}{m}=\sum_{\underline{k}\in I_{n,m}}^m\prod_{i=0}^n\binom{k_i+1}{k_{i+1}}
\end{align}
where $I_{n,m}\coloneqq\{(k_0,k_1,\dots, k_n)\;|\; k_0\coloneqq m,\; k_i\leqslant k_{i-1}, i=1,\dots, n\}$.\\
Recall that the \index{Stirling numbers of the second kind}Stirling numbers of the second kind are the non--negative integers $\stirling{n}{k}$ counting the number of ways to partition a set of $n$ labelled objects into $k$ non--empty unlabelled subsets, and they satisfy the recursive relation
\begin{align}
\label{eq:stirling}
\stirling{n+1}{k+1}=(k+1)\cdot\stirling{n}{k+1}+\stirling{n}{k}	
\end{align}
with initial conditions
\[ 	\stirling{0}{0} = 1 \quad \text{and} \quad \stirling{n}{0}=\stirling{0}{k} = 0.\]
One has
\[
	\stirling{n}{k}=\frac{1}{k!}\sum_{i=0}^k(-1)^i\cdot\binom{k}{i}\cdot(k-i)^n.
\]

We shall provide a concise expression of $F_{n,m}$ in terms of Stirling numbers of the second kind.
To this end, we shall use the following 

\begin{lemma}
For any $0\leqslant k\leqslant n+1$, we have
\begin{align}
\label{eq:stirling-formula}
k\cdot\stirling{n+1}{k}=\sum_{\ell=k-1}^n(-1)^{n-\ell}\cdot\binom{n+1}{\ell}\cdot\stirling{\ell+1}{k}	.
\end{align}	
\end{lemma}

\begin{proof}
We proceed by induction on $n$. We observe that the cases
$k=0, n+1$ are trivial, while for $k=1$ it reduces to the identity
\[
	\sum_{\ell=0}^{n+1}(-1)^\ell\binom{n+1}{\ell}=0.
\]
Therefore, the case $n=1$ is clear.
Assume the result holds for $n-1\geqslant0$. Then,
\begin{equation*}
\begin{split}
&\sum_{\ell=k-1}^n(-1)^{n-\ell}\binom{n+1}{\ell}\stirling{\ell+1}{k}=(n+1)\stirling{n+1}{k}+\sum_{\ell=k-1}^{n-1}(-1)^{n-\ell}\binom{n+1}{\ell}\stirling{\ell+1}{k}\\
&=(n+1)\stirling{n+1}{k}-\sum_{\ell=k-1}^{n-1}(-1)^{n-1-\ell}\binom{n}{\ell}\stirling{\ell+1}{k}-\sum_{\ell=k-1}^{n-1}(-1)^{n-1-\ell}\binom{n}{\ell-1}\stirling{\ell+1}{k}\\
&=(n+1)\stirling{n+1}{k}-k\stirling{n}{k}
-k\sum_{\ell=k-1}^{n-1}(-1)^{n-1-\ell}\binom{n}{\ell-1}\stirling{\ell}{k}-\sum_{\ell=k-1}^{n-1}(-1)^{n-1-\ell}\binom{n}{\ell-1}\stirling{\ell}{k-1}\\
&=(n+1)\stirling{n+1}{k}-k\stirling{n}{k}
+k\sum_{j=k-1}^{n-2}(-1)^{n-1-j}\binom{n}{j}\stirling{j+1}{k}
+\sum_{j=k-2}^{n-2}(-1)^{n-1-j}\binom{n}{j}\stirling{j+1}{k-1}\\
&=(n+1)\stirling{n+1}{k}-k\stirling{n}{k}+k^2\stirling{n}{k}-k n\stirling{n}{k}+(k-1)\stirling{n}{k-1}-n\stirling{n}{k-1}\\
&=(n+1)\stirling{n+1}{k}-(n-k+1)\left(k\cdot\stirling{n}{k}+\stirling{n}{k-1}\right)\\
&=k\stirling{n+1}{k}
\end{split} 
\end{equation*}
where the second equality follows from the recursive identity for the binomial coefficient,
the third one follows by induction and the recursive identity for the Stirling numbers of the second kind
\eqref{eq:stirling}, the fourth and fifth\footnote{Note that $\stirling{k-1}{k}=0$, 
thus we can assume the first sum starts at $j=k-1$.} one by induction, and the last one by \eqref{eq:stirling}.
\end{proof}

Relying on the above result, we get the following

\begin{proposition}
For any $n,m\geqslant 0$, we have
\begin{align}
	\label{eq:F-stirling}
	\FM{n}{m}=\sum_{k=1}^{n+1}(-1)^{n-k+1}\cdot k!\cdot k^{m}\cdot \stirling{n+1}{k}.
\end{align}
\end{proposition}

\begin{proof}
Set $B_{n,k}\coloneqq(-1)^{n-k}\cdot (k-1)!\cdot\stirling{n}{k}$.
Then, the formula \eqref{eq:F-stirling} is equivalent to
\[
	\FM{n}{m}=\sum_{k=1}^{n+1}B_{n+1,k}\cdot k^{m+1}.
\]
Therefore, it is enough to prove that the numbers
\[
	G_{n,m}\coloneqq\sum_{k=1}^{n+1}B_{n+1,k}\cdot k^{m+1}	
\]
satisfy the recursive relation \eqref{eq:cardinality-drinfeld-yetter-tableaux-recursive-one}. Note that
\eqref{eq:stirling-formula} is equivalent to 
\begin{align}
\label{eq:B-formula}
k\cdot B_{n+1,k}=\sum_{\ell=k-1}^{n}\binom{n+1}{\ell}\cdot B_{\ell+1,k}.	
\end{align}
Therefore,
\begin{equation*}
\begin{split}
	\sum_{\ell=0}^{n}\binom{n+1}{\ell}G_{\ell,m-1}&=
	\sum_{\ell=0}^{n}\binom{n+1}{\ell}\sum_{k=1}^{\ell+1}B_{\ell+1,k}\cdot k^{m}\\
	&=\sum_{\ell=0}^{n}\sum_{k=1}^{\ell+1}\binom{n+1}{\ell}B_{\ell+1,k}\cdot k^{m}\\
	&=\sum_{k=1}^{n+1}\left(\sum_{\ell=k-1}^n\binom{n+1}{\ell}B_{\ell+1,k}\right)\cdot k^m\\
	&=\sum_{k=1}^{n+1}B_{n+1,k}\cdot k^{m+1}=G_{n,m}
\end{split}
\end{equation*}
where the fourth identity follows from \eqref{eq:B-formula}.
Thus, $G_{n,m}=F_{n,m}$ since they satisfy the same recursion relation.	
\end{proof}
\subsection{Drinfeld--Yetter looms}
\label{section-d-y-looms}
In this Section we define the set $\loom{n}{m}$ of $n \times m$ Drinfeld--Yetter looms, which will give a combinatorial description of the application of steps $(2)$--$(3)$ of the algorithm described in \S \ref{subsection-dy-algebra}. To any Drinfeld--Yetter mosaic $M$ we shall associate a morphism pictorially represented by 
\begin{equation}
\label{eq:morphism-mosaic-permutations}
\begin{tikzpicture}
\draw[-, very thick,ggreen] (0,0)--(2,0);
\draw[-, very thick,ggreen] (3.8+1,0)--(5.9+1,0);
\draw[very thick]  (1,0.75) arc (90:180:0.75);
\draw[very thick] (1,0.75)--(1.2,0.75);
\node[shape=circle,draw,inner sep=1pt] (char) at (1.4,0.75) {$\sigma$}; 
\draw[very thick] (1.6,0.75)--(2,0.75);
\node at (0.3,-0.3) {$n$};
\begin{scope}[shift={(0.1,0)}]
\draw[very thick] (3.8+1,0.75)--(4.2+1,0.75);
\node[shape=circle,draw,inner sep=1pt] (char) at (4.4+1,0.75) {$\tau$}; 
\draw[very thick] (4.6+1,0.75)--(4.8+1,0.75);
\draw[very thick]  (5.55+1,0) arc (0:90:0.75);
\node at (5.55+1,-0.3) {$m$};
\node at (3.6+1,0.75) {$m$};
\end{scope}
\node at (2.2,0.75) {$n$};
\node at (3.4,0.4) {\Bigg($\varphi_{n,m}(M)$ \Bigg) };
\end{tikzpicture}
\end{equation}
where $\varphi_{n,m}(M) \in \Hom_{\DY^1}([V] \ten [n], [V] \ten [m])$ is given by Equation \eqref{eq:morphism-from-mosaic-to-hom}. We  then describe the process of removing Lie brackets and Lie cobrackets from the morphism \eqref{eq:morphism-mosaic-permutations} --by applying formulas \eqref{eq:actionrule} and \eqref{eq:coactionrule}-- through a set $\mathfrak{L}(M) \subset \loom{n}{m}$ (see Equation \eqref{eq:dy-looms-associated-to-a-mosaic}). We finally show that the collection $\{\mathfrak{L}(M)\}_{M \in \mosaic{n}{m}} $ is a partition of $\loom{n}{m}$, hence giving a description of the product of $\mathfrak{U}_\DY^1$ in terms of Drinfeld--Yetter looms.
\begin{definition}
\label{definition-refined-Drinfeld-Yetter-Tableaux}
Let $n,m \geqslant 1$ and let $\T_{\loom{n}{m}}$ be the following set of tiles
\begin{center}
\begin{tikzpicture}[scale=0.6]
\node at (-0.15,0) {$\mathcal{T}_{\loom{n}{m}}= \Big\{$};
\draw[black,very thick] (1.5,0.5) rectangle (2.5,-0.5);
\lcro{0.5}{2.5}
\node at (2.75,-0.5) {,};
\draw[black,very thick] (3,0.5) rectangle (4,-0.5);
\lmua{0.5}{4}
\node at (4.25,-0.5) {,};
\draw[black,very thick] (4.5,0.5) rectangle (5.5,-0.5);
\lmub{0.5}{5.5}
\node at (5.75,-0.5) {,};
\draw[black,very thick] (6,0.5) rectangle (7,-0.5);
\dela{0.5}{7}
\node at (7.25,-0.5) {,};
\draw[black,very thick] (7.5,0.5) rectangle (8.5,-0.5);
\delb{0.5}{8.5}
\node at (8.75,-0.5) {,};
\draw[black,very thick] (9,0.5) rectangle (10,-0.5);
\lhor{0.5}{10}
\node at (10.25,-0.5) {,};
\draw[black,very thick] (10.5,0.5) rectangle (11.5,-0.5);
\lver{0.5}{11.5}
\node at (11.75,-0.5) {,};
\draw[black,very thick] (12,0.5) rectangle (13,-0.5);
\node at (13.5,-0) {\Big\}};
\end{tikzpicture}
\end{center}
where the yellow line denotes $k \in \{0,\ldots,m-1 \}$ red horizontal lines and the blue line denotes $\ell \in \{0,\ldots, n-1 \}$ red vertical lines.  To any tile $\mathcal{T}$ of $\T_{\loom{n}{m}}$, we associate a tuple of integers $(t,b,l,r)$ that respectively indicates the number of strings occurring on the top, bottom, left and right edge of $\mathcal{T}$.  We define the set of $n \times m$ Drinfeld--Yetter looms $\loom{n}{m}$ as the set of all possible tilings $L$ of $\grid{n}{m}$ with the elements of $\T_{\loom{n}{m}}$ such that the following five conditions are satisfied:
\begin{itemize}
\item[(1):] $b_{i,j} = t_{i+1,j}$ for all $i=1,\ldots,m-1$ and $j=1,\ldots,n;$
\item[(2):] $l_{i,j}= r_{i,j+1}$ for all $i=1,\ldots,m$ and  $j=1,\ldots,n-1;$
\item[(3):] $\sum_{i=1}^n l_{i,1} + \sum_{j=1}^m b_{n,j}  = \sum_{i=1}^n r_{i,m} + \sum_{j=1}^m t_{1,j} = n+m;$
\item[(4):]$R_{1,j} \notin \Big\{ \begin{tikzpicture}[scale= 0.4] \gri{1}{1} \lhor{1}{1}

\end{tikzpicture} \ , \ \begin{tikzpicture}[scale= 0.4] \gri{1}{1}

\end{tikzpicture} \Big\}$ for all $j \in \{1, \ldots , m\};$
\item[(5):] $R_{i,1} \notin \Big\{ \begin{tikzpicture}[scale= 0.4] \gri{1}{1} \lver{1}{1}

\end{tikzpicture} \ , \ \begin{tikzpicture}[scale= 0.4] \gri{1}{1}

\end{tikzpicture} \Big\}$ for all $i \in \{1, \ldots , n\},$
\end{itemize}
where $(t_{i,j},b_{i,j},l_{i,j},r_{i,j})$ denotes the tuple of weights of $L_{i,j}$. We set by convention $\loom{0}{m} = \mosaic{0}{m}$, $\loom{n}{0}=\mosaic{n}{0}$ and $\loom{0}{0}=\mosaic{0}{0}$.
\end{definition}
As in the case of the Drinfeld--Yetter mosaics, the first two conditions avoid the existence of Drinfeld--Yetter looms in which there is discontinuity between the red lines. The fourth and fifth conditions are the analogous of the first two conditions of Definition \ref{definition-mosaics}. 
Note that, as opposed to the case of Drinfeld--Yetter mosaics, the set of tiles of $\loom{n}{m}$ depends on $n$ and $m$ (more precisely, its cardinality is $ nm + 3m+3n$). \\
For any $L \in \loom{n}{m}$, we shall interpret the left and bottom (resp. right and top) strings as ingoing (resp. ongoing) strings. By definition, the path of any ingoing string across $L$ ends in a outgoing string. We can therefore associate to $L$ a permutation in $\mathfrak{S}_{n+m}$ in the following way: we assign to any ingoing string a number in $\{1,\ldots,n+m\}$, starting from the top left to the bottom left and carrying on from the bottom left to the bottom right. In the same way, we assign to any outgoing string a number in $\{1,\ldots,n+m\}$, starting from the top left to the top right and carrying on from the top right to the bottom right.
\begin{example}
\label{example-permutation-ass-to-a-loom}
The permutation associated to the following Drinfeld--Yetter loom
\begin{center}
\begin{tikzpicture}[scale=1.5]
\gri{2}{2} \cro{1}{1} \cro{2}{2} \ver{1}{2} \hor{2}{1}
\draw[red, very thick] (2-1,-1+0.5) arc (-75:0:0.5); 
\draw[red, very thick] (1-1,-2+0.25) arc (-90:0:0.5); 
\draw[red, very thick] (0.5,0)--(0.5,-1.25);
\node at (-0.2,-0.5) {$1$};
\node at (-0.2,-1.5) {$2$};
\node at (-0.2,-1.75) {$3$};
\node at (1.5,-2.2) {$4$};
\node at (0.5,0.2) {$1$};
\node at (1.35,0.2) {$2$};
\node at (1.5,0.2) {$3$};
\node at (2.2,-1.5) {$4$};
\end{tikzpicture}
\end{center}
is $(1243) \in \mathfrak{S}_{4}$. 
\end{example}
The procedure described above defines a collection of maps $\gamma_{n,m} : \loom{n}{m} \to \mathfrak{S}_{n+m}$. Note that the maps $\gamma_{n,m}$ are in general not injective nor surjective. For example, for the following $L_1,L_2 \in \loom{2}{2}$
\begin{center}
\begin{tikzpicture}[scale=1.5]
\gri{2}{2} \cro{1}{1} \cro{2}{2} \ver{1}{2} \hor{2}{1}
\draw[red, very thick] (2-1,-1+0.5) arc (-75:0:0.5); 
\draw[red, very thick] (1-1,-2+0.25) arc (-90:0:0.5); 
\draw[red, very thick] (0.5,0)--(0.5,-1.25);
\node at (-0.2,-0.5) {$1$};
\node at (-0.2,-1.5) {$2$};
\node at (-0.2,-1.75) {$3$};
\node at (1.5,-2.2) {$4$};
\node at (0.5,0.2) {$1$};
\node at (1.35,0.2) {$2$};
\node at (1.5,0.2) {$3$};
\node at (2.2,-1.5) {$4$};
\end{tikzpicture}
\hspace{2cm}
\begin{tikzpicture}[scale=1.5]
\gri{2}{2} \cro{1}{1} \cro{2}{2} \ver{1}{2} \cro{2}{1} 
\draw[red, very thick] (2-1,-1+0.5) arc (-75:0:0.5); 
\draw[red, very thick] (0.5,0)--(0.5,-1.25);
\node at (-0.2,-0.5) {$1$};
\node at (-0.2,-1.5) {$2$};
\node at (0.5,-2.2) {$3$};
\node at (1.5,-2.2) {$4$};
\node at (0.5,0.2) {$1$};
\node at (1.35,0.2) {$2$};
\node at (1.5,0.2) {$3$};
\node at (2.2,-1.5) {$4$};
\end{tikzpicture}
\end{center}
we have $\gamma_{2,2}(L_1) = \gamma_{2,2}(L_2)$. On the other hand, one can show through a direct inspection that there is no Drinfeld--Yetter loom $L \in \loom{1}{3}$ such that $\gamma_{1,3}(L) = (12)(34) \in \mathfrak{S}_4$. \\ Next, let $M \in \mosaic{n}{m}$ with tiles $\{M_{1,1},\ldots, M_{n,m} \}$ and consider the following map
\begin{center}
\begin{tikzpicture}[scale=0.6]
\node at (5,0) {$f_M : \{ M_{1,1},\ldots, M_{n,m} \} \to \P(\mathcal{T}_{\loom{n}{m}})$} ;

\draw[black,very thick] (5.25,-0.75) rectangle (6.25,-1.75);
\node at (7,-1.25) {$\mapsto$}; \node at (7.75,-1.25) {\Big\{};
\draw[black,very thick] (8.25,-0.75) rectangle (9.25,-1.75);
\node at (9.75, -1.25) {$\Big\}$};
\cro{1.75}{6.25} \lcro{1.75}{9.25}

\draw[black,very thick] (5.25,-0.75-1.5) rectangle (6.25,-1.75-1.5);
\node at (7,-1.25-1.5) {$\mapsto$}; \node at (7.75,-1.25-1.5) {\Big\{};
\draw[black,very thick] (8.25,-0.75-1.5) rectangle (9.25,-1.75-1.5);
\node at (9.5,-1.75-1.5) {,};
\draw[black,very thick] (9.75,-0.75-1.5) rectangle (10.75,-1.75-1.5);
\node at (11.25, -1.25-1.5) {$\Big\}$};
\tbra{1.75-1.5}{6.25} \lmua{1.75-1.5}{9.25} \lmub{1.75-1.5}{10.75}

\draw[black,very thick] (5.25,-0.75-1.5-1.5) rectangle (6.25,-1.75-1.5-1.5);
\node at (7,-1.25-1.5-1.5) {$\mapsto$}; \node at (7.75,-1.25-1.5-1.5) {\Big\{};
\draw[black,very thick] (8.25,-0.75-1.5-1.5) rectangle (9.25,-1.75-1.5-1.5);
\node at (9.5,-1.75-1.5-1.5) {,};
\draw[black,very thick] (9.75,-0.75-1.5-1.5) rectangle (10.75,-1.75-1.5-1.5);
\node at (11.25, -1.25-1.5-1.5) {$\Big\}$};
\tcobra{1.75-1.5-1.5}{6.25} \dela{1.75-1.5-1.5}{9.25} \delb{1.75-1.5-1.5}{10.75}

\draw[black,very thick] (5.25,-0.75-1.5-1.5-1.5) rectangle (6.25,-1.75-1.5-1.5-1.5);
\node at (7,-1.25-1.5-1.5-1.5) {$\mapsto$}; \node at (7.75,-1.25-1.5-1.5-1.5) {\Big\{};
\draw[black,very thick] (8.25,-0.75-1.5-1.5-1.5) rectangle (9.25,-1.75-1.5-1.5-1.5);
\node at (9.75, -1.25-1.5-1.5-1.5) {$\Big\}$};
\hor{1.75-1.5-1.5-1.5}{6.25} \lhor{1.75-1.5-1.5-1.5}{9.25}

\draw[black,very thick] (5.25,-0.75-1.5-1.5-1.5-1.5) rectangle (6.25,-1.75-1.5-1.5-1.5-1.5);
\node at (7,-1.25-1.5-1.5-1.5-1.5) {$\mapsto$}; \node at (7.75,-1.25-1.5-1.5-1.5-1.5) {\Big\{};
\draw[black,very thick] (8.25,-0.75-1.5-1.5-1.5-1.5) rectangle (9.25,-1.75-1.5-1.5-1.5-1.5);
\node at (9.75, -1.25-1.5-1.5-1.5-1.5) {$\Big\}$};
\ver{1.75-1.5-1.5-1.5-1.5}{6.25} \lver{1.75-1.5-1.5-1.5-1.5}{9.25}

\draw[black,very thick] (5.25,-0.75-1.5-1.5-1.5-1.5-1.5) rectangle (6.25,-1.75-1.5-1.5-1.5-1.5-1.5);
\node at (7,-1.25-1.5-1.5-1.5-1.5-1.5) {$\mapsto$}; \node at (7.75,-1.25-1.5-1.5-1.5-1.5-1.5) {\Big\{};
\draw[black,very thick] (8.25,-0.75-1.5-1.5-1.5-1.5-1.5) rectangle (9.25,-1.75-1.5-1.5-1.5-1.5-1.5);
\node at (9.75, -1.25-1.5-1.5-1.5-1.5-1.5) {$\Big\}$};
\end{tikzpicture}
\end{center}
where the numbers $k,\ell$ of red lines corresponding to the yellow and the blue ones are 
\begin{equation*}
\begin{split}
k &= \big| \{M_{i,t} = \begin{tikzpicture}[scale=0.4]
\gri{1}{1} \tcobra{1}{1}
\end{tikzpicture}, t>j \} \big| \\
\ell&= \big| \{ M_{s,j} = \begin{tikzpicture}[scale=0.4]
\gri{1}{1} \tbra{1}{1}
\end{tikzpicture}, s>i \} \big|
\end{split}
\end{equation*}
\begin{definition}
The set $\mathfrak{L}(M)$ of all the \index{Drinfeld--Yetter looms related to a Drinfeld--Yetter mosaic}Drinfeld--Yetter looms related to $M$ is
\begin{equation}
\label{eq:dy-looms-associated-to-a-mosaic}
\mathfrak{L}(M) = \{L \in \loom{n}{m} \ | \  L_{i,j} \in f_M(M_{i,j}) \}.
\end{equation}
\end{definition}
\begin{proposition}
\label{proposition-partition-drinfeld-yetter-tableaux}
The collection $\{ \mathfrak{L}(M)\}_{M \in \mosaic{n}{m}}$ defines a partition of $\loom{n}{m}$.
\end{proposition}
\begin{proof}
It is clear that $\mathfrak{L}(M) $ is non--empty for any $M \in \mosaic{n}{m}$ and that $M_1 \neq M_2$ implies $\mathfrak{L}(M_1)  \cap \mathfrak{L}(M_1)  = \emptyset$. It remains to prove that the collection $\{ \mathfrak{L}(M)\}_{M \in \mosaic{n}{m}}$ defines a covering of $\loom{n}{m}$, i.e. that for any $L \in \loom{n}{m}$ there exists $M \in \mosaic{n}{m}$ such that $L \in \mathfrak{L}(M)$. For any $L \in \loom{n}{m}$, we construct such a $M$ in the following way: consider the following map
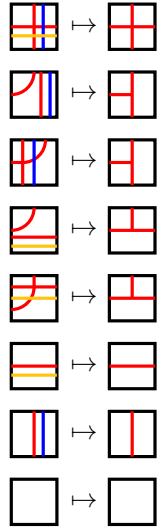
\begin{center}
\begin{tikzpicture}[scale=0.6]
\node at (3,0) {$\chi : \mathcal{T}_{\loom{n}{m}} \to \mathcal{T}_{\mathfrak{M}}$} ;
\draw[black,very thick] (2.25,-0.75) rectangle (3.25,-1.75);
\node at (3.85,-1.25) {$\mapsto$}; 
\draw[black,very thick] (4.4,-0.75) rectangle (5.4,-1.75);
\lcro{1.75}{3.25} \cro{1.75}{5.4}

\draw[black,very thick] (2.25,-0.75-1.5) rectangle (3.25,-1.75-1.5);
\node at (3.85,-1.25-1.5) {$\mapsto$}; 
\draw[black,very thick] (4.4,-0.75-1.5) rectangle (5.4,-1.75-1.5);
\lmua{1.75-1.5}{3.25} \tbra{1.75-1.5}{5.4}

\draw[black,very thick] (2.25,-0.75-1.5-1.5) rectangle (3.25,-1.75-1.5-1.5);
\node at (3.85,-1.25-1.5-1.5) {$\mapsto$}; 
\draw[black,very thick] (4.4,-0.75-1.5-1.5) rectangle (5.4,-1.75-1.5-1.5);
\lmub{1.75-1.5-1.5}{3.25} \tbra{1.75-1.5-1.5}{5.4}

\draw[black,very thick] (2.25,-0.75-1.5-1.5-1.5) rectangle (3.25,-1.75-1.5-1.5-1.5);
\node at (3.85,-1.25-1.5-1.5-1.5) {$\mapsto$}; 
\draw[black,very thick] (4.4,-0.75-1.5-1.5-1.5) rectangle (5.4,-1.75-1.5-1.5-1.5);
\dela{1.75-1.5-1.5-1.5}{3.25} \tcobra{1.75-1.5-1.5-1.5}{5.4}

\draw[black,very thick] (2.25,-0.75-1.5-1.5-1.5-1.5) rectangle (3.25,-1.75-1.5-1.5-1.5-1.5);
\node at (3.85,-1.25-1.5-1.5-1.5-1.5) {$\mapsto$}; 
\draw[black,very thick] (4.4,-0.75-1.5-1.5-1.5-1.5) rectangle (5.4,-1.75-1.5-1.5-1.5-1.5);
\delb{1.75-1.5-1.5-1.5-1.5}{3.25} \tcobra{1.75-1.5-1.5-1.5-1.5}{5.4}

\draw[black,very thick] (2.25,-0.75-1.5-1.5-1.5-1.5-1.5) rectangle (3.25,-1.75-1.5-1.5-1.5-1.5-1.5);
\node at (3.85,-1.25-1.5-1.5-1.5-1.5-1.5) {$\mapsto$}; 
\draw[black,very thick] (4.4,-0.75-1.5-1.5-1.5-1.5-1.5) rectangle (5.4,-1.75-1.5-1.5-1.5-1.5-1.5);
\lhor{1.75-1.5-1.5-1.5-1.5-1.5}{3.25} \hor{1.75-1.5-1.5-1.5-1.5-1.5}{5.4}

\draw[black,very thick] (2.25,-0.75-1.5-1.5-1.5-1.5-1.5-1.5) rectangle (3.25,-1.75-1.5-1.5-1.5-1.5-1.5-1.5);
\node at (3.85,-1.25-1.5-1.5-1.5-1.5-1.5-1.5) {$\mapsto$}; 
\draw[black,very thick] (4.4,-0.75-1.5-1.5-1.5-1.5-1.5-1.5) rectangle (5.4,-1.75-1.5-1.5-1.5-1.5-1.5-1.5);
\lver{1.75-1.5-1.5-1.5-1.5-1.5-1.5}{3.25} \ver{1.75-1.5-1.5-1.5-1.5-1.5-1.5}{5.4}

\draw[black,very thick] (2.25,-0.75-1.5-1.5-1.5-1.5-1.5-1.5-1.5) rectangle (3.25,-1.75-1.5-1.5-1.5-1.5-1.5-1.5-1.5);
\node at (3.85,-1.25-1.5-1.5-1.5-1.5-1.5-1.5-1.5) {$\mapsto$}; 
\draw[black,very thick] (4.4,-0.75-1.5-1.5-1.5-1.5-1.5-1.5-1.5) rectangle (5.4,-1.75-1.5-1.5-1.5-1.5-1.5-1.5-1.5);
\end{tikzpicture}
\end{center}
We define the Drinfeld--Yetter loom $L_M$ associated to $M$ as the one with $L_{i,j}= \chi(L_{i,j})$. It is easy to see that $L \in \mathfrak{L}(M_M)$, hence the claim is proved.
\end{proof}
\begin{remark}
Note that, recalling the auxiliary functions \eqref{eq:aux-funct-one} - \eqref{eq:aux-funct-two}, we have the following formula for the cardinality of $\mathfrak{L}(M)$:
\begin{equation}
\label{eq:cardinality-associated-refined-drinfeld-yetter-tableaux}
|\mathfrak{L}(M)| = 2^{\cob(M) + \bra(M)}.
\end{equation}
\end{remark}
\begin{notation}
As in the case of Drinfeld--Yetter mosaics, we shall need some auxiliary functions counting the amount of some tiles occurring in a Drinfeld--Yetter loom: we set
\begin{equation*}
\begin{split}
\countingone : \loom{n}{m} &\to \mathbb{Z}_{\geqslant 0} \\
\countingtwo : \loom{n}{m} &\to \mathbb{Z}_{\geqslant 0} 
\end{split}
\end{equation*}
defined by 
\begin{equation*}
\begin{split}
\countingone(L) &= \big| \big\{L_{i,j} = \begin{tikzpicture}[scale=0.4]
\gri{1}{1} \dela{1}{1}
\end{tikzpicture} \ , \ \begin{tikzpicture}[scale=0.4]
\gri{1}{1} \lmub{1}{1}
\end{tikzpicture}\big\} \big| \\
\countingtwo(L) &= \big|  \big\{L_{i,j} = \begin{tikzpicture}[scale=0.4]
\gri{1}{1} \delb{1}{1}
\end{tikzpicture} \ , \ \begin{tikzpicture}[scale=0.4]
\gri{1}{1} \lmub{1}{1}
\end{tikzpicture}\big\}\big| .
\end{split}
\end{equation*}
\end{notation}
\begin{proposition}
\label{remarksigns}
For any $M \in \mosaic{n}{m}$ and $L \in \mathfrak{L}(M)$ we have
\[ 
(-1)^{\cob(M)} (-1)^{\countingone(L)} = (-1)^{\countingtwo(L)}. 
\]
\end{proposition}
\begin{proof}
Set $k = |\{L_{i,j} = \begin{tikzpicture}[scale=0.4]
\gri{1}{1} \dela{1}{1}
\end{tikzpicture} \}|$, $\ell =  | \{L_{i,j} = \begin{tikzpicture}[scale=0.4]
\gri{1}{1} \delb{1}{1}
\end{tikzpicture} \}|$ and $h = | \{L_{i,j} = \begin{tikzpicture}[scale=0.4]
\gri{1}{1} \lmub{1}{1}
\end{tikzpicture} \}|$.
By the definition of $\alpha$ we have $\cob(M) = k + \ell$, hence $ (-1)^{\cob(M)} (-1)^k = (-1)^\ell$. To end the proof it suffices to multiply $(-1)^h$ to both sides.
\end{proof}

\section{Structure of the Drinfeld--Yetter algebra}
\label{section-explicit-formula}
\subsection{An explicit formula for the multiplication}

In this Section we provide an explicit formula for the multiplication of $ \mathfrak{U}_{\DY}^1$ with respect to the standard basis \eqref{eq:canonical-basis}.\\
To any Drinfeld--Yetter mosaic $M \in \mosaic{n}{m}$ we associate a morphism in $\Hom_{\DY^1}([n] \ten [V_1], [m] \ten [V_1])$ by considering the picture obtained by removing all borders from the Drinfeld--Yetter mosaic and turning it 45 degrees clockwise. This procedure defines a colletion of maps 
\begin{equation}
\label{eq:morphism-from-mosaic-to-hom}
\varphi_{n,m} : \mosaic{n}{m}\to \Hom_{\DY^1}([n] \ten [V_1], [m] \ten [V_1]).
\end{equation}
\begin{example}
Given the following Drinfeld--Yetter mosaic
\vspace{0.3 cm}
\begin{center}
\begin{tikzpicture}
\gri{2}{2} \cro{1}{1} \cro{1}{2} \cro{2}{1} \tbra{2}{2}
\node at (-1,-1) {$M=$};
\end{tikzpicture}
\end{center}
\vspace{0.3 cm}
the procedure described above gives the following picture
\vspace{0.3 cm}
\begin{center}
\begin{tikzpicture}[scale=0.85]
\gri{2}{2} \cro{1}{1} \cro{1}{2} \cro{2}{1} \tbra{2}{2}
\node at (3,-1) {$\longrightarrow$};
\begin{scope}[shift={(4,0)}]
\cro{1}{1} \cro{1}{2} \cro{2}{1} \tbra{2}{2}
\node at (3,-1) {$\longrightarrow$};
\end{scope}
\begin{scope}[shift={(9,0.3)},rotate=315]
\cro{1}{1} \cro{1}{2} \cro{2}{1} \tbra{2}{2}
\end{scope}
\node at (11,-1) {$\longrightarrow$};
\begin{scope}[shift={(12,-1.75)},scale=1.5]
\draw (1,0) arc (0:100:0.75);
\draw[-, very thick,ggreen] (0,0)--(1.5,0);
\draw (1.25,0.75) arc (90:180:0.75); 
\draw (0.63,0.4) arc (-300:-260:0.8); 
\draw (1,0.90) arc (90:190:0.75); 
\draw(1,0.9)--(1.2,0.9);
\end{scope}
\end{tikzpicture}
\end{center}
\vspace{0.3 cm}
leading to the following morphism of $\Hom_{\DY^1}([2] \ten [V_1], [2] \ten [V_1])$:
\[ \varphi_{2,2}(M) =(\id_{[2]} \ten \pi_1^{*(2)}) \circ \bigl(\bigl(  (23) \circ\id_{[2]} \ten \delta \circ (132)\bigr) \ten \id_{[V_1]}\bigr) \circ (\id_{[2]} \ten \pi_1).
\]
\end{example}
\begin{remark}
Note that $\varphi_{0,m}(M) = \pi_1^{*(m)}$, $\varphi_{n,0}(M) = \pi_1^{(n)}$ and $\varphi_{0,0}(M)= \id_{[1]}$.
\end{remark}
Similarly, for any $M \in \mosaic{n}{m}$ we denote by $M^\top$ the morphism of $\Hom_{\DY^1}([n+m - \cob(M) ] , [n+m - \bra(M) ])$ pictorially represented by removing all borders from $M$, turning it 45 degrees clockwise and attaching horizontal lines to the end and beginning of any diangonal line.
\begin{example}
For the Drinfeld--Yetter mosaic $M$ of the previous example, we obtain the following picture 
\begin{center}
\begin{tikzpicture}[scale=0.85]
\gri{2}{2} \cro{1}{1} \cro{1}{2} \cro{2}{1} \tbra{2}{2}
\node at (3,-1) {$\longrightarrow$};
\begin{scope}[shift={(4,0)}]
\cro{1}{1} \cro{1}{2} \cro{2}{1} \tbra{2}{2}
\node at (3,-1) {$\longrightarrow$};
\end{scope}
\begin{scope}[shift={(9,0.3)},rotate=315]
\cro{1}{1} \cro{1}{2} \cro{2}{1} \tbra{2}{2}
\end{scope}
\node at (11,-1) {$\longrightarrow$};
\begin{scope}[shift={(13.2,0)}, rotate=315, scale=0.8]
\cro{1}{1} \cro{1}{2} \cro{2}{1} \tbra{2}{2}
\draw[red, very thick] (0.5,0)--(1,0.5)--(1.5,1);
\draw[red, very thick] (1.5,0)--(2,0.5);
\draw[red, very thick] (2,-0.5)--(2.5,0);
\draw[red, very thick] (0,-0.5)--(-0.5,-1)--(-1,-1.5);
\draw[red, very thick] (0,-1.5)--(-0.5,-2);
\draw[red, very thick] (0.5,-2)--(0,-2.5);
\draw[red, very thick] (1.5,-2)--(1,-2.5)--(0.5,-3);
\end{scope}
\end{tikzpicture}
\end{center}
corresponding to the morphism
\[ M^	\top = (23) \circ (\id_{[2]} \ten \delta) \circ (132) \in \Hom_{\DY^1}([4],[3]).\]
\end{example}
Note that for any $M \in \mosaic{n}{m}$ we have 
\begin{equation}
\label{eq:morphis-associated-to-a-tableaux}
\varphi_{n,m}(M) = \big(\id_{[n]} \ten \pi_1^{*(m-\cob(M))}\big) \circ (M^\top \ten \id_{[V_1]})\circ \big(\id_{[m]} \ten \pi_1^{(n - \bra(M))}\big).
\end{equation}
\begin{lemma}
\label{lemma-normal-ordering-tableaux}
For any $n,m \geqslant 0$ we have 
\begin{equation*}
\begin{tikzpicture} 
\draw[-, very thick,ggreen] (0,0)--(2.5,0);
\draw (1,0) arc (0:90:0.75);
\draw(0.25,0.75)--(0,0.75);
\draw[very thick]  (2.25,0.75) arc (90:180:0.75);
\draw[very thick] (2.25,0.75)--(2.5,0.75);
\node at (1.65,-0.3) {$m$};
\node at (5,0.2) {$= \mathlarger{\sum_{M \in \mosaic{1}{m}} (-1)^{\cob(M)} \varphi_{1,m}(M)}$};
\end{tikzpicture}
\end{equation*}
and 
\begin{equation*}
\begin{tikzpicture} 
\draw[-, very thick,ggreen] (0,0)--(2.5,0);
\draw[very thick]  (1,0) arc (0:90:0.75);
\draw[very thick] (0.25,0.75)--(0,0.75);
\draw (2.25,0.75) arc (90:180:0.75);
\draw(2.25,0.75)--(2.5,0.75);
\node at (0.9,-0.3) {$n$}; 
\node at (5,0.2) {$= \mathlarger{\sum_{M \in \mosaic{n}{1}} (-1)^{\cob(M)} \varphi_{n,1}(M)}$};
\end{tikzpicture}.
\end{equation*}
\end{lemma}
\begin{proof}
We prove the first identity by induction on $m \geqslant 0$. For $m=0$ the claim holds trivially. For $m=1$ we have \\ \\
\begin{equation*}
\begin{tikzpicture}
\draw[-, very thick,ggreen] (0+0.5,0)--(2.5+0.5,0);
\draw (1+0.5,0) arc (0:90:0.75);
\draw(0.25+0.5,0.75)--(0+0.5,0.75);
\draw (2.25+0.5,0.75) arc (90:180:0.75);
\draw(2.25+0.5,0.75)--(2.5+0.5,0.75);
\node at (3.5+0.5,0.3) {$=$};
\draw[-, very thick,ggreen] (5,0)--(6.25,0);
\draw (1+5,0) arc (0:90:0.75);
\draw(0.25+5,0.75)--(0+5,0.75);
\draw (1+5,0.75) arc (90:180:0.75);
\draw(1+5,0.75)--(1.25+5,0.75);
\node at (1.8+5,0.4) {$+$}; 
\draw[-, very thick,ggreen] (2.25+5,0)--(3.5+5,0);
\draw (3.25+5,0.75) arc (90:180:0.75); 
\draw(3.25+5,0.75)--(3.5+5,0.75);
\draw (2.8+5,0.6) arc (-300:-260:1); 
\node at (4+5,0.4) {$-$}; 
\draw[-, very thick,ggreen] (4.5+5,0)--(6+5,0); 
\draw (5.5+5,0) arc (0:90:0.75);
\draw(4.75+5,0.75)--(4.5+5,0.75);
\draw (5.75+5,0.75) arc (90:131:0.75); 
\draw (5.75+5,0.75)--(6+5,0.75);
\node at (3.5+0.5,-1.7) {$=$};
\node at (7.7,-1.6 ) {$\varphi_{1,1} \big( \ \begin{tikzpicture}[scale=0.4] \gri{1}{1} \cro{1}{1} \end{tikzpicture} \ \big) + \varphi_{1,1} \big( \ \begin{tikzpicture}[scale=0.4]  \gri{1}{1} \tbra{1}{1} \end{tikzpicture} \ \big) - \varphi_{1,1} \big( \  \begin{tikzpicture}[scale=0.4]  \gri{1}{1} \tcobra{1}{1} \end{tikzpicture} \ \big)$};
\node at (3.5+0.5,-3) {$=$};
\node at (7,-3) {$\mathlarger{\sum_{M \in \mosaic{1}{1}}(-1)^{\cob(M)}  \varphi_{1,1}(M)}$}; 
\end{tikzpicture}
\end{equation*}
where the first equality follows from Equation \eqref{eq:dyrule}. If $m \geqslant 1$, we have \\ \\
\begin{equation*}
\begin{tikzpicture}[scale=0.9]
\draw[-, very thick,ggreen] (0,0)--(2.5,0);
\draw (1,0) arc (0:90:0.75);
\draw(0.25,0.75)--(0,0.75);
\draw[very thick]  (2.25,0.75) arc (90:180:0.75);
\draw[very thick] (2.25,0.75)--(2.5,0.75);
\node at (1.7,-0.3) {\footnotesize $m\!+\!1$};
\node at (3.5,0.35) {$=$};
\draw[-, very thick,ggreen] (4.5,0)--(7.2,0);
\draw (5.5,0) arc (0:90:0.75);
\draw(4.75,0.75)--(4.5,0.75);
\draw (6.75,0.75) arc (90:180:0.75);
\draw(6.7,0.75)--(7.2,0.75);
\draw[very thick]  (7.05,0.5) arc (90:170:0.6); 
\draw[very thick] (7.05,0.5)--(7.2,0.5);
\node at (6.5,-0.2) {\footnotesize $m$};
\node at (3.5,-2) {$=$};
\node at (4,-2) {$\Bigg($};
\draw[-, very thick,ggreen] (0+4.5,0-2.3)--(1.25+4.5,0-2.3);
\draw (1+4.5,0-2.3) arc (0:90:0.75);
\draw(0.25+4.5,0.75-2.3)--(0+4.5,0.75-2.3);
\draw (1+4.5,0.75-2.3) arc (90:180:0.75);
\draw(1+4.5,0.75-2.3)--(1.25+4.5,0.75-2.3);
\node at (1.8+4.5,0.4-2.3) {$+$}; 
\draw[-, very thick,ggreen] (2.25+4.5,0-2.3)--(3.5+4.5,0-2.3);
\draw (3.25+4.5,0.75-2.3) arc (90:180:0.75); 
\draw(3.25+4.5,0.75-2.3)--(3.5+4.5,0.75-2.3);
\draw (2.8+4.5,0.6-2.3) arc (-300:-260:1); 
\node at (4+4.5,0.4-2.3) {$-$}; 
\draw[-, very thick,ggreen] (4.5+4.5,0-2.3)--(6+4.5,0-2.3); 
\draw (5.5+4.5,0-2.3) arc (0:90:0.75);
\draw(4.75+4.5,0.75-2.3)--(4.5+4.5,0.75-2.3);
\draw (5.75+4.5,0.75-2.3) arc (90:131:0.75); 
\draw (5.75+4.5,0.75-2.3)--(6+4.5,0.75-2.3);
\node at (11,-2) {$\Bigg)$};
\node at (11.8,-2.5) {\footnotesize $m$ };
\draw[-, very thick,ggreen] (0+11.5,0-2.3)--(1.3+11.5,0-2.3);
\draw[very thick]  (1+11.5,0.75-2.3) arc (90:180:0.75);
\draw[very thick] (1+11.5,0.75-2.3)--(1.25+11.5,0.75-2.3);
\node at (3.5,-2-2.3) {$=$};
\draw[-, very thick,ggreen] (4.5,0-2.3-2.3)--(1.25+5.2,0-2.3-2.3);
\draw (1+4.5,0-2.3-2.3) arc (0:90:0.75);
\draw(0.25+4.5,0.75-2.3-2.3)--(0+4.5,0.75-2.3-2.3);
\draw (1+4.5,0.75-2.3-2.3) arc (90:180:0.75);
\draw (1+4.5,0.75-2.3-2.3)--(6.45,0.75-2.3-2.3); 
\draw[very thick]  (1+5.3,0.6-2.3-2.3-0.1) arc (90:170:0.6);
\draw[very thick] (1+5.3,0.6-2.3-2.3-0.1)--(1+5.3+0.15,0.6-2.3-2.3-0.1);  
\node at (11.4,-2.65-2.2) {\footnotesize $m$};
\node at (7,0.4-2.3-2.3) {$+$}; 
\draw[-, very thick,ggreen] (7.5,0-2.3-2.3)--(9,0-2.3-2.3);
\draw (3.25+5+0.35,0.75-2.3-2.3) arc (90:180:0.75); 
\draw(3.25+5+0.35,0.75-2.3-2.3)--(3.65+5+0.35,0.75-2.3-2.3);
\draw (2.8+5+0.35,0.6-2.3-2.3) arc (-300:-260:1); 
\draw[very thick]  (7.05+1.5+0.35,0.6-2.3-2.3-0.1) arc (90:170:0.6);
\draw[very thick] (7.05+1.5+0.35,0.6-2.3-2.3-0.1)--(7.05+1.5+0.35+0.1,0.6-2.3-2.3-0.1);
\node at (8.4,-2.65-2.2) {\footnotesize $m$};
\node at (5.8,-2.65-2.2) {\footnotesize $m$};
\node at (9.5,0.4-2.3-2.3) {$-$}; 
\draw[-, very thick,ggreen] (4.5+4.5+1,0-2.3-2.3)--(12,0-2.3-2.3); 
\begin{scope}[shift={(3,0)}]
\draw[very thick]  (7.05+1.5+0.35,0.6-2.3-2.3-0.1) arc (90:170:0.6);
\draw[very thick] (7.05+1.5+0.35,0.6-2.3-2.3-0.1)--(7.05+1.5+0.35+0.1,0.6-2.3-2.3-0.1);
\end{scope}
\begin{scope}[shift={(1,-2.3)}]
\draw (5.5+4.5,0-2.3) arc (0:90:0.75);
\draw(4.75+4.5,0.75-2.3)--(4.5+4.5,0.75-2.3);
\draw (5.75+4.5,0.75-2.3) arc (90:131:0.75); 
\draw (5.75+4.5,0.75-2.3)--(6.5+4.5,0.75-2.3);
\end{scope}
\node at (3.5,-2-2.3-2.3) {$=$};
\node at (9.5,-2-2.3-2.6)  
{$\mathlarger{\sum_{M \in \mosaic{1}{m+1}^{\begin{tikzpicture}[scale=0.3] \gri{1}{1} \cro{1}{1} \end{tikzpicture} } }} (-1)^{\cob(M)} \varphi_{1,m+1}(M) \ +\mathlarger{\sum_{M \in \mosaic{1}{m+1}^{\begin{tikzpicture}[scale=0.3] \gri{1}{1} \tbra{1}{1} \end{tikzpicture} } }} (-1)^{\cob(M)} \varphi_{1,m+1}(M) \   $};
\node at (3.5,-2-2.3-2.3-2.3) {$+$};
\node at (6.5,-2-2.3-2.3-2.6) {$\mathlarger{\sum_{M \in \mosaic{1}{m+1}^{\begin{tikzpicture}[scale=0.3] \gri{1}{1} \tcobra{1}{1} \end{tikzpicture} } }} (-1)^{\cob(M)} \varphi_{1,m+1}(M)$}; 
\begin{scope}[shift={(0,-2)}]
\node at (3.5,-2-2.3-2.3-2.3) {$=$};
\node at (6.5,-2-2.3-2.3-2.6) {$\mathlarger{\sum_{M \in \mosaic{1}{m+1} }} (-1)^{\cob(M)} \varphi_{1,m+1}(M)$}; 
\end{scope}
\end{tikzpicture}
\end{equation*}
where the second equality follows from Equation \eqref{eq:dyrule} and the fourth equality follows by the inductive hypothesis. Then the first part of the claim is proved. The proof of the second part is analogous.
\end{proof}
The next Proposition is the main result regarding Drinfeld--Yetter mosaics:
\begin{proposition}
\label{claim}
Let $n,m \geqslant 0$. Then 
\begin{equation*}
\begin{tikzpicture} 
\draw[-, very thick,ggreen] (0,0)--(2.5,0);
\draw[very thick]  (1,0) arc (0:90:0.75);
\draw[very thick] (0.25,0.75)--(0,0.75);
\draw[very thick]  (2.25,0.75) arc (90:180:0.75);
\draw[very thick] (2.25,0.75)--(2.5,0.75);
\node at (1.65,-0.3) {$m$};
\node at (0.9,-0.3) {$n$}; 
\node at (5,0.2) {$= \mathlarger{\sum_{M \in \mosaic{n}{m}} (-1)^{\cob(M)} \varphi_{n,m}(M)}$.};
\end{tikzpicture}
\end{equation*} 
\end{proposition}
\begin{proof}
The cases $n=0$, $m=0$ and $n=m=0$ hold trivially, while the cases $n=1$ and $m=1$ hold for Lemma \ref{lemma-normal-ordering-tableaux}. For $n,m \geqslant 2$ we have 
\begin{equation*}
\begin{tikzpicture}
\draw[-, very thick,ggreen] (0,0)--(2.5,0);
\draw[very thick]  (1,0) arc (0:90:0.75);
\draw[very thick] (0.25,0.75)--(0,0.75);
\draw[very thick]  (2.25,0.75) arc (90:180:0.75);
\draw[very thick] (2.25,0.75)--(2.5,0.75);
\node at (1.65,-0.3) {$m$};
\node at (0.9,-0.3) {$n$}; 
\node at (3.2,0.3) {=};
\draw[-, very thick,ggreen] (0+3.8,0)--(2.5+4,0);
\draw (1+4,0) arc (0:90:0.75);
\draw[very thick]  (1+3.5,0) arc (0:100:0.6);
\draw(0.25+4,0.75)--(0+3.8,0.75);
\draw[very thick]  (2.25+4,0.75) arc (90:180:0.75);
\draw[very thick] (2.25+4,0.75)--(2.5+4,0.75);
\node at (1.65+4,-0.3) {$m$};
\node at (0.4+4,-0.3) {\footnotesize $n\!-\!1$}; 
\node at (3.2,0.3-2.3) {=};
\draw[-, very thick,ggreen] (0+3.8,0-2.3)--(1+4,0-2.3);
\draw[very thick]  (1+3.5,0-2.3) arc (0:100:0.6);
\node at (0.4+4,-0.3-2.3) {\footnotesize $n\!-\!1$}; 
\node at (7.5,-2.1) {$\Bigg( \mathlarger{\sum_{M \in \mosaic{1}{m}} (-1)^{\cob(M)} \varphi_{1,m}(M)} \Bigg)$};
\node at (3.2,0.3-2.3-2.3) {=};
\draw[-, very thick,ggreen] (0+3.8,0-2.3-2.3)--(1+4,0-2.3-2.3);
\draw[very thick]  (1+3.5,0-2.3-2.3) arc (0:100:0.6);
\node at (0.4+4,-0.3-2.3-2.3) {\footnotesize $n\!-\!1$}; 
\node at (8,-2.1-2.3) {$\Bigg( \mathlarger{\sum_{\ell=0}^m \ \sum_{M \in \mosaic{1}{m}^{(\ell)}} (-1)^{\cob(M)} \varphi_{1,m}(M)} \Bigg)$};
\node at (3.2,0.3-2.3-2.3-2.3) {=};
\node at (6,0.3-2.3-2.3-2.5) {$\mathlarger{\sum_{\ell=0}^m \ \sum_{\substack{M \in \mosaic{1}{m}^{(\ell)} \\ \bra(M)=0 }} (-1)^{\cob(M)}} \Bigg($};
\draw[-, very thick,ggreen] (0+8.5-0.15,0-2.3-2.3-2.3)--(4.5+8.5-0.15,0-2.3-2.3-2.3);
\draw[very thick]  (1+8.5-0.15,0-2.3-2.3-2.3) arc (0:90:0.75);
\draw[very thick] (0.25+8.5-0.15,0.75-2.3-2.3-2.3)--(0+8.5-0.15,0.75-2.3-2.3-2.3);
\draw[very thick]  (2.25+8.5-0.15,0.75-2.3-2.3-2.3) arc (90:180:0.75);
\draw[very thick] (2.25+8.5-0.15,0.75-2.3-2.3-2.3)--(2.5+8.5-0.15,0.75-2.3-2.3-2.3);
\draw (1+3.25+8.5-0.15,0-2.3-2.3-2.3) arc (0:90:0.75);
\draw(0.25+3.25+8.5-0.15,0.75-2.3-2.3-2.3)--(0+3.25+8.5-0.15,0.75-2.3-2.3-2.3);
\draw [draw=black] (2.5+8.5-0.15,0.65-2.3-2.3-2.3) rectangle (3.25+8.5-0.15,1.1-2.3-2.3-2.3);
\draw (2.5+8.5-0.15,1-2.3-2.3-2.3)--(0+8.5-0.15,1-2.3-2.3-2.3);
\draw[very thick]  (3.25+8.5-0.15,1-2.3-2.3-2.3)--(4.5+8.5-0.15,1-2.3-2.3-2.3);
\node at (4.6+9-0.15,0.6-2.3-2.3-2.5) {$\Bigg)$};
\node at (3.2,0.3-2.3-2.3-2.3-2.3) {+};
\node at (6,0.3-2.3-2.3-2.5-2.3) {$\mathlarger{\sum_{\ell=0}^m \ \sum_{\substack{M \in \mosaic{1}{m}^{(\ell)} \\ \bra(M)=1  }} (-1)^{\cob(M)}} \Bigg($};
\draw[-, very thick,ggreen] (0+8.5-0.15,0-2.3-2.3-2.3-2.3)--(4.5+7.75-0.15,0-2.3-2.3-2.3-2.3);
\node at (0+8.5-0.15+4.5-0.5,0.75-2.3-2.3-2.3-2.05) {\footnotesize $m$};
\node at (0+8.5-0.15+5.2-0.5,0.75-2.3-2.3-2.05) {\footnotesize $m$};
\node at (9.2,0-2.3-2.3-2.3-2.5) { \footnotesize $n\!-\!1$};
\node at (9.9,0-2.3-2.3-2.3-2.5) { \footnotesize $\ell$};
\node at (9.9,0-2.3-2.3-2.5) { \footnotesize $\ell$};
\node at (9.2,0-2.3-2.3-2.5) { \footnotesize $n\!-\!1$};
\node at (0+8.5-0.15+3.45-0.5,0.75-2.3-2.3-2.15) {\scriptsize $M^\top$};
\node at (0+8.5-0.15+3.45-0.5,0.75-2.3-2.3-2.3-2.15) {\scriptsize $M^\top$};
\draw[very thick]  (1+8.5-0.15,0-2.3-2.3-2.3-2.3) arc (0:90:0.75);
\draw[very thick] (0.25+8.5-0.15,0.75-2.3-2.3-2.3-2.3)--(0+8.5-0.15,0.75-2.3-2.3-2.3-2.3);
\draw[very thick]  (2.25+8.5-0.15,0.75-2.3-2.3-2.3-2.3) arc (90:180:0.75);
\draw[very thick] (2.25+8.5-0.15,0.75-2.3-2.3-2.3-2.3)--(2.5+8.5-0.15,0.75-2.3-2.3-2.3-2.3);
\draw [draw=black] (2.5+8.5-0.15,0.65-2.3-2.3-2.3-2.3) rectangle (3.25+8.5-0.15,1.1-2.3-2.3-2.3-2.3);
\draw (2.5+8.5-0.15,1-2.3-2.3-2.3-2.3)--(0+8.5-0.15,1-2.3-2.3-2.3-2.3);
\draw[very thick]  (3.25+8.5-0.15,1-2.3-2.3-2.3-2.3)--(4.5+7.75-0.15,1-2.3-2.3-2.3-2.3);
\node at (4.6+8.3-0.15,0.6-2.3-2.3-2.5-2.3) {$\Bigg)$};
\node at (3.2,0.3-2.3-2.3-2.3-2.3-2.3) {=};
\node at (7,0.1-2.3-2.3-2.3-2.3-2.3) {$\mathlarger{\sum_{\ell=0}^m \quad \sum_{M \in \mosaic{1}{m}^{(\ell)} \times \mosaic{n-1}{\ell}} (-1)^{\cob(M)} \varphi_{n,m}(M)} $};
\node at (3.2,0.3-2.3-2.3-2.3-2.3-2.3-2.3) {=};
\node at (5.5,0.1-2.3-2.3-2.3-2.3-2.3-2.3) {$\mathlarger{\sum_{M \in \mosaic{n}{m}}} (-1)^{\cob(M)} \varphi_{n,m}(M)$};
\end{tikzpicture}
\end{equation*}
where the second equality follows from Lemma \ref{lemma-normal-ordering-tableaux}, the third and the fifth equalities follow from Proposition \ref{recurrence-rule-dy-tableaux}, the fourth equality follows from Equation \eqref{eq:morphis-associated-to-a-tableaux} and the sixth equality follows from Equation \eqref{eq:drinfeld-yetter-tableaux-recursive-one}.
\end{proof}
Next, we associate a permutation in $\mathfrak{S}_{n+m}$ to any $L \in \loom{n}{m}$, $\sigma \in \mathfrak{S}_n$ and $\tau \in \mathfrak{S}_m$, by gluing $\sigma$ (resp. $\tau$) to the left (resp. top) edge of $L$, moving  together strings appearing in the same tile. We thus get a permutation of $\mathfrak{S}_{n+m}$ by labelling the strings from $1$ to $n+m$, following the same argument of Example \ref{example-permutation-ass-to-a-loom}.
\begin{example}

Consider the following Drinfeld--Yetter loom $L \in \loom{2}{3}$
\begin{center}

\begin{tikzpicture}[scale=0.7]
\gri{2}{3} \ver{1}{3} \cro{1}{1} \ver{2}{3}
\draw[red, very thick] (2-1,-1+0.5) arc (-90:0:0.5); 
\draw[red, very thick] (2-0.7,-1+1)--(2-0.7, -2);
\draw[red, very thick] (1-1,-2+0.5) arc (-90:0:0.5);
\draw[red, very thick] (1-1, -2+0.35)--(3, -2+0.35);
\end{tikzpicture}
\end{center}
and the permutations $\sigma =(12) \in \mathfrak{S}_2$ and $\tau= (132) \in \mathfrak{S}_3$. By gluing $\sigma$ to the left edge of $L$ and $\tau$ to the top edge of $L$, we obtain the picture 
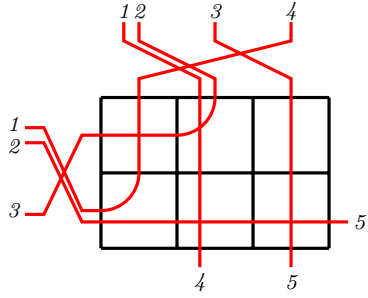
\begin{center}

\begin{tikzpicture}
\gri{2}{3} \ver{1}{3} \cro{1}{1} \ver{2}{3}
\draw[red, very thick] (2-1,-1+0.5) arc (-90:0:0.5); 
\draw[red, very thick] (2-0.7,-1+1)--(2-0.7, -2);
\draw[red, very thick] (1-1,-2+0.5) arc (-90:0:0.5);
\draw[red, very thick] (1-1, -2+0.35)--(3, -2+0.35);
\draw[red, very thick] (2.5,-2)--(2.5,-2.25);
\draw[red, very thick] (1.3,-2)--(1.3,-2.25);
\draw[red, very thick] (1.3,0)--(1.3,0.25)--(0.3,0.75)--(0.3,1);
\draw[red, very thick] (1.5,0)--(1.5,0.25)--(0.5,0.75)--(0.5,1);
\draw[red, very thick] (2.5,0)--(2.5,0.25)--(1.5,0.75)--(1.5,1);
\draw[red, very thick] (0.5,0)--(0.5,0.25)--(2.5,0.75)--(2.5,1);
\draw[red, very thick] (3,-1.65)--(3.25,-1.65);
\draw[red, very thick] (0,-0.5)--(-0.25,-0.5)--(-0.75,-1.55)--(-1,-1.55);
\draw[red, very thick] (0,-1.5)--(-0.25,-1.5)--(-0.75,-0.4)--(-1,-0.4);
\draw[red, very thick] (0,-1.65)--(-0.25,-1.65)--(-0.75,-0.6)--(-1,-0.6);
\node at (-1.15,-0.35) {\footnotesize{1}};
\node at (-1.15,-0.65) {\footnotesize{2}};
\node at (-1.15,-1.5) {\footnotesize{3}};
\node at (1.3,-2.45) {\footnotesize{4}};
\node at (2.5,-2.45) {\footnotesize{5}};
\node at (0.3,1.15)  {\footnotesize{1}};
\node at (0.5,1.15)  {\footnotesize{2}};
\node at (1.5,1.15)  {\footnotesize{3}};
\node at (2.5,1.15)  {\footnotesize{4}};
\node at (3.4,-1.65)  {\footnotesize{5}};
\end{tikzpicture}
\end{center}
to which we associate the permutation $(14)(253) \in \mathfrak{S}_5$.
\end{example}
The procedure described above defines a family of maps
$\tilde{\gamma}_{n,m}: \mathfrak{S}_n \times \loom{n}{m} \times \mathfrak{S}_m \to \mathfrak{S}_{n+m}$.
Note that for all $L \in \loom{n}{m}$ we have
$\tilde{\gamma}_{n,m}(\id_n, L, \id_m) = \gamma_{n,m}(L)$.
\begin{lemma}
\label{lemmamultiplication}
Let $M \in \mosaic{n}{m}$. Then 
\begin{center}
\begin{tikzpicture}
\draw[-, very thick,ggreen] (0,0)--(2,0);
\draw[-, very thick,ggreen] (3.8+1,0)--(5.9+1,0);
\draw[very thick]  (1,0.75) arc (90:180:0.75);
\draw[very thick] (1,0.75)--(1.2,0.75);
\node[shape=circle,draw,inner sep=1pt] (char) at (1.4,0.75) {$\sigma$}; 
\draw[very thick] (1.6,0.75)--(2,0.75);
\node at (0.3,-0.3) {$n$};
\begin{scope}[shift={(0.1,0)}]
\draw[very thick] (3.8+1,0.75)--(4.2+1,0.75);
\node[shape=circle,draw,inner sep=1pt] (char) at (4.4+1,0.75) {$\tau$}; 
\draw[very thick] (4.6+1,0.75)--(4.8+1,0.75);
\draw[very thick]  (5.55+1,0) arc (0:90:0.75);
\node at (5.55+1,-0.3) {$m$};
\node at (3.6+1,0.75) {$m$};
\end{scope}
\node at (2.2,0.75) {$n$};
\node at (3.4,0.4) {\Bigg($\varphi_{n,m}(M)$ \Bigg) };
\node at (7.5,0.4) {$=$};
\begin{scope}[shift={(-0.5,0)}]
\node at (-2.7+12,0.33) {$\mathlarger{\sum_{L \in \mathfrak{L}(M)} (-1)^{\countingone(L)}}$};
\end{scope}
\begin{scope}[shift={(-1.2,0)}]
\draw[-, very thick,ggreen] (11,0)--(14.9,0);
\draw[very thick]  (1+11,0.75) arc (90:180:0.75);
\draw[very thick] (1+11,0.75)--(1.15+11,0.75);
\begin{scope}[shift={(0.3,0)}]
\node[shape=rectangle,draw,inner sep=1pt] (char) at (1.65+11,0.75) {\tiny $ \tilde{\gamma}_{n,m}(\sigma, L, \tau) $}; 
\end{scope}
\begin{scope}[shift={(0.6,0)}]
\draw[very thick] (1.65+0.25+0.25+11,0.75)--(1.8+0.25+0.25+11,0.75);
\draw[very thick]  (2.55+0.25+0.25+11,0) arc (0:90:0.75);
\node at (2.55+0.25+0.25+11,-0.3) {$n\!+\!m$};
\end{scope}
\node at (0.3+11,-0.3) {$n\!+\!m$};
\end{scope}
\end{tikzpicture}.
\end{center}
\end{lemma}
Before proving the Lemma, let us give an
\begin{example}
Let $M$ be the following $2 \times 2$ Drinfeld--Yetter mosaic
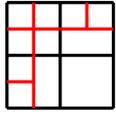
\begin{center}
\begin{tikzpicture}[scale=0.7]
\gri{2}{2} \cro{1}{1} \tbra{2}{1} \tcobra{1}{2}
\end{tikzpicture}.
\end{center}
Then the morphism $\varphi_{2,2}(M)$ is pictorially represented by 
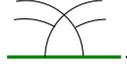
\begin{center}
\begin{tikzpicture}
\draw (1,0) arc (0:100:0.75);
\draw[-, very thick,ggreen] (0,0)--(1.5,0);
\draw (1.25,0.75) arc (90:180:0.75); 
\draw (0.63,0.4) arc (-300:-260:0.8); 
\draw(1.3,0.5) arc (90:120:0.8);
\end{tikzpicture}.
\end{center}
Gluing two permutations $\sigma, \tau \in \mathfrak{S}_2$ to $\varphi_{2,2}(M)$ and composing with $\pi_1^{*(2)}$ and $\pi_1^{(2)}$ we obtain 
\vspace{0.5cm}
\begin{center}
\begin{tikzpicture}
\draw (1,0) arc (0:90:0.75);
\draw[-, very thick,ggreen] (-1.5,0)--(3,0);
\draw (1.25,0.75) arc (90:180:0.75); 
\draw (0.63,0.4) arc (-300:-270:0.8); 
\draw(1.3,0.5) arc (90:120:0.8);
\draw(1.3,0.5)--(1.5,0.5);
\draw (1.25,0.75)--(1.5,0.75);
\draw(0.25,0.75)--(0,0.75);
\draw(0.25,0.507)--(0,0.507);
\node[shape=circle,draw,inner sep=1pt] (char) at (1.65,0.63) {$\tau$}; 
\node[shape=circle,draw,inner sep=1pt] (char) at (-0.15,0.63) {$\sigma$}; 
\draw(-0.3,0.75)--(-0.5,0.75);
\draw(-0.3,0.507)--(-0.5,0.507);
\draw (-0.5,0.75) arc (90:180:0.75);
\draw (-0.5,0.507) arc (90:180:0.507);
\draw(1.8,0.5)--(2,0.5);
\draw (1.8,0.75)--(2,0.75);
\draw (2.75,0) arc (0:90:0.75);
\draw (2.5,0) arc (0:90:0.5);
\end{tikzpicture}
\end{center}
and applying relations \eqref{eq:actionrule} and \eqref{eq:coactionrule}, we obtain 
\begin{center}
\begin{tikzpicture}
\draw (1,0) arc (0:90:0.75);
\draw[-, very thick,ggreen] (-1.5,0)--(3,0);
\draw (1.25,0.75) arc (90:180:0.75); 
\draw (0.63,0.4) arc (-300:-270:0.8); 
\draw(1.3,0.5) arc (90:120:0.8);
\draw(1.3,0.5)--(1.5,0.5);
\draw (1.25,0.75)--(1.5,0.75);
\draw(0.25,0.75)--(0,0.75);
\draw(0.25,0.507)--(0,0.507);
\node[shape=circle,draw,inner sep=1pt] (char) at (1.65,0.63) {$\tau$}; 
\node[shape=circle,draw,inner sep=1pt] (char) at (-0.15,0.63) {$\sigma$}; 
\draw(-0.3,0.75)--(-0.5,0.75);
\draw(-0.3,0.507)--(-0.5,0.507);
\draw (-0.5,0.75) arc (90:180:0.75);
\draw (-0.5,0.507) arc (90:180:0.507);
\draw(1.8,0.5)--(2,0.5);
\draw (1.8,0.75)--(2,0.75);
\draw (2.75,0) arc (0:90:0.75);
\draw (2.5,0) arc (0:90:0.5);
\node at (3.3,0.3) {$=$};
\node at (3.3,-1.7) {$+$};
\node at (8.3,0.3) {$-$};
\node at (8.3,-1.7) {$-$};
\begin{scope}[shift={(4,0)}]
\draw[-, very thick,ggreen] (0,0)--(4.1,0);
\draw (1.2,1) arc (90:180:1);
\draw (1.2,0.8) arc (90:180:0.8);
\draw (1.2,0.6) arc (90:180:0.6);
\draw (1.2,0.4) arc (90:180:0.4);
\draw (1.2,0.4)--(1.8,0.4)--(2.3,0.8);
\draw (1.7,0.8)--(1.8,0.8)--(2.3,0.6);
\draw (1.7,0.6)--(1.8,0.6)--(2.3,1);
\draw (1.7,1)--(1.8,1)--(2.3,0.4);
\draw [draw=black] (1.2,1.05) rectangle (1.7,0.55);
\node at (1.45,0.8) {$\tilde{\sigma}$};
\draw (2.3,0.8)--(2.4,0.8);
\draw (2.3,0.6)--(2.4,0.6);
\draw (2.3,1)--(2.4,1);
\draw (2.3,0.4)--(2.9,0.4);
\draw [draw=black] (2.4,1.05) rectangle (2.9,0.55);
\node at (2.65,0.8) {$\tilde{\tau}$};
\draw (3.9,0) arc (0:90:1);
\draw (3.7,0) arc (0:90:0.8);
\draw (3.5,0) arc (0:90:0.6);
\draw (3.3,0) arc (0:90:0.4);
\end{scope}

\begin{scope}[shift={(8.6,0)}]
\draw[-, very thick,ggreen] (0,0)--(4.1,0);
\draw (1.2,1) arc (90:180:1);
\draw (1.2,0.8) arc (90:180:0.8);
\draw (1.2,0.6) arc (90:180:0.6);
\draw (1.2,0.4) arc (90:180:0.4);
\draw (1.2,0.4)--(1.8,0.4)--(2.3,0.8);
\draw (1.7,0.8)--(1.8,0.8)--(2.3,0.4);
\draw (1.7,0.6)--(1.8,0.6)--(2.3,1);
\draw (1.7,1)--(1.8,1)--(2.3,0.6);
\draw [draw=black] (1.2,1.05) rectangle (1.7,0.55);
\node at (1.45,0.8) {$\tilde{\sigma}$};
\draw (2.3,0.8)--(2.4,0.8);
\draw (2.3,0.6)--(2.4,0.6);
\draw (2.3,1)--(2.4,1);
\draw (2.3,0.4)--(2.9,0.4);
\draw [draw=black] (2.4,1.05) rectangle (2.9,0.55);
\node at (2.65,0.8) {$\tilde{\tau}$};
\draw (3.9,0) arc (0:90:1);
\draw (3.7,0) arc (0:90:0.8);
\draw (3.5,0) arc (0:90:0.6);
\draw (3.3,0) arc (0:90:0.4);
\end{scope}

\begin{scope}[shift={(4,-2)}]
\draw[-, very thick,ggreen] (0,0)--(4.1,0);
\draw (1.2,1) arc (90:180:1);
\draw (1.2,0.8) arc (90:180:0.8);
\draw (1.2,0.6) arc (90:180:0.6);
\draw (1.2,0.4) arc (90:180:0.4);
\draw (1.2,0.4)--(1.8,0.4)--(2.3,1);
\draw (1.7,0.8)--(1.8,0.8)--(2.3,0.4);
\draw (1.7,0.6)--(1.8,0.6)--(2.3,0.8);
\draw (1.7,1)--(1.8,1)--(2.3,0.6);
\draw [draw=black] (1.2,1.05) rectangle (1.7,0.55);
\node at (1.45,0.8) {$\tilde{\sigma}$};
\draw (2.3,0.8)--(2.4,0.8);
\draw (2.3,0.6)--(2.4,0.6);
\draw (2.3,1)--(2.4,1);
\draw (2.3,0.4)--(2.9,0.4);
\draw [draw=black] (2.4,1.05) rectangle (2.9,0.55);
\node at (2.65,0.8) {$\tilde{\tau}$};
\draw (3.9,0) arc (0:90:1);
\draw (3.7,0) arc (0:90:0.8);
\draw (3.5,0) arc (0:90:0.6);
\draw (3.3,0) arc (0:90:0.4);
\end{scope}

\begin{scope}[shift={(8.6,-2)}]
\draw[-, very thick,ggreen] (0,0)--(4.1,0);
\draw (1.2,1) arc (90:180:1);
\draw (1.2,0.8) arc (90:180:0.8);
\draw (1.2,0.6) arc (90:180:0.6);
\draw (1.2,0.4) arc (90:180:0.4);
\draw (1.2,0.4)--(1.8,0.4)--(2.3,1);
\draw (1.7,0.8)--(1.8,0.8)--(2.3,0.6);
\draw (1.7,0.6)--(1.8,0.6)--(2.3,0.8);
\draw (1.7,1)--(1.8,1)--(2.3,0.4);
\draw [draw=black] (1.2,1.05) rectangle (1.7,0.55);
\node at (1.45,0.8) {$\tilde{\sigma}$};
\draw (2.3,0.8)--(2.4,0.8);
\draw (2.3,0.6)--(2.4,0.6);
\draw (2.3,1)--(2.4,1);
\draw (2.3,0.4)--(2.9,0.4);
\draw [draw=black] (2.4,1.05) rectangle (2.9,0.55);
\node at (2.65,0.8) {$\tilde{\tau}$};
\draw (3.9,0) arc (0:90:1);
\draw (3.7,0) arc (0:90:0.8);
\draw (3.5,0) arc (0:90:0.6);
\draw (3.3,0) arc (0:90:0.4);
\end{scope}
\end{tikzpicture}
\end{center}
where $\tilde{\sigma}$ and $\tilde{\tau}$ are the permutations of $\mathfrak{S}_3$ moving the first two strings as they were one. On the other hand, we have that 
\begin{center}
\begin{tikzpicture}[scale=0.5]
\gri{2}{2} \ver{1}{1} \ver{1}{1.25} \hor{1}{1} \hor{1.25}{1}
\draw[red, very thick] (2-1,-1+0.25) arc (-90:0:0.5); \draw[red, very thick] (2-1, -1+0.5)--(2, -1+0.5); \draw[red, very thick] (2-0.5,-1+1)--(2-0.5,-1+0.75);
\draw[red, very thick] (1-1,-2+0.5) arc (-90:0:0.5); \draw[red, very thick] (1-0.25,-2+1)--(1-0.25, -2);
\begin{scope}[shift={(2.5,0)}]
\gri{2}{2} \ver{1}{1} \ver{1}{1.25}
\draw[red, very thick] (2-1,-1+0.5) arc (-90:0:0.5);  \draw[red, very thick] (0,-1+0.25)--(2-1, -1+0.25)--(2, -1+0.25);
\draw[red, very thick] (1-1,-2+0.5) arc (-90:0:0.5); \draw[red, very thick] (1-0.25,-2+1)--(1-0.25, -2);
\draw[red,very thick] (0,-1+0.5)--(2-1, -1+0.5);
\node at (2.2,-2) {,};
\end{scope}
\begin{scope}[shift={(5,0)}]
\gri{2}{2} \ver{1}{1.15} \ver{1}{0.85}
\draw[red,very thick] (0,-1+0.5)--(2-1, -1+0.5);
\draw[red, very thick] (1-1+0.15,-2+0.5) arc (-90:0:0.5); \draw[red,very thick] (1-0.65,-2+1)--(1-0.65, -2); \draw[red,very thick] (1-1+0.25,-2+0.5)--(1-1,-2+0.5);
\draw[red, very thick] (2-1,-1+0.5) arc (-90:0:0.5);  \draw[red, very thick] (0,-1+0.25)--(2-1, -1+0.25)--(2, -1+0.25);
\node at (2.2,-2) {,};
\end{scope}
\begin{scope}[shift={(7.5,0)}]
\gri{2}{2} \ver{1}{1.15} \ver{1}{0.85} \hor{1}{1} \hor{1.25}{1}
\draw[red, very thick] (2-1,-1+0.25) arc (-90:0:0.5); \draw[red, very thick] (2-1, -1+0.5)--(2, -1+0.5); \draw[red, very thick] (2-0.5,-1+1)--(2-0.5,-1+0.75);
\draw[red, very thick] (1-1+0.15,-2+0.5) arc (-90:0:0.5); \draw[red,very thick] (1-0.65,-2+1)--(1-0.65, -2); \draw[red,very thick] (1-1+0.25,-2+0.5)--(1-1,-2+0.5);
\end{scope}
\node at (2.2,-2) {,};
\node at (-0.5,-1) {\Bigg\{};
\node at (10,-1) {\Bigg\}};
\node at (-5,-1) {$\mathfrak{L}(M) = \{L_1,L_2,L_3,L_4 \} =$};
\end{tikzpicture}
\end{center}
and it is easy to see that \\ \\
\begin{minipage}{0,6\textwidth}
\begin{equation*}
\begin{split}
r_{4}^{\tilde{\gamma}_{2,2}(\sigma, L_1, \tau)} &= \begin{tikzpicture}
\draw[-, very thick,ggreen] (0,0)--(4.1,0);
\draw (1.2,1) arc (90:180:1);
\draw (1.2,0.8) arc (90:180:0.8);
\draw (1.2,0.6) arc (90:180:0.6);
\draw (1.2,0.4) arc (90:180:0.4);
\draw (1.2,0.4)--(1.8,0.4)--(2.3,0.8);
\draw (1.7,0.8)--(1.8,0.8)--(2.3,0.6);
\draw (1.7,0.6)--(1.8,0.6)--(2.3,1);
\draw (1.7,1)--(1.8,1)--(2.3,0.4);
\draw [draw=black] (1.2,1.05) rectangle (1.7,0.55);
\node at (1.45,0.8) {$\tilde{\sigma}$};
\draw (2.3,0.8)--(2.4,0.8);
\draw (2.3,0.6)--(2.4,0.6);
\draw (2.3,1)--(2.4,1);
\draw (2.3,0.4)--(2.9,0.4);
\draw [draw=black] (2.4,1.05) rectangle (2.9,0.55);
\node at (2.65,0.8) {$\tilde{\tau}$};
\draw (3.9,0) arc (0:90:1);
\draw (3.7,0) arc (0:90:0.8);
\draw (3.5,0) arc (0:90:0.6);
\draw (3.3,0) arc (0:90:0.4);
\end{tikzpicture} \\
r_{4}^{\tilde{\gamma}_{2,2}(\sigma, L_2, \tau)} &= \begin{tikzpicture}
\draw[-, very thick,ggreen] (0,0)--(4.1,0);
\draw (1.2,1) arc (90:180:1);
\draw (1.2,0.8) arc (90:180:0.8);
\draw (1.2,0.6) arc (90:180:0.6);
\draw (1.2,0.4) arc (90:180:0.4);
\draw (1.2,0.4)--(1.8,0.4)--(2.3,0.8);
\draw (1.7,0.8)--(1.8,0.8)--(2.3,0.4);
\draw (1.7,0.6)--(1.8,0.6)--(2.3,1);
\draw (1.7,1)--(1.8,1)--(2.3,0.6);
\draw [draw=black] (1.2,1.05) rectangle (1.7,0.55);
\node at (1.45,0.8) {$\tilde{\sigma}$};
\draw (2.3,0.8)--(2.4,0.8);
\draw (2.3,0.6)--(2.4,0.6);
\draw (2.3,1)--(2.4,1);
\draw (2.3,0.4)--(2.9,0.4);
\draw [draw=black] (2.4,1.05) rectangle (2.9,0.55);
\node at (2.65,0.8) {$\tilde{\tau}$};
\draw (3.9,0) arc (0:90:1);
\draw (3.7,0) arc (0:90:0.8);
\draw (3.5,0) arc (0:90:0.6);
\draw (3.3,0) arc (0:90:0.4);
\end{tikzpicture}
\end{split}
\end{equation*}
\end{minipage}
\begin{minipage}{0,5\textwidth}
\begin{equation*}
\begin{split}
r_{4}^{\tilde{\gamma}_{2,2}(\sigma, L_3, \tau)} &= \begin{tikzpicture}
\draw[-, very thick,ggreen] (0,0)--(4.1,0);
\draw (1.2,1) arc (90:180:1);
\draw (1.2,0.8) arc (90:180:0.8);
\draw (1.2,0.6) arc (90:180:0.6);
\draw (1.2,0.4) arc (90:180:0.4);
\draw (1.2,0.4)--(1.8,0.4)--(2.3,1);
\draw (1.7,0.8)--(1.8,0.8)--(2.3,0.4);
\draw (1.7,0.6)--(1.8,0.6)--(2.3,0.8);
\draw (1.7,1)--(1.8,1)--(2.3,0.6);
\draw [draw=black] (1.2,1.05) rectangle (1.7,0.55);
\node at (1.45,0.8) {$\tilde{\sigma}$};
\draw (2.3,0.8)--(2.4,0.8);
\draw (2.3,0.6)--(2.4,0.6);
\draw (2.3,1)--(2.4,1);
\draw (2.3,0.4)--(2.9,0.4);
\draw [draw=black] (2.4,1.05) rectangle (2.9,0.55);
\node at (2.65,0.8) {$\tilde{\tau}$};
\draw (3.9,0) arc (0:90:1);
\draw (3.7,0) arc (0:90:0.8);
\draw (3.5,0) arc (0:90:0.6);
\draw (3.3,0) arc (0:90:0.4);
\end{tikzpicture} \\
r_{4}^{\tilde{\gamma}_{2,2}(\sigma, L_4, \tau)}  &= \begin{tikzpicture}
\draw[-, very thick,ggreen] (0,0)--(4.1,0);
\draw (1.2,1) arc (90:180:1);
\draw (1.2,0.8) arc (90:180:0.8);
\draw (1.2,0.6) arc (90:180:0.6);
\draw (1.2,0.4) arc (90:180:0.4);
\draw (1.2,0.4)--(1.8,0.4)--(2.3,1);
\draw (1.7,0.8)--(1.8,0.8)--(2.3,0.6);
\draw (1.7,0.6)--(1.8,0.6)--(2.3,0.8);
\draw (1.7,1)--(1.8,1)--(2.3,0.4);
\draw [draw=black] (1.2,1.05) rectangle (1.7,0.55);
\node at (1.45,0.8) {$\tilde{\sigma}$};
\draw (2.3,0.8)--(2.4,0.8);
\draw (2.3,0.6)--(2.4,0.6);
\draw (2.3,1)--(2.4,1);
\draw (2.3,0.4)--(2.9,0.4);
\draw [draw=black] (2.4,1.05) rectangle (2.9,0.55);
\node at (2.65,0.8) {$\tilde{\tau}$};
\draw (3.9,0) arc (0:90:1);
\draw (3.7,0) arc (0:90:0.8);
\draw (3.5,0) arc (0:90:0.6);
\draw (3.3,0) arc (0:90:0.4);
\end{tikzpicture} 
\end{split}
\end{equation*}
\end{minipage}
\vspace{0.5cm}
where the permutation in the middle of $r_{4}^{\tilde{\gamma}_{2,2}(\sigma, L_i, \tau)}$ is exactly $\gamma_{2,2}(L_i)$. We therefore get 
\[  \pi_1^{*(2)} \circ ((\sigma \ten \id_{[V_1]}) \circ \varphi_{2,2}(M) \circ (\tau \ten \id_{[V_1]})) \circ \pi_1^{(2)} = r_4^{\tilde{\gamma}_{2,2}(\sigma, L_1, \tau)} - r_4^{\tilde{\gamma}_{2,2}(\sigma, L_2, \tau)} + r_4^{\tilde{\gamma}_{2,2}(\sigma, L_3, \tau)} - r_4^{\tilde{\gamma}_{2,2}(\sigma, L_4, \tau)} \]
as is claimed in the statement.
\end{example}
\begin{proof}
Let $M \in \mosaic{n}{m}$ and let $\varphi_{n,m}(M)$ be the associated morphism in $\DY^1$. Applying relations \eqref{eq:actionrule}, \eqref{eq:coactionrule} we get  \\ \\
\begin{tikzpicture}
\draw[-, very thick,ggreen] (0,0)--(2,0);
\draw[-, very thick,ggreen] (3.8+1,0)--(5.9+1,0);
\draw[very thick] (1,0.75) arc (90:180:0.75);
\draw[very thick] (1,0.75)--(1.2,0.75);
\node[shape=circle,draw,inner sep=1pt] (char) at (1.4,0.75) {$\sigma$}; 
\draw[very thick] (1.6,0.75)--(2,0.75);
\node at (0.3,-0.3) {$n$};
\begin{scope}[shift={(0.1,0)}]
\draw[very thick] (3.8+1,0.75)--(4.2+1,0.75);
\node[shape=circle,draw,inner sep=1pt] (char) at (4.4+1,0.75) {$\tau$}; 
\draw[very thick] (4.6+1,0.75)--(4.8+1,0.75);
\draw[very thick]  (5.55+1,0) arc (0:90:0.75);
\node at (5.55+1,-0.3) {$m$};
\node at (3.6+1,0.75) {$m$};
\end{scope}
\node at (2.2,0.75) {$n$};
\node at (3.4,0.4) {\Bigg( $\varphi_{n,m}(M)$ \Bigg) };
\node at (7.2,0.4) {$=$};
\begin{scope}[shift={(-0.5,0)}]
\node at (-2.9+12,0.33) {$\mathlarger{\sum_{L \in \mathfrak{L}(M)} (-1)^{\countingone(L)}}$};
\begin{scope}[shift={(-0.5,0)}]
\draw[-, very thick,ggreen] (10.5,0)--(15.1,0);
\draw[very thick] (11.5,0.85) arc (90:180:0.85);
\draw[very thick] (11.8,0.45) arc (90:180:0.45);
\draw[very thick] (11.8,0.45)--(12.3,0.45);
\draw [draw=black] (11.5,1) rectangle (12,0.6);
\node at (11.75,0.8) {$\tilde{\sigma}$};
\draw[very thick] (12,0.85)--(12.3,0.85);
\draw [draw=black] (12.3,1) rectangle (13.3,0.3);
\node at (12.8,0.65) {\tiny $\gamma_{n,m}(L)$};
\draw[very thick] (13.3,0.45)--(13.8,0.45);
\draw[very thick] (13.3,0.85)--(13.6,0.85);
\draw [draw=black] (13.6,1) rectangle (14.1,0.6);
\node at (13.85,0.8) {$\tilde{\tau}$};
\draw[very thick]  (14.95,0) arc (0:90:0.85);
\draw[very thick]  (14.25,0) arc (0:90:0.45);
\node at (10.7,-0.3) {$\Lambda$};
\node at (15,-0.3) {$\Omega$};
\node at (11.4,-0.3) {$\Gamma$};
\node at (14.3,-0.3) {$\Theta$};
\end{scope}
\end{scope}
\end{tikzpicture}
\\ \\
where 
\[ \Lambda = \sum_{i=1}^n l_{i1} , \qquad \Omega = \sum_{j=1}^{m} t_{1j}, \qquad \Gamma = n+m-\Lambda, \qquad \Theta = n+m-\Omega .\]
To end the proof it suffices to note that $\tilde{\gamma}_{n,m}(\sigma,L,\tau) =(\tilde{\sigma} \ten \id_{[\Gamma]}) \circ \gamma_{n,m}(L) \circ  (\tilde{\tau} \ten \id_{[\Theta]})$.
\end{proof}
We now give an explicit formula for the multiplication of $\mathfrak{U}_{\DY}^1$.
\begin{theorem}
\label{maintheorem}
We have 
\begin{equation}
\label{eq:formulamultiplication}
r_{n}^{\sigma} \circ r_{m}^{\tau} =
\sum_{L \in \loom{n}{m}} (-1)^{\countingtwo(L)} r_{n+m}^{\tilde{\gamma}_{n,m}(\sigma, L, \tau)}
\end{equation}
\end{theorem}
\begin{proof}
The proof is pictorial. We have
\begin{center}
\begin{tikzpicture} \draw[-, very thick,ggreen] (0,0)--(5.8,0);
\draw[very thick]  (1,0.75) arc (90:180:0.75);
\draw[very thick] (1,0.75)--(1.2,0.75);
\node[shape=circle,draw,inner sep=1pt] (char) at (1.4,0.75) {$\sigma$}; 
\draw[very thick] (1.6,0.75)--(1.8,0.75);
\draw[very thick]  (2.55,0) arc (0:90:0.75);
\node at (2.55,-0.3) {$n$};
\node at (0.3,-0.3) {$n$};
\draw[very thick]  (4,0.75) arc (90:180:0.75);
\draw[very thick] (4,0.75)--(4.2,0.75);
\node[shape=circle,draw,inner sep=1pt] (char) at (4.4,0.75) {$\tau$}; 
\draw[very thick] (4.6,0.75)--(4.8,0.75);
\draw[very thick]  (5.55,0) arc (0:90:0.75);
\node at (5.55,-0.3) {$m$};
\node at (3.3,-0.3) {$m$};
\node at (6.25,0.33) {$=$};
\begin{scope}[shift={(-6,-2)}]
\begin{scope}[shift={(-0.2,0)}]
\node at (6.45,0.33) {$=$};
\node at (-1.5+8+1.5,0.2) {$\mathlarger{\sum_{M \in \mosaic{n}{m}} (-1)^{\cob(M)} }$};
\begin{scope}[shift={(-0.2,0)}]
\draw[-, very thick,ggreen] (0+8+1.5,0)--(2+8+1.5,0);
\draw[-, very thick,ggreen] (3.8+1+8+1.5,0)--(5.9+1+8+1.5,0);
\draw[very thick]  (1+8+1.5,0.75) arc (90:180:0.75);
\draw[very thick] (1+8+1.5,0.75)--(1.2+8+1.5,0.75);
\node[shape=circle,draw,inner sep=1pt] (char) at (1.4+8+1.5,0.75) {$\sigma$}; 
\draw[very thick] (1.6+8+1.5,0.75)--(2+8+1.5,0.75);
\node at (0.3+8+1.5,-0.3) {$n$};
\begin{scope}[shift={(0.1,0)}]
\draw[very thick] (3.8+1+8+1.5,0.75)--(4.2+1+8+1.5,0.75);
\node[shape=circle,draw,inner sep=1pt] (char) at (4.4+1+8+1.5,0.75) {$\tau$}; 
\draw[very thick] (4.6+1+8+1.5,0.75)--(4.8+1+8+1.5,0.75);
\draw[very thick]  (5.55+1+8+1.5,0) arc (0:90:0.75);
\node at (5.55+1+8+1.5,-0.3) {$m$};
\node at (3.6+1+8+1.5,0.75) {$m$};
\end{scope}
\node at (2.2+8+1.5,0.75) {$n$};
\node at (3.4+8+1.5,0.4) {\Bigg( $\varphi_{n,m}(M)$ \Bigg)};
\end{scope}
\end{scope}
\node at (6.25,-1.4) {$=$};
\node at (9.3,-1.6) {$\mathlarger{\sum_{M \in \mosaic{n}{m}}(-1)^{\cob(M)}  \sum_{L \in \mathfrak{L}(M)} (-1)^{\countingone(L)}}$};
\begin{scope}[shift={(-0.4,0)}]
\draw[-, very thick,ggreen] (0+12+0.5+0.1,-1.5-0.2)--(2.8+0.25+0.25+12+1.1,0-1.5-0.2);
\draw[very thick]  (1+12+0.5,0.75-1.5-0.2) arc (90:180:0.75);
\draw[very thick] (1+12+0.5,0.75-1.5-0.2)--(1.2+12+0.5,0.75-1.5-0.2);
\begin{scope}[shift={(0.3,0)}]
\node[shape=rectangle,draw,inner sep=1pt] (char) at (1.65+12+0.5+0.05,0.75-1.5-0.2) {\tiny $\tilde{\gamma}_{n,m}(\sigma,L,\tau)$}; 
\end{scope}
\begin{scope}[shift={(0.6,0)}]
\draw[very thick] (1.6+0.25+0.25+12+0.5+0.1,0.75-1.5-0.2)--(1.8+0.25+0.25+12+0.5+0.1,0.75-1.5-0.2);
\draw[very thick]  (2.55+0.25+0.25+12+0.5+0.1,0-1.5-0.2) arc (0:90:0.75);
\node at (2.55+0.25+0.25+12+0.5+0.1,-0.3-1.5-0.2) {$n\!+\!m$};
\end{scope}
\node at (0.3+12+0.5+0.1,-0.3-1.5-0.2) {$n\!+\!m$};
\end{scope}
\node at (6.25,-3) {$=$};
\node at (7.9,-3.2) {$\mathlarger{\sum_{L \in \loom{n}{m}}(-1)^{\countingtwo(L) }}$};
\draw[-, very thick,ggreen] (0+12+0.5-3.1+0.1,-1.5-0.2-1.7)--(2.8+0.25+0.25+12+1.1-3.1,0-1.5-0.2-1.7);
\draw[very thick]  (1+12+0.5-3.1,0.75-1.5-0.2-1.7) arc (90:180:0.75);
\draw[very thick] (1+12+0.5-3.1,0.75-1.5-0.2-1.7)--(1.2+12+0.5-3.1,0.75-1.5-0.2-1.7);
\begin{scope}[shift={(0.3,0)}]
\node[shape=rectangle,draw,inner sep=1pt] (char) at (1.65+12+0.5-3.1+0.05,0.75-1.5-0.2-1.7) {\tiny $\tilde{\gamma}_{n,m}(\sigma,L,\tau)$}; 
\end{scope}
\begin{scope}[shift={(0.6,0)}]
\draw[very thick] (1.6+0.25+0.25+12+0.5-3.1+0.1,0.75-1.5-0.2-1.7)--(1.8+0.25+0.25+12+0.5-3.1+0.1,0.75-1.5-0.2-1.7);
\draw[very thick]  (2.55+0.25+0.25+12+0.5-3.1+0.1,0-1.5-0.2-1.7) arc (0:90:0.75);
\node at (2.55+0.25+0.25+12+0.5-3.1+0.1,-0.3-1.5-0.2-1.7) {$n\!+\!m$};
\end{scope}
\node at (0.3+12+0.5-3.1+0.1,-0.3-1.5-0.2-1.7) {$n \!+\! m$};
\end{scope}
\end{tikzpicture}
\end{center}
where the first equality follows by Proposition \ref{claim}, the second equality follows by Lemma \ref{lemmamultiplication} and the third equality follows by Proposition \ref{proposition-partition-drinfeld-yetter-tableaux} and Proposition \ref{remarksigns}.
\end{proof}
Theorem \ref{maintheorem} gives an explicit formula for the product of $\mathfrak{U}_\DY^1$ with respect to the standard basis \eqref{eq:canonical-basis}. However, this does not provide a formula in terms of symmetric groups, i.e. a formula of the following form 
\[ r_{n}^\sigma \circ r_m^\tau  = \sum_{\pi \in \mathfrak{S}_{n+m}} c_{\sigma,\tau}^\pi r_{n+m}^\pi.\]
In the next Section we propose an approach to find such a formula consisting in determine a subset of $\loom{n}{m}$ depending on $\sigma$ and $\tau$ (namely a set of \emph{($\sigma$--$\tau$)--essential Drinfeld--Yetter looms}) satisfying the property of being a minimal set necessary to describe the product $r_{n}^\sigma \circ r_m^\tau$ through the formula \eqref{eq:formulamultiplication}.
\subsection{Essential Drinfeld--Yetter looms}
It follows by Theorem \ref{maintheorem} that the number of summands appearing in the multiplication $r_n^\sigma \circ r_m^\tau$ does not depend on $\sigma$ and $\tau$, but only on $n$ and $m$, and it is equal to $|\loom{n}{m}|$. However, as conjectured in \ref{conjecture-cardinality-looms}, we have that $|\loom{n}{m}|<< (n+m)!$. This means that there are several Drinfeld--Yetter looms that do not give any contribution to the sum \eqref{eq:formulamultiplication}; we shall call such Drinfeld--Yetter looms \index{negligible Drinfeld--Yetter loom}\emph{negligible}. In order to formalize this definition, let $\sigma \in \mathfrak{S}_n$, $\tau \in \mathfrak{S}_m$ and $\pi \in \mathfrak{S}_{n+m}$. To such permutations we associate the set 
\[ \Gamma_{n,m}^{\sigma,\tau,\pi} \coloneqq \{ L \in \loom{n}{m} \ | \ \tilde{\gamma}_{n,m}(\sigma,L,\tau) = \pi\}.\]
Consider also the following non--negative integers
\begin{equation*}
\begin{split}
P_{n,m}^{\sigma,\tau,\pi} & \coloneqq \# \{ L \in \loom{n}{m} \ | \ \tilde{\gamma}_{n,m}(\sigma,L,\tau)= \pi \ , (-1)^{\countingtwo(L)}=1\} \\
N_{n,m}^{\sigma,\tau,\pi} & \coloneqq \# \{ L \in \loom{n}{m} \ | \ \tilde{\gamma}_{n,m}(\sigma,L,\tau)= \pi \ , (-1)^{\countingtwo(L)}=-1\} 
\end{split}
\end{equation*} 
which we call the number of positive (resp. negative) Drinfeld--Yetter looms $(\sigma,\tau)$--associated to the permutation $\pi$.
It is clear that $\{ \Gamma_{n,m}^{\sigma,\tau,\pi}\}_{\pi \in \mathfrak{S}_{n+m}}$ defines a partition of $\mathfrak{S}_{n+m}$ (with eventually some empty blocks). It is also clear that 
\[ |\Gamma_{n,m}^{\sigma,\tau\pi} | = P_{n,m}^{\sigma,\tau\pi}   + N_{n,m}^{\sigma,\tau\pi} . \]

\begin{definition}
We say that $L_1,L_2 \in \loom{n}{m}$ is a pair of $(\sigma,\tau)$--negligible Drinfeld--Yetter looms (and we write $\{ L_1,L_2\} \in \Neg_{n,m}^{\sigma,\tau}$) if 
\[\tilde{\gamma}_{n,m}(\sigma,L_1,\tau)=\tilde{\gamma}_{n,m}(\sigma,L_2,\tau) \qquad \text{and} \qquad (-1)^{\countingtwo(L_1)}=-(-1)^{\countingtwo(L_2)}.\]
\end{definition}
\begin{example}
\label{example-negligible-looms}
The set of $1 \times 1$ Drinfeld--Yetter looms is
\begin{center}
\begin{tikzpicture}[scale=0.6]
\node at (0.15,0) {$\loom{1}{1}= \Big\{$};
\draw[black,very thick] (1.5,0.5) rectangle (2.5,-0.5);
\cro{0.5}{2.5}
\node at (2.75,-0.5) {,};
\draw[black,very thick] (3,0.5) rectangle (4,-0.5);
\draw[red, very thick] (4-1,-0.5+0.5) arc (-90:0:0.5); \draw[red, very thick] (4-0.35,-0.5+1)--(4-0.35, -0.5);
\node at (4.25,-0.5) {,};
\draw[black,very thick] (4.5,0.5) rectangle (5.5,-0.5);
\draw[red, very thick] (5.5-1+0.25,-0.5+0.5) arc (-90:0:0.5);  \draw[red,very thick] (5.5-0.75,-0.5+1)--(5.5-0.75, -0.5); \draw[red,very thick] (5.5-1+0.25,-0.5+0.5)--(5.5-1,-0.5+0.5);
\node at (5.75,-0.5) {,};
\draw[black,very thick] (6,0.5) rectangle (7,-0.5);
\draw[red, very thick] (7-1,-0.5+0.5) arc (-90:0:0.5); \draw[red, very thick] (7-1, -0.5+0.35)--(7, -0.5+0.35);
\node at (7.25,-0.5) {,};
\draw[black,very thick] (7.5,0.5) rectangle (8.5,-0.5);
\draw[red, very thick] (8.5-1,-0.5+0.25) arc (-90:0:0.5); \draw[red, very thick] (8.5-1, -0.5+0.75)--(8.5, -0.5+0.75); \draw[red, very thick] (8.5-0.5,-0.5+1)--(8.5-0.5,-0.5+0.75);
\node at (9,-0) {\Big\}.};
\end{tikzpicture}
\end{center}
We therefore have the following two (non--disjoint) sets of negligible $1 \times 1$ Drinfeld--Yetter looms:
\begin{center}
\begin{tikzpicture}[scale=0.6]
\node at (1,0) {$ \Big\{$};
\draw[black,very thick] (1.5,0.5) rectangle (2.5,-0.5);
\begin{scope}[shift={(2.5,1.5)}]
\cro{2}{0}
\end{scope}
\node at (2.75,-0.5) {,};
\draw[black,very thick] (3,0.5) rectangle (4,-0.5);
\begin{scope}[shift= {(3,0.5)}]
\draw[red, very thick] (1-1,-1+0.25) arc (-90:0:0.5); \draw[red, very thick] (1-1, -1+0.75)--(1, -1+0.75);  \draw[red, very thick] (1-0.5,-1+1)--(1-0.5,-1+0.75);
 \end{scope}
\node at (4.5,-0) {\Big\}};
\node at (5.75,0) {and};
\begin{scope}[shift= {(6,0)}]
\node at (1,0) {$ \Big\{$};
\draw[black,very thick] (1.5,0.5) rectangle (2.5,-0.5);
\begin{scope}[shift={(2.5,1.5)}]
\cro{2}{0}
\end{scope}
\node at (2.75,-0.5) {,};
\draw[black,very thick] (3,0.5) rectangle (4,-0.5);
\begin{scope}[shift={(-1.5,0)}]
\draw[red, very thick] (5.5-1+0.25,-0.5+0.5) arc (-90:0:0.5);  \draw[red,very thick] (5.5-0.75,-0.5+1)--(5.5-0.75, -0.5); \draw[red,very thick] (5.5-1+0.25,-0.5+0.5)--(5.5-1,-0.5+0.5);
\end{scope}
\node at (4.5,-0) {\Big\}};
\end{scope}
\end{tikzpicture}
\end{center}
which both refer to the permutation $\pi=(12) \in \mathfrak{S}_2$.
\end{example}
Fix $n,m \geqslant 1$, $\sigma \in \mathfrak{S}_n$ and $\tau \in \mathfrak{S}_m$ and consider the partition $\{ \Gamma_{n,m}^{\sigma,\tau,\pi} \}_{\pi \in \mathfrak{S}_{n+m}}$ of $\loom{n}{m}$. For any (non--empty) block $\Gamma_{n,m}^{\sigma,\tau,\pi}$, choose an ordering of its elements 
\begin{equation}
\label{eq:ordering-essential-looms}
 \Gamma_{n,m}^{\sigma,\tau,\pi}=\{ L_1, \ldots, L_{P_{n,m}^{\sigma,\tau,\pi}}, L_{P_{n,m}^{\sigma,\tau,\pi}+1}, \ldots , L_{P_{n,m}^{\sigma,\tau,\pi} + N_{n,m}^{\sigma,\tau,\pi}}\} 
\end{equation}
such that the first $P_{n,m}^{\sigma,\tau,\pi}$ elements are positive Drinfeld--Yetter looms and the last $N_{n,m}^{\sigma,\tau,\pi}$ elements are negative Drinfeld--Yetter looms. Next, take $M_{n,m}^{\sigma,\tau,\pi}= \min \{P_{n,m}^{\sigma,\tau,\pi}, N_{n,m}^{\sigma,\tau,\pi} \}$. It is then clear that all the pairs $\{L_i, L_{P_{n,m}^{\sigma,\tau,\pi} + N_{n,m}^{\sigma,\tau,\pi}-i+1} \}_{i=1,\ldots, M_{n,m}^{\sigma,\tau,\pi}}$ are disjoint elements of $\Neg_{n,m}^{\sigma,\tau}$. By removing all of these pairs from $\Gamma_{n,m}^{\sigma,\tau,\pi}$, we obtain 
\[ \Gamma_{n,m}^{\sigma,\tau,\pi,ess} \coloneqq \Gamma_{n,m}^{\sigma,\tau,\pi} \setminus \bigcup_{i=1}^{M_{n,m}^{\sigma,\tau ,\pi}} \{ L_i, L_{P_{n,m}^{\sigma,\tau,\pi} + N_{n,m}^{\sigma,\tau,\pi}-i+1} \}.\] 
Finally, we define 
\[ \loom{n}{m}^{\sigma,\tau,ess} \coloneqq \bigcup_{\pi \in \mathfrak{S}_{n+m}} \Gamma_{n,m}^{\sigma,\tau,\pi,ess}  \subset \loom{n}{m} \]
and we call it a set of $(\sigma,\tau)$--essential Drinfeld--Yetter looms\index{essential Drinfeld--Yetter loom}. Note that the construction of $\loom{n}{m}^{\sigma,\tau,ess} $ strongly depends on the choice of an ordering as in \eqref{eq:ordering-essential-looms}.
\begin{example}
We have the following two sets of essential $1 \times 1$ Drinfeld--Yetter looms (which are related to Example \ref{example-negligible-looms}):
\begin{center}
\begin{tikzpicture}[scale=0.6]
\node at (1,0) {\Big\{};
\draw[black,very thick] (1.5,0.5) rectangle (2.5,-0.5);
\begin{scope}[shift= {(1.5,0.5)}]
\draw[red, very thick] (1-1,-1+0.5) arc (-90:0:0.5); 
\draw[red, very thick] (1-0.35,-1+1)--(1-0.35, -1);
\end{scope}
\node at (2.75,-0.5) {,};
\draw[black,very thick] (3,0.5) rectangle (4,-0.5);
\begin{scope}[shift= {(3,0.5)}]
\draw[red, very thick] (1-1+0.25,-1+0.5) arc (-90:0:0.5);  \draw[red,very thick] (1-0.75,-1+1)--(1-0.75, -1); \draw[red,very thick]
 (1-1+0.25,-1+0.5)--(1-1,-1+0.5);
 \end{scope}
\node at (4.25,-0.5) {,};
\draw[black,very thick] (4.5,0.5) rectangle (5.5,-0.5);
\begin{scope}[shift= {(4.5,0.5)}]
\draw[red, very thick] (1-1,-1+0.5) arc (-90:0:0.5);  \draw[red, very thick] (1-1, -1+0.35)--(1, -1+0.35);
\end{scope}
\node at (6,-0) {\Big\}};
\node at (7.5,0) {and};
\begin{scope}[shift={(8,0)}]
\node at (1,0) {$ \Big\{$};
\draw[black,very thick] (1.5,0.5) rectangle (2.5,-0.5);
\begin{scope}[shift= {(1.5,0.5)}]
\draw[red, very thick] (1-1,-1+0.5) arc (-90:0:0.5); 
\draw[red, very thick] (1-0.35,-1+1)--(1-0.35, -1);
\end{scope}
\node at (2.75,-0.5) {,};
\draw[black,very thick] (3,0.5) rectangle (4,-0.5);
\begin{scope}[shift= {(3,0.5)}]
\draw[red, very thick] (1-1,-1+0.25) arc (-90:0:0.5); \draw[red, very thick] (1-1, -1+0.75)--(1, -1+0.75);  \draw[red, very thick] (1-0.5,-1+1)--(1-0.5,-1+0.75);
 \end{scope}
\node at (4.25,-0.5) {,};
\draw[black,very thick] (4.5,0.5) rectangle (5.5,-0.5);
\begin{scope}[shift= {(4.5,0.5)}]
\draw[red, very thick] (1-1,-1+0.5) arc (-90:0:0.5);  \draw[red, very thick] (1-1, -1+0.35)--(1, -1+0.35);
\end{scope}
\node at (6,-0) {\Big\}.};
\end{scope}
\end{tikzpicture}
\end{center}
\end{example}
 A set of essential Drinfeld--Yetter looms is a minimal set of Drinfeld--Yetter looms necessary to describe the multiplication of $\mathfrak{U}_\DY^1$, as is stated in the following
\begin{proposition}
\label{properties-essential-looms}
Let $n,m \geqslant 1$, $\sigma \in \mathfrak{S}_n$ and $\tau \in \mathfrak{S}_m$. For any $\pi \in \mathfrak{S}_{n+m}$ choose an ordering of the elements of $\Gamma_{n,m}^{\sigma, \tau,\pi}$ as in \ref{eq:ordering-essential-looms} and let $\loom{n}{m}^{\sigma, \tau, ess}$ be the associated set of essential Drinfeld--Yetter looms. Then
\begin{itemize}
\item[(i)] For any $\pi \in \mathfrak{S}_{n+m}$, the set $ \Gamma_{n,m}^{\sigma, \tau, \pi,ess}$ is made of Drinfeld--Yetter looms having all the same sign (i.e. they are all either positive or negative). 
\item[(ii)] The following formula holds
\begin{equation*}
r_n^\sigma \circ r_m^\tau = \sum_{L \in \loom{n}{m}^{\sigma, \tau, ess}} (-1)^{\countingtwo(L)} r_{n+m}^{\tilde{\gamma}_{n,m}(\sigma,L,\tau)}.
\end{equation*}
\end{itemize}
\end{proposition}
\begin{proof}
We observe that $(i)$ follows directly by the construction of $ \Gamma_{n,m}^{\sigma, \tau, \pi,ess}$. We have
\begin{equation*}
\begin{split}
r_n^\sigma \circ r_m^\tau &= \sum_{L \in \loom{n}{m}} (-1)^{\countingtwo(L)} r_{n+m}^{\tilde{\gamma}_{n,m}(\sigma,L,\tau)} \\
&= \sum_{\pi \in \mathfrak{S}_{n+m}} \ \sum_{L \in \Gamma_{n,m}^{\sigma,\tau,\pi}} (-1)^{\countingtwo(L)} r_{n+m}^{\pi} \\
&= \sum_{\pi \in \mathfrak{S}_{n+m}} \ \sum_{i=1}^{P_{n,m}^{\sigma,\tau,\pi}+N_{n,m}^{\sigma,\tau,\pi}} (-1)^{\countingtwo(L_i)} r_{n+m}^{\pi} \\
&= \sum_{\pi \in \mathfrak{S}_{n+m}} \ \sum_{i=M_{n,m}^{\sigma,\tau,\pi}+1}^{P_{n,m}^{\sigma,\tau,\pi}+N_{n,m}^{\sigma,\tau,\pi}-M_{n,m}^{\sigma,\tau,\pi}} (-1)^{\countingtwo(L_i)} r_{n+m}^\pi \\
&= \sum_{\pi \in \mathfrak{S}_{n+m}} \ \sum_{L \in \Gamma_{n,m}^{\sigma,\tau,\pi,ess}} (-1)^{\countingtwo(L)} r_{n+m}^{\pi} \\
&= \sum_{L \in \loom{n}{m}^{\sigma,\tau,ess}} (-1)^{\countingtwo(L)} r_{n+m}^{\tilde{\gamma}_{n,m}(\sigma,L,\tau)}
\end{split}
\end{equation*}
where the second and the last equality follow from the fact that the collection $\{ \Gamma_{n,m}^{\sigma,\tau,\pi,ess} \}_{\pi \in \mathfrak{S}_{n+m}}$ defines a partition of the set $\loom{n}{m}^{\sigma,\tau,ess}$.
\end{proof}
\begin{corollary}\label{maincor}
Under the previous notations, we have
\[ r_n^\sigma \circ r_m^\tau = \sum_{\pi \in \mathfrak{S}_{n+m}} (P_{n,m}^{\sigma,\tau,\pi} - N_{n,m}^{\sigma,\tau,\pi}) r_{n+m}^\pi.\]
\end{corollary}
Therefore, finding a complete description of the set $\loom{n}{m}^{\sigma,\tau, ess}$ would directly give an explicit formula for $\circ$ in terms of symmetric groups. However, finding such a description appear to be a very challenging problem, due to the \emph{elusive} nature of the set $\loom{n}{m}$. 
\appendix
\section{Explicit computations and conjectures}
\label{Appendix}
In this Appendix we collect some explicit computations and conjectures related to the set $\loom{n}{m}$ and to the algebra $\mathfrak{U}_\DY^1$.
\subsection{Explicit computations}
The proof of all the following statements can be find in \cite{rivezzithesis}.
\begin{proposition}
Let $H_{n,m} = |\loom{n}{m}|$. Then we have the following two recursive formulas for $H_{n,m}$:
\begin{equation}
\label{eq:recursive-formula-looms-one}
H_{n,m} = 2 H_{n-1,m} + (2n+1) H_{n,m-1} + \sum_{k=1}^{n-1} \binom{n-1}{k} 2^k H_{n-k,m-1} +  \sum_{k=1}^{n-2}\sum_{\ell=1}^k \binom{k}{\ell} 2^{\ell+1}  H_{n-\ell,m-1} 
\end{equation} 
and 
\begin{equation}
\label{eq:recursive-formula-looms-two}
H_{n,m} = 2 H_{n,m-1} + (2m+1) H_{n-1,m} + \sum_{k=1}^{m-1} \binom{m-1}{k} 2^k H_{n-1,m-k} +  \sum_{k=1}^{m-2}\sum_{\ell=1}^k \binom{k}{\ell}  2^{\ell+1} H_{n-1,m-\ell} .
\end{equation} 
\end{proposition}
\begin{proposition}
\label{identitymultiplication}
For any $n \geqslant 1$, we have 
\begin{equation}
\label{eq:identity-relation-one}
 r_{n}^{\id} \circ r_{1}^{\id} = (n+1) r_{n+1}^{\id} - \sum_{i=1}^n r_{n+1}^{(i,i+1)}
\end{equation}
and
\begin{equation}
\label{eq:identity-relation-two}
 r_{1}^{\id} \circ r_{n}^{\id} = (n+1) r_{n+1}^{\id} - \sum_{i=1}^n r_{n+1}^{(i,i+1)}
\end{equation}
where $(i,i+1)$ denotes the permutation that swaps $i$ and $i+1$ and fixes all the other elements.
\end{proposition}
\begin{proposition}
\label{proposition-coefficient-identity}
Let $c_{n,m}$ be the coefficient of $\id_{n+m}$ with respect to the multiplication $r_n^\id \circ r_m^\id$. Then 
\[ c_{n,m} = \frac{(n+m)!}{n!m!}.\]
\end{proposition}
\begin{proposition}
Let $n,m \geqslant 0 $ such that $n+m \leqslant 4$ and let $\sigma \in \mathfrak{S}_n$ and $\tau \in \mathfrak{S}_m$. Then 
\[ r_{n}^{\sigma} \circ r_{m}^{\tau} = r_{m}^{\tau} \circ r_{n}^{\sigma}.\]
\end{proposition}
\begin{proposition}
Let $\star$ be the multiplication in $\mathfrak{U}_\DY^1$ given by $r_n^\sigma \star r_m^\tau = r_{n+m}^{\sigma \ten \tau}$, which is pictorially represented by the incapsulation of $r_m^\tau$ inside $r_n^\sigma$:
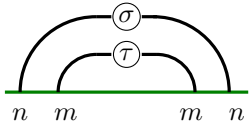
\begin{center}
\begin{tikzpicture}
\draw[-, very thick,ggreen] (-0.2,0)--(3,0);
\draw[very thick]  (1,1) arc (90:180:1);
\draw[very thick]  (1,0.5) arc (90:180:0.5);
\draw[very thick] (1,1)--(1.2,1);
\draw[very thick] (1,0.5)--(1.2,0.5);
\node[shape=circle,draw,inner sep=1pt] (char) at (1.4,1) {$\sigma$}; 
\node[shape=circle,draw,inner sep=1pt] (char) at (1.4,0.5) {$\tau$}; 
\draw[very thick] (1.6,1)--(1.8,1);
\draw[very thick] (1.6,0.5)--(1.8,0.5);
\draw[very thick]  (2.75,0) arc (0:90:1);
\draw[very thick]  (2.30,0) arc (0:90:0.5);
\node at (2.85,-0.3) {$n$};
\node at (0,-0.3) {$n$};
\node at (2.25,-0.3) {$m$};
\node at (0.6,-0.3) {$m$};
\end{tikzpicture}
\end{center}
Then for any $n,m \geqslant 0$ and $\tau \in \mathfrak{S}_m$ we have
\[ r_1^\id \circ (r_n^\id \star r_m^\tau) = n \cdot r_{n+1}^\id \star r_m^\tau + r_n^\id \star (r_1^\id \circ r_m^\tau) - \sum_{k=1}^n r_{n+1}^{(k,k+1)} \star r_m^\tau. \]
\end{proposition}

\begin{proposition}
Let $n \geqslant 0$. Then, for any $2 \leqslant k  \leqslant n $ one has
\begin{equation*}
\begin{split}
r_n^{(1,k)} \circ r_1^\id &= (k-2) r_{n+1}^{(1,k+1)} + (n-k+1)r_{n+1}^{(1,k)} - \sum_{i=2}^{k-1} r_{n+1}^{(1,k+1)(i,i+1)} \\
& - \sum_{i=k+1}^{n} r_{n+1}^{(1,k)(i,i+1)} + r_{n+1}^{(1,k+1,2,3, \ldots , k-1,k)} - r_{n+1}^{(1,k+1)(2,3,\ldots, k)} \\
& + r_{n+1}^{(1,k,k-1,k-2, \ldots, 3,2,k+1)} - r_{n+1}^{(1,k+1)(2,k,k-1,\ldots, 3,2)}.
\end{split}
\end{equation*}
\end{proposition}

\subsection{Conjectures}
\label{section-conjectures}
\begin{conjecture}
Let $c_{n,m}$ as in Proposition \ref{proposition-coefficient-identity}. Then $c_{n,m}$ is the dominant term of $r_n^\id \circ r_m^\id$, i.e. the biggest in absolute value.
\end{conjecture}
\begin{conjecture}
Let $\pi_n \in \mathfrak{S}_n$ defined by 
\[ \pi_n \coloneqq \prod_{i=1}^{\lfloor \frac{n}{2}\rfloor } (i,n-i+1) \in \mathfrak{S}_n.\] Then the following identity holds in $\mathfrak{U}_\DY^1$:
\[ r_1^\id \circ r_n^{\pi_n} = -n \cdot r_{n+1}^{\pi_{n+1}} + r_{n}^{\pi_n} \star r_1^\id + \sum_{i=1}^n r_{n+1}^{\pi_{n+1} \circ (i,i+1)}. \]
Furthermore, let $p_{n,m}$ be the coefficient of $r_{n+m}^{\pi_{n+m}}$ with respect to the multiplication $r_{n}^\id \circ r_m^{\pi_m}$. Then $p_{n,m}$ is the dominant coefficient of $r_1^\id \circ r_n^{\pi_n}$ (i.e. is the biggest in absolute value). Moreover, the $p_{n,m}$'s satisfy the recursive formula $p_{n,m} = (-1)^n (|p_{n-1,m}| + |p_{n,m-1}|)$
with initial conditions $p_{1,1}=-1$, $p_{1,m} = -m$ and $p_{n,1}=0$ for $n >1$.
\end{conjecture}
\begin{conjecture}
\label{conjecture-cardinality-looms}
Let $H_{n,m} = |\loom{n}{m}|$. Then
\begin{equation*}
\begin{split}
H_{n,m}&=\sum_{k=0}^m \sum_{i=0}^k (-1)^{m-i} \binom{k}{i} (2i+1)^m (2k+1)^n \\
&= \sum_{k=0}^m  (-1)^{m-k} T_{m,k}(2k+1)^n
\end{split}
\end{equation*} 
where $T_{m,k}$ is defined by the recurrence rule $ T_{m,k} = (2k+1) T_{m-1,k} + 2k T_{m-1,k-1}$
with initial conditions 
$ T_{0,0}=1$ and $ T_{0,k}=0$ for any $k \geqslant 1$.
\end{conjecture}
\newcommand{\Av}{\mathsf{Av}}
\begin{conjecture}
Given a set of permutations $S$, denote by $\Av(S)$ the set of all permutations avoiding the patterns of $S$ (see \cite{kitaev} and references therein for more details). Then for any $n,m \geqslant 1$, there is a set of permutations $S_{n,m}$ such that  $\gamma_{n,m}(\loom{n}{m}) = \Av(S_{n,m})$. 
\end{conjecture}
\begin{conjecture}
The center of $\mathfrak{U}_\DY^1$ is generated by $r_1^\id$.
\end{conjecture}
\subsection{Drinfeld--Yetter looms and bumpless pipedreams}
We now give a connection between Drinfeld--Yetter looms and other combinatorial objects called bumpless pipedreams.\index{bumpless pipedream}
\begin{definition}[\cite{lamschubert}, \cite{heck2019duplicitous}]
Let $n \geqslant 1$ and let $\T_{\mathcal{BPD}}$ be the following set of tiles
\begin{equation*}
\begin{tikzpicture}[scale=0.6]
\node at (0,0) {$\T_{\mathcal{BPD}} = \Big\{$};
\draw[black,very thick] (1.5,0.5) rectangle (2.5,-0.5);
\draw[violet,very thick] (2,0.5)--(2,-0.5);
\draw[violet,very thick] (1.5,0)--(2.5,0);
\node at (2.75,-0.5) {,};
\draw[black,very thick] (3,0.5) rectangle (4,-0.5);
\draw[violet,very thick] (4,0)--(3.5,0);
\draw[violet,very thick] (3.5,-0.5)--(3.5,0);
\node at (4.25,-0.5) {,};
\draw[black,very thick] (4.5,0.5) rectangle (5.5,-0.5);
\draw[violet,very thick] (5,0.5)--(5,0);
\draw[violet,very thick] (4.5,0)--(5,0);
\node at (5.75,-0.5) {,};
\draw[black,very thick] (6,0.5) rectangle (7,-0.5);
\draw[violet,very thick] (6,0)--(7,0);
\node at (7.25,-0.5) {,};
\draw[black,very thick] (7.5,0.5) rectangle (8.5,-0.5);
\draw[violet,very thick] (8,0.5)--(8,-0.5);
\node at (8.75,-0.5) {,};
\draw[black,very thick] (9,0.5) rectangle (10,-0.5);
\node at (10.5,-0) {\Big\}.};
\end{tikzpicture}
\end{equation*}
The set of $n \times n$ bumpless pipedreams $\mathcal{BPD}_{n}$ is the set of all possible tilings $B$ of $\grid{n}{n}$ with the elements of $\T_{\mathcal{BPD}}$ such that the following six conditions are satisfied:
\begin{itemize}
\item[(1):] $B_{1,j} \notin \Big\{ \begin{tikzpicture}[scale= 0.4] \draw[black,very thick] (6,0.5) rectangle (7,-0.5);
\draw[violet,very thick] (6,0)--(7,0);

\end{tikzpicture} \ , \ \begin{tikzpicture}[scale= 0.4] \gri{1}{1}

\end{tikzpicture} \Big\}$ for all $j \in \{1, \ldots , n\}$.
\item[(2):] $B_{i,1} \notin \Big\{ \begin{tikzpicture}[scale= 0.4] \draw[black,very thick] (7.5,0.5) rectangle (8.5,-0.5);
\draw[violet,very thick] (8,0.5)--(8,-0.5);

\end{tikzpicture} \ , \ \begin{tikzpicture}[scale= 0.4] \gri{1}{1}

\end{tikzpicture} \Big\}$ for all $i \in \{1, \ldots , n\}$.
\item[(3):] $B_{n,j} \notin \Big\{ \begin{tikzpicture}[scale= 0.4] \gri{1}{1} \bpdcro{1}{1}
\end{tikzpicture} \  , \ \begin{tikzpicture}[scale= 0.4] \gri{1}{1} \bpdtwo{1}{1} 
\end{tikzpicture} \ , \ \begin{tikzpicture}[scale= 0.4] \gri{1}{1} \bpdver{1}{1}
\end{tikzpicture} \Big\}$ for all $j \in \{1, \ldots , n\}$.
\item[(4):] $B_{i,n} \notin \Big\{ \begin{tikzpicture}[scale= 0.4] \gri{1}{1} \bpdcro{1}{1}
\end{tikzpicture} \ ,  \ \begin{tikzpicture}[scale= 0.4] \gri{1}{1} \bpdtwo{1}{1} 
\end{tikzpicture} \ , \ \begin{tikzpicture}[scale= 0.4] \gri{1}{1} \bpdhor{1}{1}
\end{tikzpicture} \Big\}$ for all $i \in \{1, \ldots , n\}$.
\item[(5):] Strings cross pairwise at most once (with respect to the picture obtained by removing all \emph{black borders} from $B$).
\item[(6):] None of the following configurations appear in $B$:
\\ \\
\begin{minipage}{0,16\textwidth}
\begin{tikzpicture}[scale= 0.8]
\gri{1}{2} \bpdcro{1}{1}
\end{tikzpicture}
\end{minipage}%
\begin{minipage}{0,16\textwidth}
\begin{tikzpicture}[scale= 0.8]
\gri{1}{2} \bpdcro{1}{1} \bpdver{1}{2}
\end{tikzpicture}
\end{minipage}%
\begin{minipage}{0,16\textwidth}
\begin{tikzpicture}[scale= 0.8]
\gri{1}{2} \bpdcro{1}{1} \bpdtwo{1}{2}
\end{tikzpicture}
\end{minipage}%
\begin{minipage}{0,16\textwidth}
\begin{tikzpicture}[scale= 0.8]
\gri{1}{2} \bpdcro{1}{2}
\end{tikzpicture}
\end{minipage}%
\begin{minipage}{0,16\textwidth}
\begin{tikzpicture}[scale= 0.8]
\gri{1}{2} \bpdone{1}{1} \bpdcro{1}{2}
\end{tikzpicture}
\end{minipage}%
\begin{minipage}{0,16\textwidth}
\begin{tikzpicture}[scale= 0.8]
\gri{1}{2} \bpdver{1}{1} \bpdcro{1}{2}
\end{tikzpicture}
\end{minipage}
\\ \\
\begin{minipage}{0,16\textwidth}
\begin{tikzpicture}[scale= 0.8]
\gri{1}{2} \bpdone{1}{1} \bpdhor{1}{2}
\end{tikzpicture}
\end{minipage}%
\begin{minipage}{0,16\textwidth}
\begin{tikzpicture}[scale= 0.8]
\gri{1}{2} \bpdone{1}{1} \bpdone{1}{2} 
\end{tikzpicture}
\end{minipage}%
\begin{minipage}{0,16\textwidth}
\begin{tikzpicture}[scale= 0.8]
\gri{1}{2} \bpdver{1}{1} \bpdone{1}{2} 
\end{tikzpicture}
\end{minipage}%
\begin{minipage}{0,16\textwidth}
\begin{tikzpicture}[scale= 0.8]
\gri{1}{2} \bpdone{1}{2} 
\end{tikzpicture}
\end{minipage}%
\begin{minipage}{0,16\textwidth}
\begin{tikzpicture}[scale= 0.8]
\gri{1}{2} \bpdtwo{1}{1}
\end{tikzpicture}
\end{minipage}%
\begin{minipage}{0,16\textwidth}
\begin{tikzpicture}[scale= 0.8]
\gri{1}{2} \bpdtwo{1}{1} \bpdtwo{1}{2}
\end{tikzpicture}
\end{minipage}
\\ \\
\begin{minipage}{0,16\textwidth}
\begin{tikzpicture}[scale= 0.8]
\gri{1}{2} \bpdtwo{1}{1} \bpdver{1}{2}
\end{tikzpicture}
\end{minipage}%
\begin{minipage}{0,16\textwidth}
\begin{tikzpicture}[scale= 0.8]
\gri{1}{2} \bpdtwo{1}{2} \bpdhor{1}{1}
\end{tikzpicture}
\end{minipage}%
\begin{minipage}{0,16\textwidth}
\begin{tikzpicture}[scale= 0.8]
\gri{1}{2} \bpdhor{1}{1} \bpdver{1}{2}
\end{tikzpicture}
\end{minipage}%
\begin{minipage}{0,16\textwidth}
\begin{tikzpicture}[scale= 0.8]
\gri{1}{2} \bpdver{1}{1} \bpdhor{1}{2}
\end{tikzpicture}
\end{minipage}%
\begin{minipage}{0,16\textwidth}
\begin{tikzpicture}[scale= 0.8]
\gri{1}{2} \bpdhor{1}{2}
\end{tikzpicture}
\end{minipage}%
\begin{minipage}{0,16\textwidth}
\begin{tikzpicture}[scale= 0.8]
\gri{1}{2} \bpdhor{1}{1}
\end{tikzpicture}
\end{minipage}
\\ \\
\begin{minipage}{0,11\textwidth}
\begin{tikzpicture}[scale= 0.8]
\gri{2}{1} \bpdcro{1}{1} \bpdhor{2}{1}
\end{tikzpicture}
\end{minipage}%
\begin{minipage}{0,11\textwidth}
\begin{tikzpicture}[scale= 0.8]
\gri{2}{1} \bpdcro{1}{1}
\end{tikzpicture}
\end{minipage}%
\begin{minipage}{0,11\textwidth}
\begin{tikzpicture}[scale= 0.8]
\gri{2}{1} \bpdcro{1}{1} \bpdtwo{2}{1}
\end{tikzpicture}
\end{minipage}%
\begin{minipage}{0,11\textwidth}
\begin{tikzpicture}[scale= 0.8]
\gri{2}{1} \bpdcro{2}{1}
\end{tikzpicture}
\end{minipage}%
\begin{minipage}{0,11\textwidth}
\begin{tikzpicture}[scale= 0.8]
\gri{2}{1} \bpdone{1}{1} \bpdcro{2}{1}
\end{tikzpicture}
\end{minipage}%
\begin{minipage}{0,11\textwidth}
\begin{tikzpicture}[scale= 0.8]
\gri{2}{1} \bpdhor{1}{1} \bpdcro{2}{1}
\end{tikzpicture}
\end{minipage}%
\begin{minipage}{0,11\textwidth}
\begin{tikzpicture}[scale= 0.8]
\gri{2}{1} \bpdone{1}{1} \bpdver{2}{1}
\end{tikzpicture}
\end{minipage}%
\begin{minipage}{0,11\textwidth}
\begin{tikzpicture}[scale= 0.8]
\gri{2}{1} \bpdone{1}{1} \bpdone{2}{1}
\end{tikzpicture}
\end{minipage}%
\begin{minipage}{0,11\textwidth}
\begin{tikzpicture}[scale= 0.8]
\gri{2}{1}  \bpdtwo{1}{1} \bpdhor{2}{1}
\end{tikzpicture}
\end{minipage}
\\ \\
\begin{minipage}{0,11\textwidth}
\begin{tikzpicture}[scale= 0.8]
\gri{2}{1} \bpdone{2}{1}
\end{tikzpicture}
\end{minipage}%
\begin{minipage}{0,11\textwidth}
\begin{tikzpicture}[scale= 0.8]
\gri{2}{1} \bpdtwo{1}{1}
\end{tikzpicture}
\end{minipage}%
\begin{minipage}{0,11\textwidth}
\begin{tikzpicture}[scale= 0.8]
\gri{2}{1} \bpdtwo{1}{1} \bpdtwo{2}{1}
\end{tikzpicture}
\end{minipage}%
\begin{minipage}{0,11\textwidth}
\begin{tikzpicture}[scale= 0.8]
\gri{2}{1} \bpdone{2}{1} \bpdhor{1}{1}
\end{tikzpicture}
\end{minipage}%
\begin{minipage}{0,11\textwidth}
\begin{tikzpicture}[scale= 0.8]
\gri{2}{1} \bpdver{1}{1} \bpdtwo{2}{1}
\end{tikzpicture}
\end{minipage}%
\begin{minipage}{0,11\textwidth}
\begin{tikzpicture}[scale= 0.8]
\gri{2}{1} \bpdver{1}{1} \bpdhor{2}{1}
\end{tikzpicture}
\end{minipage}%
\begin{minipage}{0,11\textwidth}
\begin{tikzpicture}[scale= 0.8]
\gri{2}{1} \bpdver{1}{1}
\end{tikzpicture}
\end{minipage}%
\begin{minipage}{0,11\textwidth}
\begin{tikzpicture}[scale= 0.8]
\gri{2}{1} \bpdver{2}{1}
\end{tikzpicture}
\end{minipage}%
\begin{minipage}{0,11\textwidth}
\begin{tikzpicture}[scale= 0.8]
\gri{2}{1} \bpdhor{1}{1} \bpdver{2}{1}
\end{tikzpicture}
\end{minipage} \\
\end{itemize}
\end{definition}
Note this presentation of bumpless pipedreams differs from the one in \cite{heck2019duplicitous} in that the direction of the strings goes from the left edge to the top one, and not from the right to the bottom. Note also that there are some similarities between bumpless pipedreams and Drinfeld--Yetter mosaics. \\
Pipedreams are relevant in may areas of combinatorics, such as calculation of Schubert polynomials (see \cite{millersturmfels}), permutation words (see \cite{marcott}), and maximal $0$-$1$--fillings of moon polynomials (see \cite{rubey}). \\
There is a canonical map $s_n:\mathcal{BPD}_{n} \to \mathfrak{S}_n$ associating to any $n \times n$ bumpless pipedream a permutation of the symmetric group $\mathfrak{S}_n$; conversely, there is a canonical map $R : \mathfrak{S}_n \to \mathcal{BPD}_n$ associating to any permutation $\sigma$ a specific bumpless pipedream, called the Rothe bumpless pipedream of $\sigma$, see \cite[\S 2]{heck2019duplicitous} for more details.  \\
We now define a map $f: \loom{n}{m} \to \mathcal{BPD}_{n+m}$ with the property that $s_{n+m} (f (L)) = \gamma_{n,m}(L)$. For any $L \in \loom{n}{m}$, consider the element $f(L) \in \mathcal{BPD}_{n+m}$ constructed according to the following steps:
\begin{itemize}
\item[(1)] \emph{Stretch} $L$ horizontally and vertically in such a way every left (resp. top) tile of the first column (resp. row) has only one string. The result of this process gives a tiling of $\grid{\Lambda}{\Omega}$ (where $\Lambda = \sum_{i=1}^n l_{i1}$ and $\Omega = \sum_{j=1}^m t_{1j}$) with elements of $\T_{\mathcal{BPD}}$. However, this in general will not be a bumpless pipedream.
\item[(2)] If $\Lambda \neq n+m$, add a row for any string occurring in the bottom edge of the last row. Then, for any tile of the $\Lambda$--th row having a string in the bottom edge, attach the tile \begin{tikzpicture}[scale= 0.4] \gri{1}{1} \bpdone{1}{1}
\end{tikzpicture} to its bottom edge and fill all the left--side tiles of such a row with the tile \begin{tikzpicture}[scale= 0.4] \gri{1}{1} \bpdhor{1}{1}
\end{tikzpicture}.
\item[(3)] If $\Omega \neq n+m$, add a column for any string occurring in the left edge of the last column. Then, for any tile of the $\Omega$--th column having a string in the left edge, attach the tile \begin{tikzpicture}[scale= 0.4] \gri{1}{1} \bpdone{1}{1}
\end{tikzpicture} to its left edge and fill all the top--side tiles of such a column with the tile \begin{tikzpicture}[scale= 0.4] \gri{1}{1} \bpdver{1}{1}
\end{tikzpicture}.
\end{itemize} 
\begin{example}
Consider the Drinfeld--Yetter loom $L \in \loom{2}{2}$ of Example \ref{example-permutation-ass-to-a-loom}. Then the procedure described above gives
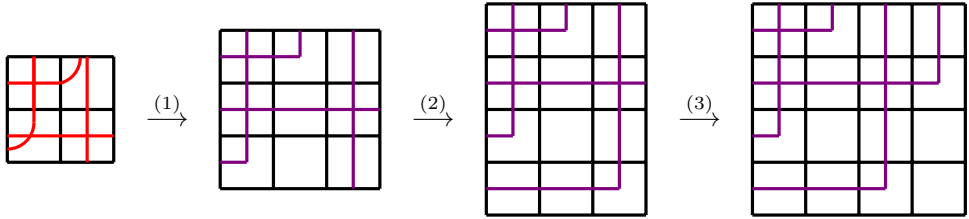
\begin{center}
\begin{tikzpicture}[scale=0.7]
\gri{2}{2} \cro{1}{1} \cro{2}{2} \ver{1}{2} \hor{2}{1}
\draw[red, very thick] (2-1,-1+0.5) arc (-75:0:0.5); 
\draw[red, very thick] (1-1,-2+0.25) arc (-90:0:0.5); 
\draw[red, very thick] (0.5,0)--(0.5,-1.25);
\node at (3,-1) {$\overset{(1)}{\longrightarrow}$};
\begin{scope}[shift={(4,0.5)}]
\gri{3}{3} \bpdcro{1}{1} \bpdone{1}{2} \bpdver{1}{3} \bpdcro{2}{1} \bpdhor{2}{2} \bpdcro{2}{3} \bpdone{3}{1} \bpdver{3}{3}
\end{scope}
\node at (8,-1) {$\overset{(2)}{\longrightarrow}$};
\begin{scope}[shift={(9,1)}]
\gri{4}{3} \bpdcro{1}{1} \bpdone{1}{2} \bpdver{1}{3} \bpdcro{2}{1} \bpdhor{2}{2} \bpdcro{2}{3} \bpdone{3}{1} \bpdver{3}{3} \bpdhor{4}{1}
\bpdhor{4}{2} \bpdone{4}{3}
\end{scope}
\node at (13,-1) {$\overset{(3)}{\longrightarrow}$};
\begin{scope}[shift={(14,1)}]
\gri{4}{4} \bpdcro{1}{1} \bpdone{1}{2} \bpdver{1}{3} \bpdcro{2}{1} \bpdhor{2}{2} \bpdcro{2}{3} \bpdone{3}{1} \bpdver{3}{3} \bpdhor{4}{1}
\bpdhor{4}{2} \bpdone{4}{3} \bpdver{1}{4} \bpdone{2}{4}
\end{scope}
\end{tikzpicture}
\end{center}
and it is easy to see that $\gamma_{2,2}(L) = s_{4}(f(L)) = (1243) \in \mathfrak{S}_4$.
\begin{remark}
The reasoning above shows that the set of $\loom{n}{m}$ of $n \times m$ Drinfeld--Yetter looms maps into the set $\mathcal{BPD}_{n+m}$  of $(n + m) \times (n + m)$ bumpless pipedreams, and so in the symmetric group $\mathfrak{S}_{n+m}$. However, this association is not injective nor surjective. Hence $\loom{n}{m}$ gives a refinement only of the subset of $\mathcal{BPD}_{n+m}$ given by the Rothe bumpless pipedreams.
\end{remark}
\end{example}

\end{document}